\documentclass[final,1p,times,fleqn]{elsarticle}

\usepackage[utf8]{inputenc}
\usepackage{amsmath,amssymb,amsfonts,mathtools,bm}
\usepackage{amsthm}
\usepackage{booktabs}
\usepackage{multirow}
\usepackage{float}
\usepackage{graphicx}
\usepackage{subcaption}
\usepackage{enumitem}
\usepackage{comment}
\usepackage{setspace}
\usepackage[colorlinks=true,linkcolor=black,citecolor=blue,urlcolor=blue]{hyperref}
\usepackage{microtype}

\hypersetup{
  pdftitle={A High-Order Geometric-Compatible DGSEM with Subcell Limiting for the GPR System},
  pdfauthor={Jielin Yang}
}
\journal{Journal of Computational Physics}
\biboptions{numbers,sort&compress}
\newcommand{\A}{\mathbf A}
\newcommand{\vct}{\mathbf v}
\newcommand{\m}{\mathbf m}
\newcommand{\w}{\mathbf w}
\newcommand{\alphaA}{\bm{\alpha}}
\newcommand{\Jv}{\mathbf J}
\newcommand{\betaJ}{\bm{\beta}_J}
\newcommand{\dd}{\,\mathrm d}
\newcommand{\jump}[1]{\left[\!\left[#1\right]\!\right]}

\newcommand{\tr}{\operatorname{tr}}
\newcommand{\dev}{\operatorname{dev}}
\newcommand{\cof}{\operatorname{cof}}
\newcommand{\tcross}{\mathbin{\times_{\!T}}}
\newcommand{\blanksubfigure}[2][0.31\textwidth]{%
  \begin{subfigure}[t]{#1}
    \centering
    \fbox{\rule{0pt}{0.14\textheight}\rule{0.92\linewidth}{0pt}}
    \caption{#2}
  \end{subfigure}}
\newcommand{\plotsubfigure}[3][0.31\textwidth]{%
  \begin{subfigure}[t]{#1}
    \centering
    \includegraphics[width=\linewidth]{#2}
    \caption{#3}
  \end{subfigure}}

\newtheorem{theorem}{Theorem}[section]
\newtheorem{lemma}[theorem]{Lemma}
\newtheoremstyle{boldremark}
  {6pt}{6pt}{\normalfont}{}{\bfseries}{.}{0.5em}
  {\thmname{#1}\thmnumber{ #2}\thmnote{ \textbf{(#3)}}}
\theoremstyle{boldremark}
\newtheorem{remark}{Remark}[section]

\begin{document}
\begin{frontmatter}

\title{A High-Order Geometric-Compatible DGSEM with Subcell Limiting
for the GPR System}

\author[nd]{Jielin Yang\corref{cor1}}
\cortext[cor1]{Corresponding author.}
\ead{jyang38@nd.edu}
\affiliation[nd]{organization={Department of Applied and Computational
Mathematics and Statistics (ACMS), University of Notre Dame},
city={Notre Dame}, postcode={46556}, state={Indiana}, country={USA}}

\begin{abstract}
We construct an equation-of-state-independent geometric-compatible
finite-volume (GC-FV) method and its high-order geometric-compatible nodal
discontinuous Galerkin spectral-element extension (GC-DG) for a reduced
Godunov--Peshkov--Romenski system on fixed Eulerian meshes.  Density is recovered from
$\rho=\rho_0\det\A$, so mass conservation is a hidden companion law of the
distortion equation.  Starting from a straight-path
Dal Maso--LeFloch--Murat (DLM) fluctuation pair,
minimum-change cofactor projections produce a single-valued density flux
and exact induced density Lax--Friedrichs (LF) dissipation without jump quotients.
The summation-by-parts property lifts this two-state identity to high order
and yields a subcell representation of GC-DG.  We then blend GC-DG
with the compatible GC-FV operator on the same Legendre--Gauss--Lobatto
(LGL) subcells.  Because both
operators induce the same boundary density flux, the global subcell blend
retains the exact semidiscrete mass balance while suppressing shock-generated
oscillations.  A complementary modal/face-jump indicator provides the
elementwise blending coefficient.  The method accommodates fluid, elastic,
viscous, and elastoplastic closures, with or without thermal impulse.
\end{abstract}

\begin{keyword}
GPR equations \sep geometric compatibility \sep discontinuous Galerkin
spectral element method \sep summation-by-parts \sep subcell limiter
\sep shock indicator
\end{keyword}

\end{frontmatter}

\section{Introduction}
\label{sec:introduction}

The Godunov--Peshkov--Romenski (GPR) equations provide a first-order
hyperbolic description of continuum mechanics in which distortion,
thermal impulse, and relaxation connect inviscid and viscous fluids with
elastic, viscoelastic, and plastic solids
\cite{PeshkovRomenski2016,Dumbser2016,Peshkov2018,Jackson2017,
JacksonNikiforakis2019,PeshkovEtAl2019}.  Their applications include
compressible flow, elastic waves, and dynamic rupture
\cite{Dumbser2016,Tavelli2020,Gabriel2021}.  The unified model is
numerically demanding: it combines nonconservative products, shocks,
material-dependent stiff sources, and constraints among its state
variables.

High-order GPR computations have been advanced particularly by
path-conservative ADER finite-volume and DG methods, often coupled to
WENO reconstruction, adaptive meshes, and a posteriori subcell limiting
\cite{Dumbser2016,Jackson2017,Tavelli2020}.  Structure preservation is
an additional challenge, richer than for Euler and not exhausted by the
divergence constraint familiar from MHD.  Thermodynamic compatibility is
one such objective: the discrete primary equations should imply the
appropriate thermodynamic companion balance.  Compatible FV and
high-order DG formulations have been developed through entropy-primary
updates, path-integral fluxes, and residual corrections
\cite{Busto2022,BustoDumbser2022DG,Abgrall2023}.  Other structural
objectives include the determinant--density relation
\cite{Boscheri2024,Boscheri2026},
curl-type involutions
\cite{DumbserChiocchettiPeshkov2020,Abgrall2025,Dumbser2026},
and physical admissibility, including positivity checks near shocks
\cite{Tavelli2020}; a posteriori subcell limiting is one common
stabilization mechanism \cite{DumbserSubcell2014}.  Satisfying one does not
automatically secure the others.

Among these structural objectives, we are particularly interested in
geometric-compatible discretizations of the GPR equations.  We call a scheme
\emph{geometric-compatible} if its discrete
distortion equation induces the conservative determinant law and
thereby maintains $\rho=\rho_0\det\A$ semidiscretely, when the relation
holds initially, without an independently evolved density.  This nonlinear
algebraic constraint can be treated more directly in a Lagrangian
description, where material motion, cell-volume evolution, and deformation
kinematics are tied to the same moving mesh.  Compatible cell-centered
Lagrangian discretizations have long exploited this connection in
hydrodynamics and nonlinear elasticity
\cite{Bauer2006,Maire2009,KluthDespres2010,BoscheriLagrangian2022}.
Most recently, Boscheri et al.~\cite{Boscheri2026} extended this viewpoint
to a structure-preserving GPR scheme on moving unstructured meshes.

On fixed Eulerian meshes, by contrast, exact preservation of the GPR
determinant--density companion law has received very little attention; to
the best of our knowledge, the only construction known to us that
explicitly targets this property is the first-order FV method of Boscheri
et al.~\cite{Boscheri2024}.  This reduced fixed-grid strategy removes the
independent mass equation, recovers density from $\det\A$, and
corrects the distortion update so that it induces a conservative density
equation.  The nonlinear determinant constraint is thereby recast as a
hidden discrete conservation law.  That construction also addresses the
thermodynamic companion balance, but its geometric component is the point
of departure here.  It is first order in space and uses jump-quotient
corrections together with Rusanov-type dissipation interpreted through a
vanishing-viscosity regularization, which obscures a direct high-order
extension.

The central idea in the construction of Boscheri et
al.~\cite{Boscheri2024} is not to evolve the density equation independently,
but to regard it as a companion law induced by the remaining reduced
equations and to preserve this implication at the discrete level.  This
strategy closely parallels the earlier entropy-conservative construction of
Tadmor, in which the entropy equation is likewise treated as a companion law
of the primary conservation laws.  Early work of
Osher~\cite{Osher1984} characterized entropy-admissible numerical
fluxes, and Osher and Tadmor~\cite{OsherTadmor1988} analyzed the
interplay of numerical viscosity, accuracy, and entropy conditions.
Tadmor's construction~\cite{Tadmor1987} makes the nondissipative part
explicit.  In one dimension, for a conserved state $u$, entropy pair
$(\eta,q)$, and physical flux $f$, define the entropy variables
$V=\partial_u\eta$ and potential $\psi=V\cdot f-q$.
A consistent two-point flux $F^{EC}$ is selected to satisfy
$(V_R-V_L)\cdot F^{EC}=\psi_R-\psi_L$.
This finite-jump identity turns the desired discrete entropy balance
into an algebraic equation for the flux; additional symmetry or
constitutive ansatzes can select an explicit solution
\cite{Chandrashekar2013}.  Entropy-stable interface dissipation is
added after the equality is established.
In both constructions, the discrete primary equations reproduce a hidden
companion conservation law just as at the continuous level: entropy
conservation in Tadmor's setting and mass conservation in ours.

Extending Tadmor's two-point construction to high-order entropy-stable
DG methods has been an important development, and its mechanism is
equally central here.  In the nodal discontinuous Galerkin spectral element
method (DGSEM), Gauss--Lobatto collocation gives
a diagonal quadrature norm and a differentiation matrix whose
summation-by-parts (SBP) identity discretely reproduces integration
by parts.  Writing the volume operator as a sum of symmetric
two-point fluxes lets their pair identities cancel over interior
node pairs; the remaining terms are boundary contributions closed by
the interface flux.  Thus the companion-law proof does not require
exact integration of nonlinear volume fluxes
\cite{FisherCarpenter2013,CarpenterFisherNielsenFrankel2014,
GassnerWintersKopriva2016,ChenShu2017,Chan2018}.
Castro et al.~\cite{Castro2013} provide the entropy framework for
nonconservative systems, while Renac~\cite{Renac2019} develops its
DGSEM/SBP implementation.  In the present setting, the same SBP
telescoping mechanism lifts the geometric-compatible FV pair to a
high-order DGSEM whose distortion equation induces a high-order discrete
density equation.

Nonlinear hyperbolic systems may develop discontinuous solutions, so
stabilization is needed to suppress numerical oscillations.  Classical
choices include TVD/TVB limiters
\cite{CockburnShu2001,Krivodonova2004} and oscillation-eliminating (OE)
limiters \cite{PengSunWu2025,YangFuOEESDG}.  When applied directly to the
distortion polynomial, however, most such limiters can destroy geometric
compatibility because the determinant is nonlinear.

Subcell stabilization offers a more natural route.  A posteriori methods
replace a troubled DG polynomial by a robust FV evolution on a subgrid
\cite{DumbserSubcell2014}, whereas the hybrid construction of Hennemann et
al.~\cite{Hennemann2021} rewrites an SBP operator in an equivalent
finite-volume form and blends it with a compatible low-order subcell
scheme using one coefficient per element.  Subsequent developments allow
subcell-local blending, invariant-domain or entropy constraints, and more
localized dissipation
\cite{Pazner2021,RuedaRamirezPaznerGassner2022,
RuedaRamirezBolmKuzminGassner2024,LinChan2024,Vilar2025,Zhao2026}.
Rueda-Ram\'irez and Gassner~\cite{RuedaRamirezGassner2024} further showed
that nonconservative SBP discretizations admit such a flux-difference
representation when the nonconservative product separates into local and
symmetric two-point factors.  This framework fits the present setting
naturally: the geometric-compatible finite-volume (GC-FV) and
geometric-compatible discontinuous Galerkin (GC-DG) operators satisfy the
same hidden density law, which is
therefore inherited by their convex residual blend.

A central practical question is how to determine the blending coefficient.
The Persson--Peraire sensor measures the energy in the highest polynomial
modes \cite{PerssonPeraire2006}, and Hennemann et
al.~\cite{Hennemann2021} modified this modal indicator to drive one
coefficient per element.  Later subcell methods have considered
feature-based or constraint-driven coefficients, derivative-based
smoothness detectors, and two-level strategies that combine an element
modal sensor with oscillation measures at internal subcell interfaces
\cite{Pazner2021,RuedaRamirezPaznerGassner2022,LinChan2024,
Vilar2025,Zhao2026}.  Nevertheless, a simple elementwise indicator that
simultaneously detects unresolved modes inside an element and jumps across
element faces remains less developed.  These two signals provide
complementary information, and we therefore construct the global blending
coefficient from the union of an element-modal channel and a
scale-normalized face-jump channel.

The principal contributions are as follows.  First, we retain the central
fixed-grid idea of Boscheri et al.~\cite{Boscheri2024}---namely, treating the
density equation as a companion law induced by the reduced GPR system---but
develop a distinct construction inspired by Tadmor's entropy-conservative
methodology.  A two-state companion identity and a bounded projection define
the low-order \textbf{geometric-compatible finite-volume (GC-FV)} scheme,
which is then extended through the SBP property of nodal DGSEM to the
high-order \textbf{geometric-compatible discontinuous Galerkin (GC-DG)}
scheme.  Second, we rewrite GC-DG in an equivalent subcell form and use this
representation to construct a corresponding geometric-compatible global
subcell limiter that blends GC-DG with GC-FV while retaining the hidden
density law.  Third, we propose a new indicator for this limiter by combining
complementary element-modal and face-jump information.  The resulting method
is assessed for smooth convergence, general constitutive laws, nonlinear-elastic
and thermal-wave Riemann problems, and genuinely
multidimensional configurations.

The remainder of this paper is organized as follows.
Section~\ref{sec:model} introduces the complete and reduced GPR systems and
specifies the constitutive assumptions used in this work.
Section~\ref{sec:fv-compatible} constructs the low-order GC-FV scheme and
proves its determinant companion law.
Section~\ref{sec:high-order} develops the high-order GC-DG scheme through
the SBP construction and derives its equivalent subcell form.
Section~\ref{sec:gpr-subcell} introduces the geometric-compatible global
subcell limiter and the complementary modal/face-jump indicator.
Section~\ref{sec:time} presents the time discretization and quantifies its
fully discrete mass drift.
Numerical results are reported in Section~\ref{sec:numerics}.
Finally, Section~\ref{sec:conclusions} summarizes the conclusions and
outlines directions for future work.  The acoustic-tensor derivation and
supporting algebraic proofs are collected in the appendices.

\section{The GPR model and constitutive scope}
\label{sec:model}

\subsection{Complete energy-primary GPR system}

The starting point of this paper is the complete GPR system written with
total energy as a primary variable.  In three space dimensions its state is
\begin{equation}
 Q_E=(\rho,\m,\mathcal E,\A,\Jv)^T,
 \label{eq:full-GPR-energy-state}
\end{equation}
where $\rho>0$ is mass density, $\m=\rho\vct$ is momentum density,
$\mathcal E$ is total-energy density, $\A=(A_{ik})$ is the distortion, and
$\Jv=(J_k)$ is the thermal impulse.  The constitutive closure is
inherited from the usual entropy-primary GPR potential
$\mathcal E=\mathcal E(\rho,\m,S,\A,\Jv)$, with $S=\rho s$ the
physical entropy density.  For matrices of the same size, the colon denotes
the Frobenius product
$X:Y=\sum_{i,j}X_{ij}Y_{ij}=\operatorname{tr}(X^TY)$, with
$\|X\|_F^2=X:X$.  Its conjugate variables are defined by
\begin{equation}
 \dd\mathcal E
 =r\,\dd\rho+\vct\cdot\dd\m+T\,\dd S
  +\alphaA:\dd\A+\betaJ\cdot\dd\Jv,
 \qquad
 (r,\vct,T,\alphaA,\betaJ)
 =\partial_{(\rho,\m,S,\A,\Jv)}\mathcal E.
 \label{eq:full-GPR-Gibbs}
\end{equation}
Following standard GPR terminology
\cite{PeshkovRomenski2016,Dumbser2016,Boscheri2024}, the tuple in
\eqref{eq:full-GPR-Gibbs} consists of the \emph{thermodynamic dual
variables} (also called the main-field or Godunov variables); as energy
gradients, they are the thermodynamic forces associated with the state
variables.  More precisely, $r$ is the mass-density dual variable,
$\vct$ is the momentum dual variable and hence the material velocity,
$T>0$ is the entropy dual variable and hence the temperature,
$\alphaA=\partial_{\A}\mathcal E$ is the distortion dual variable (the
elastic or strain thermodynamic force), and
$\betaJ=\partial_{\Jv}\mathcal E$ is the thermal-impulse dual
variable (the thermal thermodynamic force).  Repeated indices below are
summed over $i,k,m\in\{1,2,3\}$, $x_k$ is the $k$-th Cartesian coordinate,
and $\delta_{ik}$ is the Kronecker symbol.  Define the shear and thermal
stress tensors by
\begin{equation}
 \sigma_{ik}=A_{ji}\alphaA_{jk},
 \qquad
 \phi_{ik}=J_i\beta_{J,k}.
 \label{eq:full-GPR-stresses}
\end{equation}
We use the GPR flux-stress convention of
\eqref{eq:full-GPR-stresses}: because the momentum flux contains
$pI_3+\bm\sigma+\bm\phi$, the corresponding Cauchy stress tensor is
$-(pI_3+\bm\sigma+\bm\phi)$.  In particular, $\bm\sigma$ has the opposite
sign from the deviatoric Cauchy-stress contribution.
Without optional artificial second-order regularization, the complete
energy-primary first-order GPR equations are
\cite{Dumbser2016,Boscheri2024}
\begin{subequations}
\label{eq:full-GPR-energy-system}
\begin{align}
 \partial_t\rho+\partial_{x_k}(\rho v_k)&=0,
 \label{eq:full-GPR-mass}\\
 \partial_t(\rho v_i)+\partial_{x_k}
 \bigl(\rho v_iv_k+p\delta_{ik}+\sigma_{ik}+\phi_{ik}\bigr)&=0,
 \label{eq:full-GPR-momentum}\\
 \partial_t\mathcal E
 +\partial_{x_k}\!\left[
 \mathcal E v_k+v_i(p\delta_{ik}+\sigma_{ik}+\phi_{ik})
 +T\beta_{J,k}\right]&=0,
 \label{eq:full-GPR-energy}\\
 \partial_tA_{ik}+\partial_{x_k}(A_{im}v_m)
 +v_m(\partial_{x_m}A_{ik}-\partial_{x_k}A_{im})
 &=-\frac{\alphaA_{ik}}{\theta_1},
 \label{eq:full-GPR-A}\\
 \partial_tJ_k+\partial_{x_k}(J_mv_m+T)
 +v_m(\partial_{x_m}J_k-\partial_{x_k}J_m)
 &=-\frac{\beta_{J,k}}{\theta_2}.
 \label{eq:full-GPR-J}
\end{align}
\end{subequations}
Here $p$ is the constitutive pressure and $\theta_1,\theta_2>0$ are the
strain- and thermal-relaxation coefficients.

\subsection{Constitutive assumptions}

Set
\begin{equation}
 g(\A)=\det\A,\qquad \w=\partial_{\A}g=\cof\A,
 \qquad S_A=-\frac{\alphaA}{\theta_1},\qquad
 \w:S_A=0.
 \label{eq:isochoric-source}
\end{equation}
Here $\cof\A$ denotes the cofactor matrix of $\A$.
The last identity is the only constitutive condition required by the
geometric construction.  It states that distortion relaxation is
isochoric and therefore does not enter the determinant companion law.
The pressure, temperature, stresses, and relaxation times may otherwise
come from any smooth admissible GPR closure.

The standard GPR closure separates the specific total energy into
microscopic, mesoscopic, and macroscopic contributions
\cite{PeshkovRomenski2016,Dumbser2016,PeshkovEtAl2019}:
\begin{equation}
 \frac{\mathcal E}{\rho}
 =e_h(\rho,s)+e_s(\rho,\A)+e_J(\rho,\Jv)
 +\frac12|\vct|^2.
 \label{eq:GPR-energy-decomposition}
\end{equation}
The hydrodynamic or volumetric energy $e_h$ supplies the caloric equation
of state (EOS),
whereas $e_s$ and $e_J$ describe, respectively, mesoscopic shear
deformation and the thermal impulse.  Thus these three contributions are
separate parts of the energy rather than additional variables required to
define the hydrodynamic EOS.  With
\begin{align}
 G&=\A^T\A,\quad
 D=\mathring G=\dev G=G-\tfrac13\tr(G)I_3,
 \quad \Psi=D:D,
 \label{eq:GPR-shear-invariants}\\
 e_s(\rho,\A)&=\tfrac14 b^2(\rho)\Psi,
 &e_J(\rho,\Jv)&=\tfrac12c_h^2|\Jv|^2,
 \label{eq:standard-mesoscopic-energy}\\
 \alphaA&=\rho b^2(\rho)\A D,
 &\betaJ&=\rho c_h^2\Jv.
 \label{eq:J-thermal-closure}
\end{align}
Here $b(\rho)\ge0$ and $c_h\ge0$ are the shear- and heat-wave speeds.  The
constant-shear closure used in the ideal/stiffened-gas examples is
$b(\rho)=c_s$.  In
particular, the corresponding shear-energy density is
\begin{equation}
 \mathcal E_s=\rho e_s
 =\tfrac14\rho c_s^2\mathring G:\mathring G.
 \label{eq:J-shear}
\end{equation}
The pressure and temperature follow from the same potential,
\begin{equation}
 p=\rho^2\partial_\rho(e_h+e_s+e_J)\big|_{s,\A,\Jv},
 \qquad
 T=\partial_s(e_h+e_s+e_J)\big|_{\rho,\A,\Jv}.
 \label{eq:GPR-pressure-temperature}
\end{equation}
Consequently, density-dependent shear or heat-wave speeds contribute to
the pressure; for constant $c_s$ and $c_h$, the pressure is supplied by
$e_h$ alone.

The hydrodynamic energies and EOS used in the numerical examples are as
follows.

\paragraph{Ideal and stiffened gases}
For a stiffened gas,
\begin{equation}
 e_h(\rho,T)=c_vT+\frac{p_\infty}{\rho},
 \qquad
 p+p_\infty=(\gamma-1)\rho c_vT,
 \label{eq:stiffened-gas-EOS}
\end{equation}
or, equivalently,
\begin{equation}
 p=(\gamma-1)\rho e_h-\gamma p_\infty.
 \label{eq:stiffened-gas-caloric}
\end{equation}
Here $c_v>0$ is the specific heat at constant volume,
$\gamma=c_p/c_v>1$ is the heat-capacity ratio, and $p_\infty\ge0$ is the
stiffening pressure.  The ideal-gas EOS is recovered by setting
$p_\infty=0$ \cite{Dumbser2016,Jackson2017,PeshkovEtAl2019}.

\paragraph{Barton--Romenski energy}
For the nonlinear elastic solid tests, the Barton--Romenski specific
internal energy used in nonlinear-elastic Riemann solvers and Eulerian
large-deformation calculations
\cite{BartonEtAl2009,Barton2010,Zhang2017,PeshkovEtAl2019}
can be written in the GPR decomposition as
\begin{align}
 e_h^{\rm BR}(\rho,s)
 &=\frac{K_0}{2\alpha_B^2}(\nu^{\alpha_B}-1)^2
 +c_vT_0\nu^{\gamma_B}\bigl(e^{s/c_v}-1\bigr),
 \label{eq:Barton-volumetric-energy}\\
 e_s^{\rm BR}(\rho,\A)
 &=\frac14b_0^2\nu^{\beta_B}\Psi,
 \qquad b^2(\rho)=b_0^2\nu^{\beta_B}.
 \label{eq:Barton-shear-energy}
\end{align}
Here and below, $\nu:=\rho/\rho_0$ is the compression ratio relative to
the reference density $\rho_0$.  Equivalently, this is the isotropic
hyperelastic energy expressed through
the invariants $I_1=\tr G$, $I_2=\tfrac12[(\tr G)^2-\tr(G^2)]$, and
$I_3=\det G=\nu^2$ on the geometric constraint
$\rho=\rho_0\det\A$, using
$I_1^2/3-I_2=\Psi/2$
\cite{BartonEtAl2009,Barton2010,Zhang2017}.  Here
$K_0=c_{L,0}^2-4b_0^2/3$ is the reference bulk-wave-speed squared,
$c_{L,0}$ and $b_0$ are the reference longitudinal and shear-wave speeds,
$T_0$ is the reference temperature, and
$\alpha_B,\beta_B,\gamma_B$ control the nonlinear density dependence of
the cold, shear, and thermal contributions, respectively
\cite{BartonEtAl2009,Barton2010}.  This closure gives
$T=T_0\nu^{\gamma_B}e^{s/c_v}$ and is the constitutive form used in the
five-wave and cylindrical-solid tests below
\cite{Zhang2017,JacksonNikiforakis2019}.

\begin{remark}[EOS scope]
\label{rem:model-EOS-scope}
The geometric-compatible construction below is not restricted to the two
closures displayed above; it only requires a smooth admissible GPR energy
and the isochoric identity \eqref{eq:isochoric-source}.  Other standard GPR
choices include the Mie--Gr\"uneisen hydrodynamic EOS widely used for
shock-loaded metals \cite{PeshkovEtAl2019,JacksonNikiforakis2019} and more
general Godunov--Romenski nonlinear hyperelastic energies
\cite{BartonEtAl2009,Barton2010}, including damage-dependent variants used
in GPR models of brittle fracture \cite{Tavelli2020,Gabriel2021}.  Such a closure can be inserted without
altering the companion projection, although its dissipation coefficient must be
obtained from the corresponding acoustic tensor or full normal Jacobian.
\end{remark}

The associated distortion source used in the computations is
\begin{equation}
 S_A=-\frac{3}{\tau_1}g^{5/3}\A D,
 \qquad g=\det\A>0,
 \label{eq:standard-A-relaxation}
\end{equation}
which satisfies \eqref{eq:isochoric-source}.  With the energy-density
convention in \eqref{eq:GPR-energy-decomposition}--\eqref{eq:J-thermal-closure},
comparison with $S_A=-\alphaA/\theta_1$ in
\eqref{eq:full-GPR-A} gives
\begin{equation}
 \theta_1=\frac{\rho\tau_1b^2(\rho)}{3g^{5/3}}.
 \label{eq:standard-theta1}
\end{equation}
This relation also fixes the density factor that can differ when the GPR
energy is written per unit mass rather than per unit volume.  The limits
$c_s=0$ and $\tau_1=\infty$ are interpreted as $S_A=0$.  When heat conduction is
active, the standard relaxation coefficient is
\begin{equation}
 \theta_2=\rho z_2\tau_2c_h^2,\qquad
 z_2=\frac{\rho_0T_0}{\rho T}.
 \label{eq:J-thermal-relaxation}
\end{equation}
An admissible state has finite constitutive quantities, $\det\A>0$, and
positive density and temperature.  The geometric-compatible construction below does not rely on a
particular caloric EOS or on convexity of a thermodynamic entropy.

\subsection{Fixed planar 2D--3C convention}

All results use a two-dimensional mesh but retain the three-dimensional GPR
constitutive law, as in standard GPR computations
\cite{Dumbser2016,Jackson2017,Boscheri2024}.  We impose
\begin{equation}
 \partial_zQ_E=0,\qquad v_3=0,\qquad J_3=0,\qquad
 \A=\begin{pmatrix}A_\parallel&0\\0&a_3\end{pmatrix}
 \in GL^+(3),
 \qquad a_3=A_{33}>0.
 \label{eq:J-planar-A}
\end{equation}
Here \emph{2D--3C} means two spatial coordinates but three-component
constitutive vectors and \(3\times3\) distortion tensors subject to the
block restriction in \eqref{eq:J-planar-A}.
The retained full-system state is
\begin{equation}
 Q_E^{\rm 2D\text{-}3C}=(\rho,m_1,m_2,\mathcal E,
 A_{11},A_{12},A_{21},A_{22},A_{33},J_1,J_2)^T,
 \label{eq:full-planar-state}
\end{equation}
and all determinants, cofactors, traces, and deviators remain genuinely
three-dimensional:
\begin{equation}
 g=a_3\det A_\parallel,\qquad
 \w=\begin{pmatrix}a_3\cof_2A_\parallel&0\\
 0&\det A_\parallel\end{pmatrix},
 \qquad \w:\A=3g.
 \label{eq:J-planar-cofactor}
\end{equation}
Here $A_\parallel\in\mathbb R^{2\times2}$ is the in-plane distortion block,
and $\cof_2$ denotes the two-dimensional cofactor; explicitly, for
\(C=\left(\begin{smallmatrix}c_{11}&c_{12}\\c_{21}&c_{22}\end{smallmatrix}\right)\),
\(\cof_2C=\left(\begin{smallmatrix}c_{22}&-c_{21}\\-c_{12}&c_{11}\end{smallmatrix}\right)\).
After imposing \eqref{eq:J-planar-A}, the momentum and thermal blocks are
\(\m=(m_1,m_2)^T\) and \(\Jv=(J_1,J_2,0)^T\), and every distortion
increment is stored through the five entries displayed in
\eqref{eq:full-planar-state}; formulas written as \(3\times3\) matrix contractions use
their block-diagonal embeddings.
The four cross-block entries remain zero, but $A_{33}$ is evolved:
\begin{equation}
 \partial_tA_{33}+v_1\partial_xA_{33}
 +v_2\partial_yA_{33}=S_{A,33}.
 \label{eq:A33}
\end{equation}
Freezing $A_{33}$ would change both the determinant and the standard
three-dimensional shear constitutive law, and is not considered.
The conditions $J_3=0$ and $\partial_z=0$ are invariant under
\eqref{eq:full-GPR-J}; hence only $J_1,J_2$ are evolved.  All formulas for
$\bm\phi$ and $\betaJ$ use their three-component embeddings.

\section{First-order geometric-compatible finite-volume discretization}
\label{sec:fv-compatible}

This section constructs the first-order geometric-compatible finite-volume
scheme, denoted GC-FV.  It evolves the reduced state
$U=(Y,\A,\Jv)^T$, with $Y=(\m,\mathcal E)^T$, and recovers density from
$\rho=\rho_0\det\A$.  We first construct nondissipative conservative fluxes
and one-sided fluctuations, and then add a projected local
Lax--Friedrichs (LF) stabilization.  The completed GC-FV method is valid for a
general admissible EOS whenever $\w:S_A=0$.

For an oriented face, let the unit normal $n$ point from $L$ to $R$, and
write $\Delta a=a_R-a_L$, $\overline a=(a_L+a_R)/2$, and
$v_{n,a}=\vct_a\cdot n$.  Superscripts $-$ and $+$ denote the fluctuations
assigned to the left and right states at this fixed orientation.

\subsection{Nondissipative GC-FV construction}
\label{sec:geometry}

In this subsection we construct the nondissipative numerical fluxes and
distortion fluctuations used by GC-FV.  We start from the reduced system of
Boscheri et al., establish the determinant identities needed at a finite
jump, and then construct a two-state companion identity and a
fluctuation pair that realizes a selected density flux.  LF
stabilization and the complete GC-FV update follow in
Sections~\ref{sec:dissipative-FV} and \ref{sec:complete-FV}.

\subsubsection{Boscheri reduction and hidden companion laws}
\label{sec:Boscheri-reduction}

We first remove the independent mass equation so that its replacement,
$\rho=\rho_0\det\A$, becomes a hidden conservation law to enforce
discretely.  This is the fixed-grid reduction introduced by Boscheri et
al.~\cite{Boscheri2024}.  Starting from the energy-primary
GPR system \eqref{eq:full-GPR-energy-system}, discard the
mass equation and reconstruct density algebraically from the distortion:
\begin{equation}
 U=(Y,\A,\Jv)^T,
 \qquad Y=(\m,\mathcal E)^T,
 \qquad \rho=\rho_0g(\A),
 \qquad g(\A)=\det\A,
 \label{eq:model-state}
\end{equation}
where $\rho_0>0$ is a constant reference density.  Under the planar
convention \eqref{eq:J-planar-A}, the reduced state is
\begin{equation}
 U=(m_1,m_2,\mathcal E,
 A_{11},A_{12},A_{21},A_{22},A_{33},J_1,J_2)^T.
 \label{eq:J-state}
\end{equation}
In three space dimensions $U=(Y,\A,\Jv)^T$ retains all three momentum
components, nine distortion components, and three thermal-impulse components.
The reduced balance laws, together with the algebraic density recovery, are
\begin{subequations}
\label{eq:gpr}
\begin{align}
 \partial_t(\rho v_i)+\partial_{x_k}
 \bigl(\rho v_iv_k+p\delta_{ik}+\sigma_{ik}+\phi_{ik}\bigr)&=0,
 \label{eq:gpr-momentum}\\
 \partial_t\mathcal E+\partial_{x_k}
 \bigl((\mathcal E+p)v_k+v_i(\sigma_{ik}+\phi_{ik})
       +T\beta_{J,k}\bigr)&=0,
 \label{eq:gpr-energy}\\
 \partial_tA_{ik}+\partial_{x_k}(A_{im}v_m)
 +v_m(\partial_{x_m}A_{ik}-\partial_{x_k}A_{im})&=S_{A,ik},
 \label{eq:gpr-A}\\
 \partial_tJ_k+\partial_{x_k}(J_mv_m+T)
 +v_m(\partial_{x_m}J_k-\partial_{x_k}J_m)
 &=-\frac{\beta_{J,k}}{\theta_2}.
 \label{eq:gpr-J}
\end{align}
\begin{equation}
 \rho=\rho_0\det\A.
 \label{eq:gpr-density-recovery}
\end{equation}
\end{subequations}
The last line is an algebraic recovery relation, not an independently
evolved balance law.  Here $\bm\sigma=(\sigma_{ik})$ and
$\bm\phi=(\phi_{ik})$ are the shear and
thermal stresses in \eqref{eq:full-GPR-stresses}, $S_A=(S_{A,ik})$ is the isochoric source in
\eqref{eq:isochoric-source}, and $I_3$ denotes the $3\times3$ identity
matrix.  The momentum-flux tensor and energy-flux vector used below are
\begin{equation}
 \mathbf f_m(U)=\rho\vct\otimes\vct+pI_3+\bm\sigma+\bm\phi,
 \qquad
 (\mathbf f_{\mathcal E}(U))_k=(\mathcal E+p)v_k
 +v_i(\sigma_{ik}+\phi_{ik})+T\beta_{J,k}.
 \label{eq:physical-Y-fluxes}
\end{equation}

Contracting \eqref{eq:gpr-A} with the cofactor
$\w=\partial_{\A}g=\cof\A$, and using
$\w:S_A=0$ from \eqref{eq:isochoric-source}, gives
\begin{equation}
 \partial_tg+\nabla\cdot(g\vct)=0,
 \qquad
 \partial_t\rho+\nabla\cdot(\rho\vct)=0.
 \label{eq:continuous-companion}
\end{equation}

The determinant and density equations are therefore hidden companion laws of the reduced system.  Reduction turns the algebraic determinant
constraint into a discrete chain-rule problem: the cofactor contraction of
the distortion scheme must reproduce a conservative density scheme.  The
same pair identities telescope under DGSEM summation by parts (SBP)
in Section~\ref{sec:high-order}.  All claims here are semidiscrete.
Let \(\Omega\subset\mathbb R^2\) denote the computational domain and
\(\partial\Omega\) its physical boundary.

\subsubsection{Necessary algebraic formulas}
\label{sec:G-algebra-tools}

The finite-volume construction below must evaluate the straight-path
Dal Maso--LeFloch--Murat (DLM)
jump of the distortion equation and replace the differential determinant
chain rule by exact two-state identities.  We therefore collect here the
matrix algebra and normal operators needed for those two steps.

For $X,Y,Z\in\mathbb R^{3\times3}$, define the tensor cross product by
\begin{equation}
 (X\tcross Y)_{ij}
 :=\varepsilon_{ipq}\varepsilon_{jrs}X_{pr}Y_{qs},
 \label{eq:tensor-cross-definition}
\end{equation}
where $i,j,p,q,r,s\in\{1,2,3\}$ and $\varepsilon_{ipq}$ is the alternating
symbol.  This product is bilinear and symmetric and satisfies
\begin{equation}
 X\tcross Y=Y\tcross X,
 \qquad
 (X\tcross Y):Z=(X\tcross Z):Y=(Y\tcross Z):X.
 \label{eq:tensor-cross-symmetry}
\end{equation}
In particular,
\begin{equation}
 \cof X=\frac12X\tcross X,
 \qquad
 \det X=\frac16(X\tcross X):X.
 \label{eq:tensor-cross-minors}
\end{equation}
These polynomial identities are standard in polyconvex finite-strain
formulations \cite{Bonet2015,Betsch2018}.  If $X$ is invertible, they are
equivalent to $\cof X=(\det X)X^{-T}$.

For a unit normal $n$, write
\begin{equation}
 f_A(U;n)=(\A\vct)\otimes n,
 \qquad h(U;n)=g\,\vct\cdot n,
 \label{eq:geometry-notation}
\end{equation}
where $f_A$ is the conservative normal distortion flux and $h$ is the
physical normal flux of the determinant law.  For a matrix increment $X$,
the normal nonconservative operator is
\begin{equation}
 B_n(\vct)X=(\vct\cdot n)X-(X\vct)\otimes n.
 \label{eq:B-normal}
\end{equation}
The flux $f_A$ and operator $B_n$ enter the straight-path generalized jump
in \eqref{eq:G-straight-path-jump}; the tensor-cross identities yield the
finite-jump determinant formulas derived next.

\subsubsection{Polynomial two-state determinant algebra}
\label{sec:G-polarization-algebra}

We next establish the exact finite-jump determinant identities needed to
replace a differential chain rule by two-state algebra.  In particular,
they determine the endpoint candidate determinant fluxes in
\eqref{eq:G-baseline-h-candidates}, isolate the finite-jump defect in
\eqref{eq:G-baseline-defect}, and quantify the endpoint linearization
errors used by the projected LF stabilization.  The resulting two-state
companion pair is subsequently lifted by the SBP operator in the GC-DG
construction.
For two distortion states, put
\begin{equation}
 \Delta\A:=\A_R-\A_L,
 \qquad
 \A_c:=\frac12(\A_L+\A_R),
 \qquad
 g_a:=\det\A_a,
 \qquad
 \w_a:=\cof\A_a,
 \quad a\in\{L,R\}.
 \label{eq:G-polarization-data}
\end{equation}
The bilinearity of \eqref{eq:tensor-cross-definition} gives
\begin{equation}
 \Delta\w=\A_c\tcross\Delta\A,
 \qquad
 \frac12(\w_L+\w_R)=\cof\A_c+\frac14\cof(\Delta\A).
 \label{eq:G-cofactor-polarization}
\end{equation}
Define the quadratic and cubic jump polynomials
\begin{equation}
 a_{\Delta A}:=\Delta\w:\Delta\A
 =(\A_c\tcross\Delta\A):\Delta\A,
 \qquad
 d_{\Delta A}:=\det(\Delta\A)
 =\frac16(\Delta\A\tcross\Delta\A):\Delta\A.
 \label{eq:G-jump-polynomials}
\end{equation}
The exact determinant polarization is
\begin{equation}
 \Delta g
 =\cof\A_c:\Delta\A+\frac14d_{\Delta A}.
 \label{eq:G-determinant-polarization}
\end{equation}
In particular $a_{\Delta A}=O(\|\Delta\A\|_F^2)$ and
$d_{\Delta A}=O(\|\Delta\A\|_F^3)$ on compact state sets.  The schemes below use endpoint
cofactors and do not require a determinant jump quotient.

\subsubsection{Companion projection of the baseline fluctuations}

We first construct a straight-path baseline for the nonconservative
distortion equation.  We then derive the two-state identity required by
the hidden determinant law and project the baseline onto a prescribed
density flux.

\paragraph{Straight-path path-conservative baseline}
The nonconservative distortion equation is represented by two
one-sided face fluctuations.  Following the DLM framework
\cite{DalMaso1995} and its path-conservative discretization
\cite{Pares2006}, we choose a path between the face states to define
their generalized jump.
For a system
$U_t+\nabla\cdot F(U)+B(U):\nabla U=0$, a DLM path
$\Phi(s;U_L,U_R)$ joins $U_L$ to $U_R$.  Define the generalized physical
jump associated with this path by
\begin{equation}
 \mathcal J_\Phi(U_L,U_R;n)
 :=F_n(U_R)-F_n(U_L)
 +\int_0^1B_n(\Phi(s))\,\partial_s\Phi(s)\,\dd s.
 \label{eq:G-DLM-fluctuation-sum}
\end{equation}
Here $F_n=F\cdot n$ and $B_n$ is the normal nonconservative operator.
We call a numerical pair \emph{path-sum consistent} with the selected
path when
\begin{equation}
 \mathcal D^-(U_L,U_R;n)+\mathcal D^+(U_L,U_R;n)
 =\mathcal J_\Phi(U_L,U_R;n).
 \label{eq:path-sum-definition}
\end{equation}
This is the defining face-sum relation of a path-conservative
fluctuation pair \cite{Pares2006}; the path fixes the sum but not its
left/right allocation.  It refers to the selected path, not to a
unique physical shock path.
For distortion, $F_n=f_A(U;n)$ and the integral uses
$B_n(\vct(\Phi))\partial_s\A(\Phi)$ from
\eqref{eq:geometry-notation}--\eqref{eq:B-normal}; denote its block by
$\mathcal J_{A,\Phi}:=(\mathcal J_\Phi)_A$.
In the normal direction, the homogeneous distortion operator has the
quasilinear form
\begin{equation}
 v_n\partial_n\A+(\A\,\partial_n\vct)\otimes n.
 \label{eq:G-normal-A-operator}
\end{equation}
We now specify the DLM path that defines the baseline generalized jump.
In the variables entering \eqref{eq:G-normal-A-operator}, choose the
straight path
\begin{equation}
 \Phi_A^{\rm lin}(s)=\A_L+s\Delta\A,\qquad
 \Phi_v^{\rm lin}(s)=\vct_L+s\Delta\vct,\qquad 0\le s\le1,
 \label{eq:G-linear-A-v-path}
\end{equation}
where $\Delta\A=\A_R-\A_L$ and $\Delta\vct=\vct_R-\vct_L$.  Integrating
\eqref{eq:G-normal-A-operator} along \eqref{eq:G-linear-A-v-path} gives the
distortion block of the generalized jump defined in
\eqref{eq:G-DLM-fluctuation-sum}:
\begin{equation}
 \mathcal J_A^{\rm lin}:=\mathcal J_{A,\Phi^{\rm lin}}
 =\overline v_n\Delta\A+(\A_c\Delta\vct)\otimes n,
 \qquad \A_c=\frac12(\A_L+\A_R).
 \label{eq:G-straight-path-jump}
\end{equation}
One exchange-symmetric endpoint allocation of $\mathcal J_A^{\rm lin}$ is
\begin{align}
 \mathcal H_{A,0}^-&=\frac12\overline v_n\Delta\A
 +\frac12(\A_L\Delta\vct)\otimes n,
 \label{eq:G-baseline-minus}\\
 \mathcal H_{A,0}^+&=\frac12\overline v_n\Delta\A
 +\frac12(\A_R\Delta\vct)\otimes n.
 \label{eq:G-baseline-plus}
\end{align}
Their sum is $\mathcal J_A^{\rm lin}$, so this is a
path-conservative baseline for the selected straight path.  The first
term transports the distortion jump; the second represents the
velocity-gradient and nonconservative coupling.  Other allocations
occur in GPR schemes \cite{Dumbser2016,Boscheri2024}; the choice above
will give the compact cubic defect \eqref{eq:G-baseline-defect}.  Path conservation here describes
only the baseline: the subsequent companion-law projection need not
preserve its fluctuation sum.

\begin{remark}[Endpoint versus average allocation]
\label{rem:G-endpoint-versus-half-split}
The standard path-conservative GPR split assigns one half of the
straight-path generalized jump to each endpoint
\cite{Dumbser2016,Boscheri2024}:
\begin{equation}
 \mathcal H_{A,\mathrm{avg}}^-
 =\mathcal H_{A,\mathrm{avg}}^+
 =\frac12\mathcal J_A^{\rm lin}
 =\frac12\overline v_n\Delta\A
  +\frac12(\A_c\Delta\vct)\otimes n.
 \label{eq:G-standard-average-allocation}
\end{equation}
The endpoint pair \eqref{eq:G-baseline-minus}--\eqref{eq:G-baseline-plus}
differs from \eqref{eq:G-standard-average-allocation} by
$\mp\tfrac14(\Delta\A\,\Delta\vct)\otimes n$, so both choices have the same DLM
sum.  We use the endpoint allocation because it makes both the GC-FV
companion projection and the later nonconservative subcell factorization
more direct.  The standard average allocation can also be made
geometric-compatible and rewritten in subcell form, but the resulting
algebra is more involved and is outside the scope of this work.
\end{remark}

The thermal-impulse equation has the normal homogeneous operator
\begin{equation}
 v_n\partial_n\Jv+n\bigl(\Jv\cdot\partial_n\vct+\partial_nT\bigr).
 \label{eq:G-normal-J-operator}
\end{equation}
Along the same straight path in $(\Jv,\vct,T)$, an exchange-symmetric
path-conservative baseline is
\begin{align}
 \mathcal H_{J,0}^-&=\frac12\overline v_n\Delta\Jv
 +\frac12n\bigl(\Jv_L\cdot\Delta\vct+\Delta T\bigr),
 \label{eq:G-J-baseline-minus}\\
 \mathcal H_{J,0}^+&=\frac12\overline v_n\Delta\Jv
 +\frac12n\bigl(\Jv_R\cdot\Delta\vct+\Delta T\bigr).
 \label{eq:G-J-baseline-plus}
\end{align}
Indeed, their sum is
$\overline v_n\Delta\Jv+n(\overline\Jv\cdot\Delta\vct+\Delta T)$,
the straight-path generalized jump of \eqref{eq:G-normal-J-operator}.
This straight-path pair is used without further modification in the thermal-impulse block of GC-FV.

\paragraph{Tadmor-type two-state companion identity}
To induce the hidden determinant (and hence density) equation from the
distortion equation, we now derive a two-state constraint following the
strategy of Tadmor's entropy-conservative flux construction
\cite{Tadmor1987,Chandrashekar2013}.  Its fluctuation-pair form is also the
nonconservative analogue of the entropy-conservative condition studied by
Castro et al.~\cite{Castro2013}.  For a cell $K$, write a general
semidiscrete FV distortion equation as
\begin{equation}
 \frac{\dd\A_K}{\dd t}
 +\frac1{|K|}\sum_{f\subset\partial K}|f|\,
 \mathcal H_A^-(U_K,U_{K_f};n_{Kf})=S_{A,K}.
 \label{eq:G-general-A-FV}
\end{equation}
Here $K_f$ is the neighbor across $f$, $n_{Kf}$ is the outward unit
normal of $K$, and $\mathcal H_A^-$ is the fluctuation assigned to $K$.
At an interior face oriented from $L$ to $R$, the neighbor receives
$\mathcal H_A^+(U_L,U_R;n)
=\mathcal H_A^-(U_R,U_L;-n)$.
Contracting with $\w_K=\cof\A_K$ gives the exact semidiscrete chain
rule $\dd g_K/\dd t=\w_K:\dd\A_K/\dd t$, while $\w:S_A=0$.
The physical determinant flux is
\begin{equation}
 h(U;n)=g(U)\,\vct(U)\cdot n
 \label{eq:G-physical-h}
\end{equation}
by \eqref{eq:continuous-companion}; it is not a numerical choice.
If the contracted equations share one numerical determinant flux
$\widehat h$, their one-sided scalar fluctuations must satisfy
\begin{equation}
 \w_L:\mathcal H_A^-=\widehat h-h_L,
 \qquad
 \w_R:\mathcal H_A^+=h_R-\widehat h,
 \label{eq:G-one-sided-h-reconstruction}
\end{equation}
where $h_a=h(U_a;n)=g_av_{n,a}$.  Eliminating $\widehat h$ gives the
Tadmor-type geometric two-state identity
\begin{equation}
 \boxed{\w_L:\mathcal H_A^-+\w_R:\mathcal H_A^+
 =h_R-h_L.}
 \label{eq:G-pair}
\end{equation}
Conversely, this identity defines a single-valued flux
\begin{equation}
 \widehat h
 =h_L+\w_L:\mathcal H_A^-
 =h_R-\w_R:\mathcal H_A^+.
 \label{eq:G-hhat-reconstruction}
\end{equation}
Thus \eqref{eq:G-pair} is the discrete chain rule for the hidden law,
not the DLM path-sum condition \eqref{eq:path-sum-definition}.

\paragraph{Projection to a prescribed density flux}
We now specify a density flux and determine the nondissipative
distortion fluctuations by projecting the path-conservative baseline
onto the Tadmor-type companion constraint just derived.  The baseline
does not, in general, induce a single-valued determinant flux: its two
endpoint reconstructions
\begin{equation}
 h_{0,L}:=h_L+\w_L:\mathcal H_{A,0}^-,
 \qquad
 h_{0,R}:=h_R-\w_R:\mathcal H_{A,0}^+
 \label{eq:G-baseline-h-candidates}
\end{equation}
differ by the cubic determinant defect
\begin{equation}
 h_{0,L}-h_{0,R}
 =\frac12\overline v_nd_{\Delta A},\qquad
 d_{\Delta A}=\det(\Delta\A).
 \label{eq:G-baseline-defect}
\end{equation}
This is the finite-increment term missed by a linearized chain rule.
The geometric-compatible condition permits any consistent single-valued target
determinant flux.  For the general-EOS method choose
\begin{equation}
 F_\rho^0:=\frac12(\rho_Lv_{n,L}+\rho_Rv_{n,R}),
 \qquad h_G^t:=\frac{F_\rho^0}{\rho_0}.
 \label{eq:G-central-mass-flux}
\end{equation}
This arithmetic target is symmetric, EOS independent, and becomes the
ordinary central--LF density flux after stabilization.

We now impose the target flux as a hard constraint and otherwise change the
straight-path baseline as little as possible:
\begin{equation}
 \boxed{
 (\mathcal H_A^{-,G},\mathcal H_A^{+,G})
 =\arg\min_{\mathcal H^-,\mathcal H^+}
 \frac12\left(
 \|\mathcal H^--\mathcal H_{A,0}^-\|_F^2+
 \|\mathcal H^+-\mathcal H_{A,0}^+\|_F^2\right)}
 \label{eq:G-volume-optimization}
\end{equation}
subject to
\begin{equation}
 \w_L:\mathcal H^-=h_G^t-h_L,\qquad
 \w_R:\mathcal H^+=h_R-h_G^t.
 \label{eq:G-volume-constraints}
\end{equation}
The constraints are exactly the two one-sided forms of
\eqref{eq:G-hhat-reconstruction}; the objective gives the minimum
Frobenius modification of the selected DLM baseline.  Define
\begin{equation}
 P_a:=\frac{\w_a}{\w_a:\w_a},\qquad a\in\{L,R\}.
 \label{eq:G-endpoint-projector}
\end{equation}
Since $g_a>0$ implies $\w_a\ne0$, $P_a$ is well defined and
$\w_a:P_a=1$.  Define the two constraint residuals
\begin{align}
 R_L^G&:=h_G^t-h_{0,L},&
 R_R^G&:=h_{0,R}-h_G^t.
 \label{eq:G-projection-residuals}
\end{align}
Using \eqref{eq:G-baseline-h-candidates},
\eqref{eq:G-central-mass-flux}, and the determinant polarizations
\eqref{eq:G-cofactor-polarization}--\eqref{eq:G-determinant-polarization},
these residuals reduce to
\begin{equation}
 \begin{aligned}
 R_L^G&=\frac14\Delta g\,\Delta v_n
       +\frac14\overline v_n(a_{\Delta A}-d_{\Delta A}),\\
 R_R^G&=-\frac14\Delta g\,\Delta v_n
       -\frac14\overline v_n(a_{\Delta A}+d_{\Delta A}).
 \end{aligned}
 \label{eq:G-projection-residuals-closed}
\end{equation}
Lemma~\ref{lem:app-single-projection} then gives
\begin{equation}
 \boxed{\begin{aligned}
 \mathcal H_A^{-,G}&=\mathcal H_{A,0}^-+R_L^GP_L,\\
 \mathcal H_A^{+,G}&=\mathcal H_{A,0}^++R_R^GP_R.
 \end{aligned}}
 \label{eq:G-volume-pair}
\end{equation}
Thus both endpoints reconstruct $h_G^t$, and \eqref{eq:G-pair} follows
immediately.  All objects in \eqref{eq:G-volume-pair} have been defined in
\eqref{eq:G-polarization-data},
\eqref{eq:G-baseline-minus}--\eqref{eq:G-baseline-h-candidates}, and
\eqref{eq:G-central-mass-flux}--\eqref{eq:G-endpoint-projector}.

\begin{lemma}[Minimum-change determinant companion pair]
\label{lem:G-core-companion}
For $g_L,g_R>0$, problem \eqref{eq:G-volume-optimization}--
\eqref{eq:G-volume-constraints} has the unique solution
\eqref{eq:G-volume-pair}, which contains no jump quotient.  It satisfies
\begin{equation}
 h_L+\w_L:\mathcal H_A^{-,G}
 =h_R-\w_R:\mathcal H_A^{+,G}=h_G^t,
 \qquad
 \w_L:\mathcal H_A^{-,G}+\w_R:\mathcal H_A^{+,G}=h_R-h_L.
 \label{eq:G-core-lemma-identities}
\end{equation}
It also has the fixed-normal exchange symmetry
\begin{equation}
 \mathcal H_A^{-,G}(U_R,U_L;n)
 =-\mathcal H_A^{+,G}(U_L,U_R;n).
 \label{eq:G-core-lemma-exchange}
\end{equation}
On every compact admissible set with $g>0$, the pair has a smooth,
direction-independent equal-state extension.
\end{lemma}

The endpoint reconstruction equality and the summed contraction identity
in \eqref{eq:G-core-lemma-identities} follow directly from the two
projection constraints \eqref{eq:G-volume-constraints}.  Uniqueness follows
from Lemma~\ref{lem:app-single-projection}; fixed-normal exchange symmetry
follows from Lemma~\ref{lem:app-GC-exchange-orientation}, and the smooth
equal-state extension from Appendix~\ref{app:GC-DG-proof}.
Lemma~\ref{lem:G-core-companion} is the complete
algebraic content needed from the nondissipative GC-FV core.

For the conservative block $Y=(\m,\mathcal E)^T$, the
nondissipative numerical flux used in GC-FV is the arithmetic mean
\begin{equation}
 f_Y(U;n):=(\mathbf f_m(U)n,
 \mathbf f_{\mathcal E}(U)\cdot n)^T,\qquad
 F_Y^0:=\frac12\{f_Y(U_L;n)+f_Y(U_R;n)\}.
 \label{eq:G-FY0}
\end{equation}
The physical fluxes $\mathbf f_m$ and $\mathbf f_{\mathcal E}$ are
defined in \eqref{eq:physical-Y-fluxes}.  Thus $F_Y^0$ together with
\eqref{eq:G-volume-pair} specifies the nondissipative numerical flux
and distortion fluctuations; the complete GC-FV face operator is
assembled after adding the common LF correction.

\begin{remark}[Path-sum consistency versus algebraic regularity]
\label{rem:G-path-regularity}
In the sense of \eqref{eq:path-sum-definition}, the straight-path baseline
is path-sum consistent with the selected linear path: its two one-sided
fluctuations sum to $\mathcal J_{A,\Phi}$.  The geometric projection changes
this sum, so the projected nondissipative pair is generally no longer
path-sum consistent with that path.  The later LF projection can introduce
an additional defect.  These two defects are quantified in
\eqref{eq:G-volume-path-defect} and \eqref{eq:PLF-path-defect}.
Enforcing a zero correction sum while prescribing the density flux would
retain path-sum consistency according to \eqref{eq:path-sum-definition}, but
it introduces a Gram denominator proportional to
$\|\Delta\w\|_F^2$.  The resulting expression degenerates as the states
coalesce and has no bounded, direction-independent equal-state extension;
we therefore retain the algebraically regular projected pair.
\end{remark}

\subsection{Dissipative fluxes and fluctuations for GC-FV}
\label{sec:dissipative-FV}

The preceding subsection constructed the nondissipative GC-FV face pair.
Shock computations also require numerical viscosity.  We now add local
Lax--Friedrichs/Rusanov state-jump dissipation
\cite{Rusanov1962,CockburnShu2001} to obtain the complete GC-FV pair.  The distortion increment
must be projected so that its induced density law retains exact LF
dissipation; the momentum--energy and thermal-impulse increments are
ordinary LF.

\paragraph{Minimum-change projected LF correction}
Let $a_f$ bound the spectral radius of
the reduced GPR operator normal to a face and put $\ell_f=a_f/2$.  Directly
adding LF dissipation to the distortion fluctuations would give the raw
increments
\begin{equation}
 D_{A,\mathrm{raw}}^-=-\ell_f\Delta\A,\qquad
 D_{A,\mathrm{raw}}^+=+\ell_f\Delta\A.
 \label{eq:G-raw-LF-increments}
\end{equation}
Adding these increments to either core gives contractions
$\w_L:D_{A,\mathrm{raw}}^-=-\ell_f\w_L:\Delta\A$ and
$\w_R:D_{A,\mathrm{raw}}^+=+\ell_f\w_R:\Delta\A$, whereas exact density LF
requires $\mp\ell_f\Delta g$.  The
polarization identities \eqref{eq:G-cofactor-polarization}--
\eqref{eq:G-determinant-polarization} give
\begin{equation}
 \w_L:\Delta\A=\Delta g-\frac12(a_{\Delta A}-d_{\Delta A}),\qquad
 \w_R:\Delta\A=\Delta g+\frac12(a_{\Delta A}+d_{\Delta A}),
 \label{eq:G-endpoint-linearization-errors}
\end{equation}
where $a_{\Delta A},d_{\Delta A}$ are defined in
\eqref{eq:G-jump-polynomials}.  Thus raw
LF misses the exact density jump at finite states because $\det\A$ is
nonlinear.

We now impose the missing density-LF property as a hard constraint while
changing the trusted raw $A$-LF increments as little as possible.  The
semidiscrete determinant chain rule and $\rho=\rho_0g$ require
\begin{equation}
 \w_L:D_A^-=-\ell_f\Delta g,\qquad
 \w_R:D_A^+=+\ell_f\Delta g.
 \label{eq:PLF-optimization-constraints}
\end{equation}
Indeed, multiplication by $\rho_0$ converts these conditions into the
desired left and right density increments $\mp\ell_f\Delta\rho$.  Among
all matrix pairs obeying them, we select the one closest to the raw LF
increments:
\begin{equation}
 \boxed{
 (D_A^-,D_A^+)
 =\arg\min_{(D^-,D^+)}
 \frac12\left(
 \|D^-+\ell_f\Delta\A\|_F^2
 +\|D^+-\ell_f\Delta\A\|_F^2\right)}
 \label{eq:PLF-optimization}
\end{equation}
subject to
\begin{equation}
 \w_L:D^-=-\ell_f\Delta g,\qquad
 \w_R:D^+=+\ell_f\Delta g.
 \label{eq:PLF-optimization-admissible-set}
\end{equation}
The constraints decouple by endpoint, so
Lemma~\ref{lem:app-single-projection} gives the unique solution
\begin{equation}
 \boxed{D_A^-=-\ell_f(\Delta\A+\gamma_LP_L),\qquad
 D_A^+=+\ell_f(\Delta\A+\gamma_RP_R),}
 \label{eq:PLF}
\end{equation}
where the polynomial identities give
\begin{equation}
 \gamma_L=\Delta g-\w_L:\Delta\A
 =\frac12(a_{\Delta A}-d_{\Delta A}),\qquad
 \gamma_R=\Delta g-\w_R:\Delta\A
 =-\frac12(a_{\Delta A}+d_{\Delta A}).
 \label{eq:PLF-data}
\end{equation}
Here $a_{\Delta A}=\Delta\w:\Delta\A$ and
$d_{\Delta A}=\det(\Delta\A)$ are defined in
\eqref{eq:G-jump-polynomials}.
The final stabilized distortion fluctuations are
\begin{equation}
 \widehat{\mathcal H}_{A}^-
 =\mathcal H_A^{-,G}+D_A^-,
 \qquad
 \widehat{\mathcal H}_{A}^+
 =\mathcal H_A^{+,G}+D_A^+.
 \label{eq:G-PLF-stabilized-pair}
\end{equation}
Here $\mathcal H_A^{\pm,G}$ is the nondissipative pair from
\eqref{eq:G-volume-pair}.
The correction is cofactor-parallel, orthogonal to the
determinant-neutral space
\begin{equation}
 \mathcal N_a:=\{X:\w_a:X=0\}.
 \label{eq:PLF-neutral-space}
\end{equation}
because $D(\det)(\A_a)[X]=\w_a:X$.  Thus the four neutral planar modes
retain raw LF damping, while the determinant-active mode supplies exactly
$-\ell_f\Delta\rho$.

\begin{lemma}[Projected LF correction and exact density LF]
\label{lem:G-PLF}
For $g_L,g_R>0$, the minimum-change problem
\eqref{eq:PLF-optimization}--\eqref{eq:PLF-optimization-admissible-set}
has the unique solution \eqref{eq:PLF}--\eqref{eq:PLF-data}, which contains
no jump quotient and is smooth on compact admissible state sets with $g>0$.
For $a=L,R$, define
\begin{equation}
 \Pi_aX:=(\w_a:X)P_a.
 \label{eq:PLF-Pi-definition}
\end{equation}
Then its endpoint increments satisfy
\begin{gather}
 \w_L:D_A^-=-\ell_f\Delta g,\qquad
 \w_R:D_A^+=+\ell_f\Delta g,
 \label{eq:PLF-lemma-contractions}\\
 \Delta\A+\gamma_aP_a=(I-\Pi_a)\Delta\A+\Delta g\,P_a.
 \label{eq:PLF-lemma-mode-split}
\end{gather}
Thus the projection leaves the raw LF increment unchanged in every
determinant-tangent direction and produces the exact scalar LF increment
$\mp\ell_f\Delta\rho$ in the induced density law.  It also obeys the
orientation identity
\begin{equation}
 D_A^-(U_R,U_L;-n)=D_A^+(U_L,U_R;n).
 \label{eq:PLF-lemma-orientation}
\end{equation}
\end{lemma}

The contraction and mode identities are immediate from
$\w_a:P_a=1$ and $\gamma_a=\Delta g-\w_a:\Delta\A$.  Under simultaneous state
exchange and normal reversal, the polynomial rules in
Lemma~\ref{lem:app-GC-exchange-orientation} give
$\Delta\A\mapsto-\Delta\A$, $a_{\Delta A}\mapsto a_{\Delta A}$,
$d_{\Delta A}\mapsto-d_{\Delta A}$, and
$P_L\leftrightarrow P_R$, which proves
\eqref{eq:PLF-lemma-orientation}.  This lemma isolates all properties of the
dissipative projection needed below.

For the conservative block $Y=(\m,\mathcal E)^T$, whose physical flux
$f_Y$ is defined in \eqref{eq:physical-Y-fluxes}, the common ordinary LF
increments are
\begin{equation}
 D_Y^-=-\ell_f\Delta Y,\qquad D_Y^+=+\ell_f\Delta Y.
 \label{eq:J-FV-DY}
\end{equation}
Together \eqref{eq:J-FV-DY} and \eqref{eq:PLF} define the mechanical
dissipation.

The thermal-impulse block receives ordinary LF increments,
\begin{equation}
 D_J^-=-\ell_f\Delta\Jv,\qquad D_J^+=+\ell_f\Delta\Jv.
 \label{eq:GC-FV-DJ}
\end{equation}
No projection is required because $\Jv$ does not enter the determinant
constraint.  They are added to the straight-path baseline
\eqref{eq:G-J-baseline-minus}--\eqref{eq:G-J-baseline-plus}.

\begin{remark}[Incomplete LF strength in the distortion equation]
\label{rem:distortion-density-LF-tradeoff}
In the reduced formulation, geometric compatibility is obtained by deriving
the density update from the distortion equation through the nonlinear map
$\rho=\rho_0\det\A$.  At a finite jump, a raw LF increment in the $A$
equation therefore does not, in general, induce the corresponding LF
increment in the density equation.  Consequently, within this companion-law
construction one cannot in general retain simultaneously the complete raw
$A$-equation LF strength and the exact density LF dissipation.  We choose the
latter: the projection \eqref{eq:PLF} preserves the exact density LF term by
modifying the cofactor-normal part of the $A$ increment, so the resulting
$A$ equation has incomplete LF strength.  The same tradeoff is present in
Boscheri et al.~\cite{Boscheri2024}; moreover, their induced density
dissipation is proportional to
$\Delta g+\tfrac12\det(\Delta\A)$ and hence is not exact density LF in general.
Thus the present construction guarantees the complete density dissipation,
although not the complete distortion dissipation.  None of the numerical
experiments reported below exhibited an instability attributable to this
incomplete $A$-equation LF strength.
\end{remark}

\subsection{Complete GC-FV update and companion law}
\label{sec:complete-FV}

We now combine the nondissipative GC-FV pair with the projected LF
increments and state the resulting cellwise hidden-density balance.

\subsubsection{GC-FV update}
\label{sec:GC-FV}

Using the nondissipative flux $F_Y^0$ in \eqref{eq:G-FY0}, the
complete one-sided face fluctuations are
\begin{align}
 \widehat{\mathcal K}_G^-&=
 \left(F_Y^0-f_Y(U_L;n)-\ell_f\Delta Y,
 \mathcal H_A^{-,G}+D_A^-,
 \mathcal H_{J,0}^-+D_J^-\right)^T,
 \label{eq:GC-FV-left}\\
 \widehat{\mathcal K}_G^+&=
 \left(f_Y(U_R;n)-F_Y^0+\ell_f\Delta Y,
 \mathcal H_A^{+,G}+D_A^+,
 \mathcal H_{J,0}^++D_J^+\right)^T.
 \label{eq:GC-FV-right}
\end{align}
The symmetry of $F_Y^0$ and the exchange identities for the
nondissipative core and projected LF correction imply
\begin{equation}
 \widehat{\mathcal K}_G^-(U_R,U_L;-n)
 =\widehat{\mathcal K}_G^+(U_L,U_R;n).
 \label{eq:GC-FV-orientation}
\end{equation}
For a cell $K$, the complete update is
\begin{equation}
 \boxed{\frac{\dd U_K}{\dd t}
 +\frac1{|K|}\sum_{f\subset\partial K}|f|\,
 \widehat{\mathcal K}_G^{-}(U_K,U_{K_f};n_{Kf})=S_U(U_K),}
 \label{eq:common-FV-scheme}
\end{equation}
where $S_U=(0_Y,S_A,-\betaJ/\theta_2)^T$, $K_f$ is the neighboring
cell, and $n_{Kf}$ is the outward unit normal.

\begin{theorem}[GC-FV determinant companion law]
\label{thm:GC-FV}
For admissible states, constant $\rho_0$, and $\w:S_A=0$,
scheme \eqref{eq:common-FV-scheme} induces
\begin{equation}
 \boxed{\frac{\dd\rho_K}{\dd t}
 +\frac1{|K|}\sum_{f\subset\partial K}|f|\widehat F_{\rho,f}=0,
 \qquad
 \widehat F_{\rho,f}=F_{\rho,f}^0-\ell_f\Delta\rho_f,}
 \label{eq:GC-FV-density}
\end{equation}
with $\rho_K=\rho_0\det\A_K$.  Momentum and total energy are conservative
and have the ordinary central--LF flux.
\end{theorem}

\begin{proof}
For a face $f$ oriented outward from $K$, Lemmas~\ref{lem:G-core-companion}
and~\ref{lem:G-PLF} give
\begin{align}
 \rho_0\w_K:
 \bigl(\mathcal H_{A,f}^{-,G}+D_{A,f}^-\bigr)
 &=\rho_0\bigl(h_{G,f}^t-h_K-\ell_f\Delta g_f\bigr)\notag\\
 &=\widehat F_{\rho,f}-\rho_K\vct_K\cdot n_{Kf}.
 \label{eq:GC-FV-companion-face-proof}
\end{align}
Using $\rho_K=\rho_0\det\A_K$, the exact chain rule
$\dd\rho_K/\dd t=\rho_0\w_K:\dd\A_K/\dd t$, and $\w:S_A=0$, contraction
of the distortion block in \eqref{eq:common-FV-scheme} therefore yields
\begin{equation}
 \frac{\dd\rho_K}{\dd t}
 +\frac1{|K|}\sum_{f\subset\partial K}|f|
 \bigl(\widehat F_{\rho,f}-\rho_K\vct_K\cdot n_{Kf}\bigr)=0.
 \label{eq:GC-FV-companion-intermediate-proof}
\end{equation}
The closed-cell identity
$\sum_{f\subset\partial K}|f|n_{Kf}=0$ cancels the second term inside
the parentheses and proves \eqref{eq:GC-FV-density}.  The same identity,
together with the single-valued central--LF flux in the conservative
$Y$ block, gives the stated momentum and total-energy conservation.
\end{proof}

\begin{remark}[Comparison with Boscheri et al.]
\label{rem:GC-FV-Boscheri-comparison}
Both the present GC-FV method and Boscheri et al.~\cite{Boscheri2024}
recover a conservative determinant law from a reduced GPR system.  The
present construction replaces their jump-quotient flux correction and
additional elementwise production by bounded endpoint projections without
jump quotients.  It thereby provides exact induced density LF dissipation, a
smooth equal-state extension, and the two-state structure used by the
high-order SBP and subcell constructions below.  In Boscheri et al., the
induced density dissipation is proportional to
$\Delta g+\tfrac12\det(\Delta\A)$ rather than $\Delta g$ in general.  Their method,
on the other hand, retains the selected path sum and also enforces a
thermodynamic companion law, whereas the present paper claims only
geometric compatibility.  The two path-sum defects of GC-FV are
\begin{align}
 \mathcal P_A^{V,G}
 &:=
 \mathcal H_A^{-,G}+\mathcal H_A^{+,G}-\mathcal J_A^{\rm lin}
 =R_L^GP_L+R_R^GP_R,
 \label{eq:G-volume-path-defect}\\
 \mathcal P_A^{D,G}
 &:=
 D_A^-+D_A^+
 =-\frac{\ell_f}{2}
 \{a_{\Delta A}(P_L+P_R)+d_{\Delta A}(P_R-P_L)\}.
 \label{eq:PLF-path-defect}
\end{align}
For smooth pairs these are respectively $O(\|\Delta U\|^3)$ and
$O(\ell_f\|\Delta U\|^2)$, as proved in
Appendix~\ref{app:GC-path-estimates}.  Thus, in smooth regions, both defects
vanish at higher order in the state jump.  This local estimate is not a
complete convergence-order analysis.  Across a shock the defects may remain
finite rather than small; nevertheless, none of the numerical experiments
reported below exhibited an instability attributable to path-sum
inconsistency.  The associated trade-offs are detailed in
Remarks~\ref{rem:distortion-density-LF-tradeoff} and
\ref{rem:G-path-regularity}.
\end{remark}

\subsubsection{Wave-speed estimate}
\label{sec:GC-FV-wave-speed}

The LF coefficient $\ell_f=a_f/2$ in \eqref{eq:G-raw-LF-increments}
requires $a_f$ to bound the largest normal signal speed.  A
reference-state estimate is generally insufficient for a non-ideal EOS,
a strongly deformed solid, or an active thermal impulse.  For the mechanical
subsystem we use the classical acoustic-tensor (Christoffel-tensor)
reduction of nonlinear elasticity.  This construction underlies the
characteristic analysis and Godunov solvers in
\cite{TrangensteinColella1991,MillerColella2001,BartonEtAl2009,Barton2010},
the more recent first-order hyperbolic solid formulation
\cite{DiGiusto2024}, and the GPR characteristic analysis of
\cite{Dumbser2016}.  Its derivation for the energy class used here is given
in Appendix~\ref{app:acoustic-tensor}.

Let $\{n,t\}$ be a local orthonormal basis of the computational plane and
let $W_n$ denote the planar acoustic tensor represented in this basis.  The directional
increments $\chi$, $\delta\Sigma_G$, and $\delta\Psi$ appearing below are
defined in \eqref{eq:app-acoustic-increments}.  Once the two columns
$W_n n$ and $W_n t$ have been evaluated, we define
\begin{align}
 \mu_n(W_n)&:=\frac12\left[
 (W_n)_{11}+(W_n)_{22}
 +\sqrt{\bigl((W_n)_{11}-(W_n)_{22}\bigr)^2
       +4(W_n)_{12}^2}\right],
 \label{eq:acoustic-largest-eigenvalue}\\
 a_f[W_n]&:=C_{\rm LF}\max_{a\in\{L,R\}}
 \left(|v_{n,a}|+\sqrt{\mu_{n,a}(W_n)}\right).
 \label{eq:finite-deformation-LF-speed}
\end{align}
Thus no eigendecomposition of the full mechanical Jacobian is required.
For the two EOS classes used in Section~\ref{sec:model}, the formulas used to
assemble $W_n$ are as follows.
\begin{itemize}
\item \emph{Ideal or stiffened gas.}
From \eqref{eq:stiffened-gas-EOS},
\begin{equation}
 c_{0,{\rm SG}}^2=\frac{\gamma(p+p_\infty)}{\rho},
 \qquad
 W_n^{\rm SG}\xi=c_{0,{\rm SG}}^2\chi n
 +c_s^2(\chi\Sigma_G+\delta\Sigma_G)n.
 \label{eq:acoustic-ideal-action}
\end{equation}
Substitution into
\eqref{eq:acoustic-largest-eigenvalue}--\eqref{eq:finite-deformation-LF-speed}
gives $a_f^{\rm SG}$.  At the undeformed reference state,
$\mu_n=c_{0,{\rm SG}}^2+4c_s^2/3$.

\item \emph{Barton--Romenski material.}
For \eqref{eq:Barton-volumetric-energy}--\eqref{eq:Barton-shear-energy},
\begin{align}
 c_{0,{\rm BR}}^2={}&
 \frac{K_0}{\alpha_B}
 \left[(2\alpha_B+1)\nu^{2\alpha_B}
 -(\alpha_B+1)\nu^{\alpha_B}\right]
 \notag\\
 &+\gamma_B(\gamma_B+1)c_vT_0\nu^{\gamma_B}
 \left(e^{s/c_v}-1\right),
 \label{eq:Barton-acoustic-c0}
 \\[-1ex]
 W_n^{\rm BR}\xi={}&\left[c_{0,{\rm BR}}^2\chi
 +\frac{\beta_Bb_0^2\nu^{\beta_B}}4
 \bigl((1+\beta_B)\chi\Psi+\delta\Psi\bigr)\right]n
 \notag\\
 &+b_0^2\nu^{\beta_B}
 \bigl((1+\beta_B)\chi\Sigma_G+\delta\Sigma_G\bigr)n.
 \label{eq:acoustic-Barton-action}
\end{align}
Substitution of $W_n^{\rm BR}$ into
\eqref{eq:acoustic-largest-eigenvalue}--\eqref{eq:finite-deformation-LF-speed}
gives $a_f^{\rm BR}$.  At the undeformed thermodynamic reference state
$(\A,\nu,s)=(I_3,1,0)$, one has
$\mu_n=c_{L,0}^2$.
Unlike the gas formulas, the general finite-deformation result cannot be
reduced to one scalar sound speed because the deformed state is
directionally anisotropic; nevertheless,
\eqref{eq:acoustic-Barton-action} and
\eqref{eq:acoustic-largest-eigenvalue} are a closed-form $2\times2$
evaluation.
\end{itemize}

\paragraph{Additional estimate for the $J$-active formulation}
The preceding construction is complete for both mechanical EOS classes and
does not involve the thermal impulse.  When the latter is
active, $W_n$ no longer contains the entire spectrum because the thermal
and longitudinal modes are coupled.
Dumbser et al.~\cite{Dumbser2016} derived exact coupled eigenvalues at the
reference state $A=I_3$, $J=0$; Jackson
\cite{Jackson2017,JacksonNikiforakis2019} instead formulated the Rusanov
coefficient through the spectral radius of the full normal quasilinear
matrix; and Boscheri et al.~\cite{Boscheri2024} explicitly adopted the
heuristic estimate
$\sqrt{\gamma p/\rho+4c_s^2/3+c_h^2}$.  To avoid a full eigendecomposition
at every face while retaining finite-$J$ protection, we use the calibrated
single-parameter estimate below only for the ideal/stiffened-gas $J$-active
configurations considered in this paper:
\begin{equation}
 \boxed{
 a_f^J=C_{\rm LF}\max_{a\in\{L,R\}}
 \left[|v_{n,a}|+
 \sqrt{\mu_{n,a}(W_n^{\rm SG})+c_h^2\frac{T_a}{c_v}
 +\kappa_Jc_h^2|\Jv_a|_2^2}\right],
 \qquad \kappa_J\ge0 .}
 \label{eq:active-J-LF-speed}
\end{equation}
Here $|\Jv_a|_2=(\Jv_a\cdot\Jv_a)^{1/2}$ is the Euclidean norm of the
thermal-impulse vector at side $a\in\{L,R\}$.  The parameter $\kappa_J$
accounts for the finite-$J$ coupling.  The term $c_h^2T_a/c_v$ is the
reference-state thermal-wave contribution supplied by the linearized
constant-$c_v$ GPR closure; it is not an empirical correction.  The term
$\kappa_Jc_h^2|\Jv_a|_2^2$ is a problem-specific empirical safety
correction, not part of the exact mechanical acoustic-tensor reduction.
This estimate is used only with the constant-$c_v$ ideal/stiffened-gas
closure.  For a general caloric EOS, the thermal contribution must be
bounded from the corresponding full normal quasilinear matrix.  Here
$\kappa_J$ is a problem-dependent thermal safety factor, while
$C_{\rm LF}\ge1$ remains the common LF factor; setting $c_h=0$ recovers
\eqref{eq:finite-deformation-LF-speed}.  We use $\kappa_J=10$ in all
$J$-active tests, solely as a reproducible working value rather than a
universal spectral bound.

\section{High-order nodal DGSEM extension}
\label{sec:high-order}

The GC-FV construction supplies a regular two-state fluctuation pair and an
exact determinant companion identity.  We first insert this pair into the
standard tensor-product summation-by-parts (SBP) flux-differencing operator.
This gives the high-order geometric-compatible DGSEM scheme, denoted
\textbf{GC-DG}, in the form in which its geometric compatibility is most
transparent.  We then derive an algebraically equivalent subcell
representation.  This representation
does not define another spatial discretization; it exposes one high-order
flux or directed fluctuation at every LGL subcell interface and is therefore
the form needed by the limiter in Section~\ref{sec:gpr-subcell}.

\subsection{Nodal space and SBP operators}
\label{sec:DG-SBP}

Partition $\Omega$ into affine rectangles $K$ and map the reference square
$\widehat K=[-1,1]^2$ to $K$ by
\begin{equation}
 \Phi_K(\xi_1,\xi_2)
 =\left(x_K+\frac{h_{K,x}}2\xi_1,
 y_K+\frac{h_{K,y}}2\xi_2\right).
 \label{eq:J-map}
\end{equation}
With $\mathbb P_N$ denoting the one-dimensional polynomials of degree at
most $N$, set
\begin{equation}
 \mathbb Q_N(\widehat K)=\mathbb P_N([-1,1])\otimes\mathbb P_N([-1,1]),
 \qquad
 V_h^N=\{v\in L^2(\Omega):v|_K\circ\Phi_K\in\mathbb Q_N(\widehat K)\}.
 \label{eq:J-DG-space}
\end{equation}
Let $\{\xi_i,\omega_i\}_{i=0}^N$ be the Legendre--Gauss--Lobatto (LGL)
rule and $l_i$ its nodal Lagrange basis.  The same nodes $\xi_i$ are used
in both reference directions.  On each element,
\begin{equation}
 U_h|_K\circ\Phi_K
 =\sum_{i,j=0}^NU_{ij}^Kl_i(\xi_1)l_j(\xi_2),
 \quad D^{\rm SBP}_{im}=l_m'(\xi_i),
 \quad M=\operatorname{diag}(\omega_0,\ldots,\omega_N),
 \quad Q=MD^{\rm SBP}.
 \label{eq:J-nodal}
\end{equation}
The diagonal-norm SBP identities are
\begin{equation}
 \boxed{Q+Q^T=B_{\rm SBP},\qquad Q\mathbf1=0,
 \qquad B_{\rm SBP}=\operatorname{diag}(-1,0,\ldots,0,1).}
 \label{eq:J-SBP}
\end{equation}
We use $\lambda_{K,x}=2/h_{K,x}$,
$\lambda_{K,y}=2/h_{K,y}$ and
$J_K=h_{K,x}h_{K,y}/4$.  The face Jacobian is $J_f=h_{K,y}/2$ on a
vertical face and $J_f=h_{K,x}/2$ on a horizontal face.

\subsection{The two-dimensional high-order GC-DG scheme}
\label{sec:GC-DG}

We now assemble the high-order DGSEM scheme directly from the two-state
fluctuations constructed for GC-FV.  The underlying step is the classical
SBP flux-differencing construction, originating in high-order SBP finite
differences and split-form nodal DGSEM
\cite{FisherCarpenter2013,CarpenterFisherNielsenFrankel2014,
GassnerWintersKopriva2016,Chan2018}.  To recall it briefly, for the scalar
law $u_t+f(u)_x=0$, let $f^\star$ be a symmetric consistent two-point flux
and write its nondissipative FV fluctuation as
$\mathcal D_0^-(u_L,u_R)=f^\star(u_L,u_R)-f(u_L)$.  On the LGL nodes of one
element, its SBP lift is
\begin{equation}
 \mathcal L_i^H[\mathcal D_0^-]
 :=2\sum_{m=0}^ND^{\rm SBP}_{im}\mathcal D_0^-(u_i,u_m)
 =2\sum_{m=0}^ND^{\rm SBP}_{im}f^\star(u_i,u_m),
 \label{eq:G-scalar-SBP-lift}
\end{equation}
where the second equality follows from $D^{\rm SBP}\mathbf1=0$.  Symmetry of
$f^\star$ and $Q+Q^T=B_{\rm SBP}$ give
\begin{equation}
 \sum_{i=0}^N\omega_i\mathcal L_i^H[\mathcal D_0^-]
 =f(u_N)-f(u_0).
 \label{eq:G-scalar-SBP-balance}
\end{equation}
Thus all nodal-pair interactions telescope to the element endpoints.  A
boundary \emph{simultaneous approximation term} (SAT) then replaces the
physical endpoint flux by the complete dissipative FV numerical flux.  The
SAT is the standard SBP boundary penalty that weakly imposes the numerical
interface coupling while preserving the discrete integration-by-parts
balance \cite{FisherCarpenter2013,CarpenterFisherNielsenFrankel2014}.  In
this precise sense, the high-order
operator is obtained by applying the SBP lift to the nondissipative FV
pair, while retaining the full FV operator at element interfaces.

We apply exactly this assembly rule blockwise to GC-FV.  No new two-state
function is introduced: for two states separated by a face with oriented
normal $n$, the volume building block is the nondissipative part of the
GC-FV pair,
\begin{equation}
 \begin{aligned}
 \mathcal K^{-,G}(U_L,U_R;n)
 &=\left(F_Y^0-f_Y(U_L;n),\,
          \mathcal H_A^{-,G},\,\mathcal H_{J,0}^-\right)^T,\\
 \mathcal K^{+,G}(U_L,U_R;n)
 &=\left(f_Y(U_R;n)-F_Y^0,\,
          \mathcal H_A^{+,G},\,\mathcal H_{J,0}^+\right)^T.
 \end{aligned}
 \label{eq:G-full-pair}
\end{equation}
Here $F_Y^0$ is defined in \eqref{eq:G-FY0},
$\mathcal H_A^{\pm,G}$ is the projected geometric pair
\eqref{eq:G-volume-pair}, and $\mathcal H_{J,0}^\pm$ is the unmodified
straight-path pair
\eqref{eq:G-J-baseline-minus}--\eqref{eq:G-J-baseline-plus}.  Thus the
minimum-change projection to the prescribed density flux acts only on
$A$; the $J$ equation retains
the standard path-conservative discretization.  At an element face we use
the same complete dissipative GC-FV operator
$\widehat{\mathcal K}_G^-$ in \eqref{eq:GC-FV-left}, evaluated with the
local outward normal.

Applying \eqref{eq:G-scalar-SBP-lift} on every $x$- and $y$-directed LGL
line, and adding the GC-FV boundary operators as SATs, gives the
two-dimensional strong-form GC-DG equation at node $(i,j)$:
\begin{align}
 \frac{\dd U_{ij}^K}{\dd t}
 &+2\lambda_{K,x}\sum_{m=0}^ND^{\rm SBP}_{im}
 \mathcal K^{-,G}(U_{ij}^K,U_{mj}^K;e_1)
 +2\lambda_{K,y}\sum_{m=0}^ND^{\rm SBP}_{jm}
 \mathcal K^{-,G}(U_{ij}^K,U_{im}^K;e_2)
 \nonumber\\
 &+\frac{\lambda_{K,x}}{\omega_i}\left[
 \delta_{i0}\widehat{\mathcal K}_G^-(U_{0j}^K,U_{0j}^{K,\rm ext};-e_1)
 +\delta_{iN}\widehat{\mathcal K}_G^-(U_{Nj}^K,U_{Nj}^{K,\rm ext};e_1)
 \right]
 \nonumber\\
 &+\frac{\lambda_{K,y}}{\omega_j}\left[
 \delta_{j0}\widehat{\mathcal K}_G^-(U_{i0}^K,U_{i0}^{K,\rm ext};-e_2)
 +\delta_{jN}\widehat{\mathcal K}_G^-(U_{iN}^K,U_{iN}^{K,\rm ext};e_2)
 \right]
 =S_U(U_{ij}^K).
 \label{eq:GC-DGSEM}
\end{align}
Here $K$ labels the current element; $i,j\in\{0,\ldots,N\}$ label its
tensor-product LGL node, and $m$ is the linewise summation index.  More
precisely, $U_{ij}^K$ is the nodal value of the reduced state $U$ in
\eqref{eq:model-state} at $\Phi_K(\xi_i,\xi_j)$, with the element map,
derivative entries $D^{\rm SBP}_{im}$, and quadrature weights $\omega_i$ defined in
\eqref{eq:J-map}--\eqref{eq:J-nodal}.  The geometric factors are
$\lambda_{K,x}=2/h_{K,x}$ and $\lambda_{K,y}=2/h_{K,y}$,
$e_1=(1,0)^T$ and $e_2=(0,1)^T$ are the Cartesian unit vectors, and
$\delta_{ab}$ is the Kronecker symbol.  The volume fluctuation
$\mathcal K^{-,G}$ is the nondissipative GC-FV pair
\eqref{eq:G-full-pair}, whereas $\widehat{\mathcal K}_G^-$ is the complete
dissipative GC-FV face operator \eqref{eq:GC-FV-left}.

The trace $U_{ij}^{K,\rm ext}$ is taken from the neighboring element at the
same LGL face node on a conforming mesh, or from the prescribed boundary
state at a physical boundary.  The signs $-e_1,e_1,-e_2,e_2$ are the local
outward normals on the left, right, bottom, and top faces, respectively;
at a corner both incident simultaneous-approximation terms are included.
Finally, $S_U=(0_Y,S_A,-\betaJ/\theta_2)^T$ is the same relaxation source
as in the GC-FV update \eqref{eq:common-FV-scheme}.  Thus the first two sums
are the $x$- and $y$-directed SBP lifts of the nondissipative GC-FV pair,
while the four Kronecker-supported terms insert the complete GC-FV operator
only at element faces.  The LF correction is not placed in the volume
because the entries of the SBP differentiation matrix have both signs.

Because the high-order DGSEM scheme is obtained by applying the SBP lift
to the GC-FV pair while retaining its complete face operator, contracting
its distortion equation produces a high-order discrete density equation.
Here and below,
$f_\rho(U;n):=\rho\vct\cdot n=\rho_0h(U;n)$ denotes the physical density
flux in the normal direction $n$.
The two identities required below are
\begin{align}
 \rho_0\w_L:\mathcal H_A^{-,G}(U_L,U_R;n)
 &=F_\rho^0(U_L,U_R;n)-f_\rho(U_L;n),
 \label{eq:GC-DG-volume-density-contraction}\\
 \rho_0\w_L:\widehat{\mathcal H}_{A}^-(U_L,U_R;n)
 &=\widehat F_\rho(U_L,U_R;n)-f_\rho(U_L;n),
 \qquad \widehat F_\rho=F_\rho^0-\ell_f\Delta\rho,
 \label{eq:GC-DG-density-face-flux}
\end{align}
where \eqref{eq:GC-DG-volume-density-contraction} is the left-endpoint
identity in Lemma~\ref{lem:G-core-companion}, while
\eqref{eq:GC-DG-density-face-flux} additionally uses
Lemma~\ref{lem:G-PLF}.  Define the density SAT at node $(i,j)$ by
\begin{align}
 \mathcal B_{\rho,ij}^K
 &:=\frac{\lambda_{K,x}}{\omega_i}\left\{
 \delta_{i0}\left[\widehat F_\rho(U_{0j}^K,U_{0j}^{K,\rm ext};-e_1)
                  -f_\rho(U_{0j}^K;-e_1)\right]\right.\nonumber\\
 &\hspace{44mm}\left.
 +\delta_{iN}\left[\widehat F_\rho(U_{Nj}^K,U_{Nj}^{K,\rm ext};e_1)
                  -f_\rho(U_{Nj}^K;e_1)\right]\right\}\nonumber\\
 &\quad+\frac{\lambda_{K,y}}{\omega_j}\left\{
 \delta_{j0}\left[\widehat F_\rho(U_{i0}^K,U_{i0}^{K,\rm ext};-e_2)
                  -f_\rho(U_{i0}^K;-e_2)\right]\right.\nonumber\\
 &\hspace{44mm}\left.
 +\delta_{jN}\left[\widehat F_\rho(U_{iN}^K,U_{iN}^{K,\rm ext};e_2)
                  -f_\rho(U_{iN}^K;e_2)\right]\right\}.
 \label{eq:GC-DG-density-SAT}
\end{align}

\begin{theorem}[GC-DG determinant companion law]
\label{thm:GC-DG-compatibility}
Assume an affine conforming Cartesian mesh, admissible nodal states,
constant $\rho_0$, and the isochoric source condition
$\w:S_A=0$ in \eqref{eq:isochoric-source}.  Then the distortion block of
the high-order DGSEM scheme \eqref{eq:GC-DGSEM} induces, at every LGL node,
\begin{equation}
 \boxed{
 \frac{\dd\rho_{ij}^K}{\dd t}
 +2\lambda_{K,x}\sum_{m=0}^ND^{\rm SBP}_{im}
 F_\rho^0(U_{ij}^K,U_{mj}^K;e_1)
 +2\lambda_{K,y}\sum_{m=0}^ND^{\rm SBP}_{jm}
 F_\rho^0(U_{ij}^K,U_{im}^K;e_2)
 +\mathcal B_{\rho,ij}^K=0,}
 \label{eq:GC-DG-density}
\end{equation}
where $\rho_{ij}^K=\rho_0\det\A_{ij}^K$ and
$\mathcal B_{\rho,ij}^K$ is defined by
\eqref{eq:GC-DG-density-SAT}.  Thus the induced density equation uses the
same symmetric volume flux $F_\rho^0$ and complete face flux
$\widehat F_\rho$ prescribed by the GC-FV construction.
\end{theorem}

\begin{proof}
At every nodal point, the determinant chain rule and the isochoric source
condition give
\begin{equation}
 \frac{\dd\rho_{ij}^K}{\dd t}
 =\rho_0\w_{ij}^K:\frac{\dd\A_{ij}^K}{\dd t},
 \qquad
 \rho_0\w_{ij}^K:S_A(U_{ij}^K)=0.
 \label{eq:GC-DG-nodal-chain-rule}
\end{equation}
For the $x$-directed volume term,
\begin{align}
 &2\lambda_{K,x}\sum_{m=0}^ND^{\rm SBP}_{im}
 \rho_0\w_{ij}^K:\mathcal H_A^{-,G}
 (U_{ij}^K,U_{mj}^K;e_1)\nonumber\\
 &\qquad=
 2\lambda_{K,x}\sum_{m=0}^ND^{\rm SBP}_{im}
 \left[F_\rho^0(U_{ij}^K,U_{mj}^K;e_1)
       -f_\rho(U_{ij}^K;e_1)\right]\nonumber\\
 &\qquad=
 2\lambda_{K,x}\sum_{m=0}^ND^{\rm SBP}_{im}
 F_\rho^0(U_{ij}^K,U_{mj}^K;e_1),
 \label{eq:GC-DG-x-volume-contraction}
\end{align}
where the first equality uses
\eqref{eq:GC-DG-volume-density-contraction} and the second uses
$D^{\rm SBP}\mathbf1=0$.  The same calculation on each $y$-line gives the second
volume term in \eqref{eq:GC-DG-density}.  Finally,
\eqref{eq:GC-DG-density-face-flux} converts the four face fluctuations in
\eqref{eq:GC-DGSEM} exactly into the density SAT
\eqref{eq:GC-DG-density-SAT}.  Substitution into
\eqref{eq:GC-DG-nodal-chain-rule} proves
\eqref{eq:GC-DG-density} at every nodal point.
\end{proof}

The conservative momentum and total-energy blocks retain their standard
DGSEM balances.  Moreover, $D^{\rm SBP}\mathbf1=0$ and the equal-state limits of
the volume and face pairs imply constant preservation in the absence of
sources.  Appendix~\ref{app:GC-DG-proof} shows that the companion projection
is smooth and nonsingular at equal states and that its correction vanishes
quadratically with the state jump.  It therefore introduces no singular or
leading-order perturbation that would be incompatible with the high-order
SBP construction.  A proof and formal convergence-order analysis of the
complete nonlinear GC-DG spatial operator are beyond the scope of this
paper.  The isentropic-vortex study in
Section~\ref{sec:isentropic-vortex} provides numerical verification of its
high-order convergence.

\subsection{Equivalent subcell form}
\label{sec:GC-DG-subcell-equivalent}

We now derive an equivalent \emph{subcell form} of the high-order DGSEM
scheme \eqref{eq:GC-DGSEM}.  The purpose of this rewrite is to expose one
high-order flux or directed fluctuation at every LGL subcell interface, so
that the subcell limiter can be constructed directly.

Our subcell form, also called a staggered form, follows the SBP
flux-differencing construction of Hennemann et
al.~\cite{Hennemann2021}.  We first use this construction to rewrite the
conservative blocks in an algebraically equivalent subcell form.  We then
follow Rueda-Ram\'irez and Gassner (RRG)~\cite{RuedaRamirezGassner2024} to obtain
the corresponding form for the nonconservative blocks.  The latter step
directly yields the subcell form of the $J$ equation but produces a raw distortion form that
does not, in general, retain the determinant companion law.  We therefore
apply the minimum-change projection introduced in
\eqref{eq:G-volume-optimization}--\eqref{eq:G-volume-constraints} to modify
the distortion subcell fluctuations.  The
two equivalence results at the end of this subsection prove that the final
subcell form is algebraically identical to the original GC-DG scheme
\eqref{eq:GC-DGSEM}.

\subsubsection{Subcell form of the conservative terms}
\label{sec:GC-DG-conservative-subcell}

The conservative variables are $Y=(\m,\mathcal E)^T$.  We also apply the
same algebra to the auxiliary density equation, although density is not an
independently evolved variable.  For compactness, let
$\zeta\in\{Y,\rho\}$, and denote its symmetric consistent volume flux by
\begin{equation}
 F_\zeta^{\star,d}(U_L,U_R)=
 \begin{cases}
 F_Y^0(U_L,U_R;e_d),&\zeta=Y,\\
 F_\rho^0(U_L,U_R;e_d),&\zeta=\rho,
 \end{cases}
 \qquad d\in\{x,y\}.
 \label{eq:GC-DG-conservative-volume-flux}
\end{equation}
Here $d\in\{x,y\}$ selects the normal $e_x=e_1$ or $e_y=e_2$,
$F_Y^0$ and $F_\rho^0$ are the GC-FV targets defined in
\eqref{eq:G-FY0} and \eqref{eq:G-central-mass-flux}, respectively, and
the superscript $\star$ emphasizes that the flux is a symmetric two-point
volume quantity.  The label $H$ below denotes the high-order operator.
The corresponding element-boundary flux is denoted by
$\widehat F_\zeta^d$.  Its two blocks are
\begin{equation}
 \widehat F_Y^d=F_Y^0-\ell_f\Delta Y,
 \qquad
 \widehat F_\rho^d=F_\rho^0-\ell_f\Delta\rho,
 \label{eq:GC-DG-conservative-boundary-flux}
\end{equation}
with all quantities evaluated for the two face states and normal $e_d$.

Fix an element $K$ and suppress its superscript on the nodal states.  On
an $x$-line with transverse index $j$, the
integrated flux-differencing residual inherited directly from
\eqref{eq:GC-DGSEM} is
\begin{align}
 \mathcal R_{\zeta,ij}^{H,x}
 ={}&2\sum_{m=0}^NQ_{im}
 F_\zeta^{\star,x}(U_{ij},U_{mj})
 +\delta_{iN}(\widehat F_{\zeta,N+1/2,j}^{x}-f_{\zeta,ij}^{x})
 -\delta_{i0}(\widehat F_{\zeta,-1/2,j}^{x}-f_{\zeta,ij}^{x}).
 \label{eq:GC-DG-conservative-residual-x}
\end{align}
Similarly, on a $y$-line with transverse index $i$,
\begin{align}
 \mathcal R_{\zeta,ij}^{H,y}
 ={}&2\sum_{m=0}^NQ_{jm}
 F_\zeta^{\star,y}(U_{ij},U_{im})
 +\delta_{jN}(\widehat F_{\zeta,i,N+1/2}^{y}-f_{\zeta,ij}^{y})
 -\delta_{j0}(\widehat F_{\zeta,i,-1/2}^{y}-f_{\zeta,ij}^{y}).
 \label{eq:GC-DG-conservative-residual-y}
\end{align}
Here $\mathcal R_{\zeta,ij}^{H,d}$ is the quadrature-integrated directional
residual, $f_{\zeta,ij}^d:=f_\zeta(U_{ij};e_d)$ is the corresponding physical
flux, and the indices $-1/2$ and $N+1/2$ denote the left and right
element boundaries on the current coordinate line.  The matrix
$Q=MD^{\rm SBP}$ and the SBP identity are defined in \eqref{eq:J-SBP}; here
$S_{\rm SBP}:=Q-Q^T$ denotes its skew part.  Using that identity, consistency of
$F_\zeta^{\star,d}$ combines the
diagonal $B_{\rm SBP}$ contribution with the SAT in
\eqref{eq:GC-DG-conservative-residual-x}--
\eqref{eq:GC-DG-conservative-residual-y}.  For example,
\begin{equation}
 \mathcal R_{\zeta,ij}^{H,x}
 =\sum_{m=0}^N(S_{\rm SBP})_{im}F_\zeta^{\star,x}(U_{ij},U_{mj})
 -\delta_{i0}\widehat F_{\zeta,-1/2,j}^{x}
 +\delta_{iN}\widehat F_{\zeta,N+1/2,j}^{x}.
 \label{eq:GC-DG-conservative-skew-x}
\end{equation}

For $i,j=0,\ldots,N-1$, define the internal subcell fluxes
\begin{align}
 \overline F_{\zeta,i+1/2,j}^{H,x}
 &:=\sum_{\ell=0}^{i}\sum_{m=0}^N
 (S_{\rm SBP})_{\ell m}F_\zeta^{\star,x}(U_{\ell j},U_{mj}),
 \label{eq:GC-DG-conservative-prefix-x}\\
 \overline F_{\zeta,i,j+1/2}^{H,y}
 &:=\sum_{\ell=0}^{j}\sum_{m=0}^N
 (S_{\rm SBP})_{\ell m}F_\zeta^{\star,y}(U_{i\ell},U_{im}).
 \label{eq:GC-DG-conservative-prefix-y}
\end{align}
The overbarred quantities are the high-order subcell fluxes at the
interfaces between consecutive LGL nodes; at the four element boundaries
they are set equal to the corresponding $\widehat F_\zeta^d$.  Subtracting
two successive prefix sums gives, including the boundary nodes,
\begin{align}
 \mathcal R_{\zeta,ij}^{H,x}
 &=\overline F_{\zeta,i+1/2,j}^{H,x}
  -\overline F_{\zeta,i-1/2,j}^{H,x},
 \label{eq:GC-DG-conservative-telescope-x}\\
 \mathcal R_{\zeta,ij}^{H,y}
 &=\overline F_{\zeta,i,j+1/2}^{H,y}
  -\overline F_{\zeta,i,j-1/2}^{H,y}.
 \label{eq:GC-DG-conservative-telescope-y}
\end{align}
Consequently the conservative part of the original GC-DG equation is
already, without approximation,
\begin{equation}
 \frac{\dd Y_{ij}}{\dd t}
 =-\frac{\lambda_{K,x}}{\omega_i}
 (\overline F_{Y,i+1/2,j}^{H,x}-\overline F_{Y,i-1/2,j}^{H,x})
 -\frac{\lambda_{K,y}}{\omega_j}
 (\overline F_{Y,i,j+1/2}^{H,y}-\overline F_{Y,i,j-1/2}^{H,y}).
 \label{eq:GC-DG-Y-prefix}
\end{equation}
Equations \eqref{eq:GC-DG-conservative-prefix-x}--
\eqref{eq:GC-DG-conservative-telescope-y} prove the equivalence for every
conservative block; no separate equivalence argument is needed later.

\subsubsection{Subcell form of the nonconservative terms}
\label{sec:GC-DG-nonconservative-subcell}

For a nonconservative block $Z$, the construction of
Rueda-Ram\'irez and Gassner~\cite{RuedaRamirezGassner2024} applies whenever
the two-point volume term admits the local-times-symmetric factorization
\begin{equation}
 \Phi_{Z,km}^{\star,d}
 =\Phi_{Z,k}^{\rm loc,d}\circ\Phi_{Z,km}^{\rm sym,d},
 \qquad
 \Phi_{Z,km}^{\rm sym,d}=\Phi_{Z,mk}^{\rm sym,d}.
 \label{eq:GC-DG-RRG-factorization}
\end{equation}
Together with the symmetric conservative part $F_Z^{\star,d}$, the choices
below satisfy the diagonal consistency
$F_Z^{\star,d}(U,U)+\Phi_Z^{\rm loc,d}(U)\circ
\Phi_Z^{\rm sym,d}(U,U)=F_Z(U;e_d)$, so the rewritten operator has the
physical equal-state limit of the original nonconservative block.  Here
$Z\in\{\A,\Jv\}$, $d\in\{x,y\}$ is the coordinate direction, and
$k,m\in\{0,\ldots,N\}$ are node indices on the current line.  The factor
$\Phi_{Z,k}^{\rm loc,d}$ depends only on the receiving state $U_k$,
$\Phi_{Z,km}^{\rm sym,d}$ is symmetric in the two states, and $\circ$
denotes the scalar, matrix--vector, or tensor action prescribed by that
block.  We write
$F_{Z,km}^{\star,d}:=F_Z^{\star,d}(U_k,U_m)$.  For an internal subcell interface adjacent to
the receiving node $k$, put
$\iota=\min(k,q)$, where $q\in\{k-1,k+1\}$, and define
\begin{equation}
 \mathfrak G_{Z,(k,q)}^{H,d}
 =\sum_{\ell=0}^{\iota}\sum_{m=0}^N
 (S_{\rm SBP})_{\ell m}F_{Z,\ell m}^{\star,d}
 +\Phi_{Z,k}^{\rm loc,d}\circ
  \sum_{\ell=0}^{\iota}\sum_{m=0}^N
 (S_{\rm SBP})_{\ell m}\Phi_{Z,\ell m}^{\rm sym,d}.
 \label{eq:GC-DG-RRG-gamma}
\end{equation}
Here
$F_{Z,\ell m}^{\star,d}:=F_Z^{\star,d}(U_\ell,U_m)$ and
$\Phi_{Z,\ell m}^{\rm sym,d}:=
\Phi_Z^{\rm sym,d}(U_\ell,U_m)$.  Thus
$\mathfrak G_{Z,(k,q)}^{H,d}$ is the directed high-order interface
quantity seen by node $k$ at the subcell interface between $k$ and its
neighbor $q$.  At an element boundary it is chosen so that the directed
fluctuation defined below equals the corresponding $Z$ block of the GC-FV
surface operator $\widehat{\mathcal K}_G^-$ in
\eqref{eq:GC-FV-left}.  In direct analogy with
\eqref{eq:GC-DG-conservative-telescope-x}--
\eqref{eq:GC-DG-conservative-telescope-y}, define the
quadrature-integrated directional residual by
\begin{align}
 \mathcal R_{Z,k}^{H,d}
 &:=\sum_{m=0}^N(S_{\rm SBP})_{km}\left(
 F_{Z,km}^{\star,d}
 +\Phi_{Z,k}^{\rm loc,d}\circ\Phi_{Z,km}^{\rm sym,d}\right)
 -\delta_{k0}\mathfrak G_{Z,(0,-1)}^{H,d}
 +\delta_{kN}\mathfrak G_{Z,(N,N+1)}^{H,d}
 \nonumber\\
 &=\mathfrak G_{Z,(k,k+1)}^{H,d}
  -\mathfrak G_{Z,(k,k-1)}^{H,d}.
 \label{eq:GC-DG-RRG-telescope}
\end{align}
Unlike a conservative flux, the two quantities assigned to the same
subcell interface need not agree, because their local factors are
evaluated at different receiving nodes.  Introducing the one-sided
fluctuations
\begin{equation}
 \Gamma_{Z,k-1/2}^{+,H,d}
 :=F_Z(U_k;e_d)-\mathfrak G_{Z,(k,k-1)}^{H,d},
 \qquad
 \Gamma_{Z,k+1/2}^{-,H,d}
 :=\mathfrak G_{Z,(k,k+1)}^{H,d}-F_Z(U_k;e_d).
 \label{eq:GC-DG-RRG-directed}
\end{equation}
Here $F_Z(U_k;e_d)$ is the physical conservative flux of block $Z$ at
the receiving node.  The superscripts $+$ and $-$ denote the directed
fluctuations assigned to node $k$ from its left and right subcell
interfaces, respectively.  With these definitions,
\eqref{eq:GC-DG-RRG-telescope} becomes
\begin{equation}
 \mathcal R_{Z,k}^{H,d}
 =\Gamma_{Z,k-1/2}^{+,H,d}
  +\Gamma_{Z,k+1/2}^{-,H,d}.
 \label{eq:GC-DG-RRG-fluctuation-form}
\end{equation}
Thus $\mathcal R_{Z,k}^{H,d}$ plays exactly the same role as
$\mathcal R_{\zeta,ij}^{H,d}$ in the conservative rewrite: it is the
directional spatial residual appearing with the corresponding
$x$- or $y$-direction factor in \eqref{eq:GC-DGSEM}.

\begin{remark}[Endpoint allocation and the nonconservative subcell form]
\label{rem:G-endpoint-subcell}
The endpoint allocation selected in
Remark~\ref{rem:G-endpoint-versus-half-split} enters here directly: the
receiving-node factor is evaluated at the corresponding endpoint state,
while the second factor is the same symmetric two-point velocity.  Hence
the straight-path $A$ and $J$ pairs immediately have the form
\eqref{eq:GC-DG-RRG-factorization}.  A standard half-jump allocation may
also admit a geometric-compatible subcell representation, but its
local--symmetric decomposition is less direct and is outside the scope of
this work.
\end{remark}

This standard construction directly yields the thermal-impulse block.
Indeed, with $n=e_d$,
\begin{equation}
 F_J^\star=\overline v_n\,\overline\Jv+\overline T n,
 \qquad
 \Phi_J^{\rm loc}(U;n)\mathbf z
 =\mathcal C_J(\Jv,n)\mathbf z
 :=n(\Jv\cdot\mathbf z)-(n\cdot\mathbf z)\Jv,
 \qquad
 \Phi_J^{\rm sym}=\overline\vct.
 \label{eq:GC-DG-J-RRG-data}
\end{equation}
Here the bars denote arithmetic two-state averages, $v_n=\vct\cdot n$,
and $\mathbf z$ is the vector argument of the local linear operator
$\mathcal C_J$.  These choices identify, respectively, the symmetric
conservative flux, receiving-state local operator, and symmetric factor
required in \eqref{eq:GC-DG-RRG-factorization}.  This is exactly the split
form already represented by
$\mathcal H_{J,0}^\pm$ in \eqref{eq:G-full-pair}.  Applying
\eqref{eq:GC-DG-RRG-gamma} linewise in both directions therefore gives
\begin{equation}
 \frac{\dd\Jv_{ij}}{\dd t}
 =-\frac{\lambda_{K,x}}{\omega_i}
 (\Gamma_{J,i-1/2,j}^{+,H,x}+\Gamma_{J,i+1/2,j}^{-,H,x})
 -\frac{\lambda_{K,y}}{\omega_j}
 (\Gamma_{J,i,j-1/2}^{+,H,y}+\Gamma_{J,i,j+1/2}^{-,H,y})
 -\frac{(\betaJ)_{ij}}{\theta_{2,ij}}.
 \label{eq:GC-DG-J-subcell}
\end{equation}
This is an exact rewrite of the $J$ block of \eqref{eq:GC-DGSEM}; no
projection to a prescribed density flux is required.

The distortion block has the analogous factorization
\begin{equation}
 F_A^\star=\overline v_n\,\overline\A,
 \qquad
 \Phi_A^{\rm loc}(U;n)\mathbf z
 =\mathcal C_A(\A,n)\mathbf z
 :=(\A\mathbf z)\otimes n-(n\cdot\mathbf z)\A,
 \qquad
 \Phi_A^{\rm sym}=\overline\vct.
 \label{eq:GC-DG-A-factorization}
\end{equation}
As in \eqref{eq:GC-DG-J-RRG-data}, the bars are arithmetic two-state
averages and $\mathbf z$ is the vector argument of the receiving-state linear
operator $\mathcal C_A$.  Thus the three displayed quantities are the
symmetric conservative flux, local operator, and symmetric factor in
\eqref{eq:GC-DG-RRG-factorization}.
Hence \eqref{eq:GC-DG-RRG-gamma} gives the raw, unprojected directed
fluctuations $\Gamma_{A,k\pm1/2}^{\mp,H,d}$.  Their two-dimensional
subcell form is
\begin{equation}
 \left(\frac{\dd\A_{ij}}{\dd t}\right)_{\rm raw}
 =-\frac{\lambda_{K,x}}{\omega_i}
 (\Gamma_{A,i-1/2,j}^{+,H,x}+\Gamma_{A,i+1/2,j}^{-,H,x})
 -\frac{\lambda_{K,y}}{\omega_j}
 (\Gamma_{A,i,j-1/2}^{+,H,y}+\Gamma_{A,i,j+1/2}^{-,H,y})+S_{A,ij}.
 \label{eq:GC-DG-A-raw-staggered}
\end{equation}
This is the standard subcell rewrite of the straight-path distortion
baseline.  Unlike the $J$ equation, however, the raw form
\eqref{eq:GC-DG-A-raw-staggered} does not inherit the prescribed determinant
companion law and must therefore be modified.

\subsubsection{Geometric-compatible subcell form for the distortion}
\label{sec:GC-DG-A-subcell}

Following the preceding construction, we now modify the raw distortion
subcell fluctuations so that their cofactor contractions reproduce the
same subcell density equation as the conservative rewrite.  This is the
subcell counterpart of the minimum-change construction
\eqref{eq:G-volume-optimization}--\eqref{eq:G-volume-constraints} used for
GC-FV.

Let $\overline F_\rho^{H,d}$ be the subcell density flux already
constructed in \eqref{eq:GC-DG-conservative-prefix-x}--
\eqref{eq:GC-DG-conservative-prefix-y}.  At each receiving node, we seek the
smallest Frobenius modification of the two raw $A$ fluctuations subject to
the required density contractions:
\begin{align}
 \widetilde\Gamma_{A,k+1/2}^{-,H,d}
 &=\arg\min_X\frac12\|X-\Gamma_{A,k+1/2}^{-,H,d}\|_F^2,
 &\rho_0\w_k:X
 &=\overline F_{\rho,k+1/2}^{H,d}-f_{\rho,k}^{d},
 \label{eq:GC-DG-subcell-opt-right}\\
 \widetilde\Gamma_{A,k-1/2}^{+,H,d}
 &=\arg\min_X\frac12\|X-\Gamma_{A,k-1/2}^{+,H,d}\|_F^2,
 &\rho_0\w_k:X
 &=f_{\rho,k}^{d}-\overline F_{\rho,k-1/2}^{H,d}.
 \label{eq:GC-DG-subcell-opt-left}
\end{align}
In these two independent problems, $X$ is a matrix of the same size as
$\A$, $\w_k=\cof\A_k$ is the cofactor at the receiving node, and
$f_{\rho,k}^d:=\rho_k v_{d,k}$ is its physical density flux in direction
$d$.  The target $\overline F_{\rho,k\pm1/2}^{H,d}$ is the single-valued
subcell density flux already defined above.  These are the same
one-constraint minimum-change problems as in the GC-FV construction; they
are solved explicitly by
Lemma~\ref{lem:app-single-projection}.  We again use the cofactor
direction $P_k$ employed in that projection, defined by
\eqref{eq:G-endpoint-projector} with the
endpoint state replaced by the receiving state $U_k$, and set
\begin{equation}
 h_k^d=g_kv_{d,k},\qquad
 \overline h_{k\pm1/2}^{H,d}
 =\frac{\overline F_{\rho,k\pm1/2}^{H,d}}{\rho_0}.
 \label{eq:GC-DG-staggered-data}
\end{equation}
Here $g_k=\det\A_k$, $v_{d,k}=\vct_k\cdot e_d$, $h_k^d$ is the physical
determinant flux, $\overline h_{k\pm1/2}^{H,d}$ is its prescribed subcell
counterpart, and $\w_k:P_k=1$.  The explicit solution is
\begin{equation}
 \boxed{\begin{aligned}
 \widetilde\Gamma_{A,k+1/2}^{-,H,d}
 &=\Gamma_{A,k+1/2}^{-,H,d}
 +\left(\overline h_{k+1/2}^{H,d}-h_k^d
 -\w_k:\Gamma_{A,k+1/2}^{-,H,d}\right)P_k,\\
 \widetilde\Gamma_{A,k-1/2}^{+,H,d}
 &=\Gamma_{A,k-1/2}^{+,H,d}
 +\left(h_k^d-\overline h_{k-1/2}^{H,d}
 -\w_k:\Gamma_{A,k-1/2}^{+,H,d}\right)P_k.
 \end{aligned}}
 \label{eq:GC-DG-staggered-projection}
\end{equation}
In particular,
\begin{equation}
 \rho_0\w_k:\widetilde\Gamma_{A,k+1/2}^{-,H,d}
 =\overline F_{\rho,k+1/2}^{H,d}-f_{\rho,k}^{d},
 \qquad
 \rho_0\w_k:\widetilde\Gamma_{A,k-1/2}^{+,H,d}
 =f_{\rho,k}^{d}-\overline F_{\rho,k-1/2}^{H,d}.
 \label{eq:GC-DG-staggered-companion}
\end{equation}

Combining the three blocks gives the complete two-dimensional subcell
representation
\begin{equation}
 \boxed{\begin{aligned}
 \frac{\dd Y_{ij}}{\dd t}
 &=-\frac{\lambda_{K,x}}{\omega_i}
 (\overline F_{Y,i+1/2,j}^{H,x}-\overline F_{Y,i-1/2,j}^{H,x})
 -\frac{\lambda_{K,y}}{\omega_j}
 (\overline F_{Y,i,j+1/2}^{H,y}-\overline F_{Y,i,j-1/2}^{H,y}),\\
 \frac{\dd\A_{ij}}{\dd t}
 &=-\frac{\lambda_{K,x}}{\omega_i}
 (\widetilde\Gamma_{A,i-1/2,j}^{+,H,x}
  +\widetilde\Gamma_{A,i+1/2,j}^{-,H,x})
 -\frac{\lambda_{K,y}}{\omega_j}
 (\widetilde\Gamma_{A,i,j-1/2}^{+,H,y}
  +\widetilde\Gamma_{A,i,j+1/2}^{-,H,y})+S_{A,ij},\\
 \frac{\dd\Jv_{ij}}{\dd t}
 &=-\frac{\lambda_{K,x}}{\omega_i}
 (\Gamma_{J,i-1/2,j}^{+,H,x}+\Gamma_{J,i+1/2,j}^{-,H,x})
 -\frac{\lambda_{K,y}}{\omega_j}
 (\Gamma_{J,i,j-1/2}^{+,H,y}+\Gamma_{J,i,j+1/2}^{-,H,y})
 -\frac{(\betaJ)_{ij}}{\theta_{2,ij}}.
 \end{aligned}}
 \label{eq:GC-DG-subcell-form}
\end{equation}
All nodal indices, LGL weights, geometric factors, and source terms in
\eqref{eq:GC-DG-subcell-form} have the same meanings as in
\eqref{eq:GC-DGSEM}.  The conservative subcell fluxes
$\overline F_Y^{H,d}$ are defined by
\eqref{eq:GC-DG-conservative-prefix-x}--
\eqref{eq:GC-DG-conservative-prefix-y}; the thermal-impulse fluctuations
$\Gamma_J^{\pm,H,d}$ follow from
\eqref{eq:GC-DG-RRG-gamma}--\eqref{eq:GC-DG-RRG-directed}; and the distortion
fluctuations $\widetilde\Gamma_A^{\pm,H,d}$ are the projected quantities in
\eqref{eq:GC-DG-staggered-projection}.

\begin{lemma}[Equivalence of the geometric-compatible distortion subcell form]
\label{lem:GC-DG-A-subcell-equivalence}
The distortion equation in
\eqref{eq:GC-DG-subcell-form}, constructed with
$\widetilde\Gamma_A^{\pm,H,d}$, is identical at every nodal point to the distortion
block of the high-order DGSEM scheme \eqref{eq:GC-DGSEM}.
\end{lemma}

\begin{proof}
The key observation is that, at a fixed receiving nodal point, the
cofactor-normal minimum-change projection commutes with the linear weighted
summation defining the SBP lift.  Indeed, the projection direction depends
only on the receiving state, whereas both the Frobenius contraction and the
SBP lifting sum are linear.
At one node and in one direction, let
$\mathcal R_{A,k}^{\rm DG,H,d}$ denote the quadrature-integrated directional
residual of the distortion block in the original high-order DGSEM scheme,
and set
\begin{equation}
 \mathcal R_{A,k}^{\rm raw,H,d}
 :=\Gamma_{A,k-1/2}^{+,H,d}
  +\Gamma_{A,k+1/2}^{-,H,d},
 \qquad
 \mathcal R_{g,k}^{H,d}
 :=\overline h_{k+1/2}^{H,d}
  -\overline h_{k-1/2}^{H,d}.
 \label{eq:GC-DG-A-raw-density-residuals}
\end{equation}
Adding the two projected fluctuations in
\eqref{eq:GC-DG-staggered-projection} gives
\begin{equation}
 \widetilde\Gamma_{A,k-1/2}^{+,H,d}
 +\widetilde\Gamma_{A,k+1/2}^{-,H,d}
 =\mathcal R_{A,k}^{\rm raw,H,d}
 +\left(\mathcal R_{g,k}^{H,d}
 -\w_k:\mathcal R_{A,k}^{\rm raw,H,d}\right)P_k.
 \label{eq:GC-DG-projection-commutes}
\end{equation}
We now derive the original DGSEM residual in the same form.  By
\eqref{eq:G-volume-pair}, its ordered-pair distortion fluctuation is the
raw straight-path fluctuation plus $c_{km}P_k$, where
$c_{km}:=R_L^G(U_k,U_m;e_d)$ is defined by
\eqref{eq:G-projection-residuals}.  Equal-state consistency gives
$c_{kk}=0$.  Hence, using $2Q=S_{\rm SBP}+B_{\rm SBP}$, the diagonal
$B_{\rm SBP}$ contribution to the lifted correction vanishes.  Applying the
linear SBP lift and using that $P_k$ depends only on the receiving node then
give
\begin{equation}
 \begin{aligned}
 \mathcal R_{A,k}^{\rm DG,H,d}
 &=\mathcal R_{A,k}^{\rm raw,H,d}
   +\sum_{m=0}^N(S_{\rm SBP})_{km}(c_{km}P_k)\\
 &=\mathcal R_{A,k}^{\rm raw,H,d}+C_k^dP_k,
 \qquad
 C_k^d:=\sum_{m=0}^N(S_{\rm SBP})_{km}c_{km}.
 \end{aligned}
 \label{eq:GC-DG-projection-prefix-commute}
\end{equation}
Equation \eqref{eq:GC-DG-projection-prefix-commute} is the explicit use of
the commutation stated at the beginning of the proof.  The receiving-node
companion identity \eqref{eq:GC-DG-volume-density-contraction}, together
with the common GC-FV face operator, requires the original DGSEM residual
to induce the prescribed density residual; hence
\begin{equation}
 \w_k:\mathcal R_{A,k}^{\rm DG,H,d}
 =\mathcal R_{g,k}^{H,d}.
 \label{eq:GC-DG-original-residual-contraction}
\end{equation}
Contracting \eqref{eq:GC-DG-projection-prefix-commute} with $\w_k$ and using
$\w_k:P_k=1$ therefore yields
\begin{equation}
 C_k^d
 =\mathcal R_{g,k}^{H,d}
  -\w_k:\mathcal R_{A,k}^{\rm raw,H,d}.
 \label{eq:GC-DG-original-projection-coefficient}
\end{equation}
Substitution into \eqref{eq:GC-DG-projection-prefix-commute} gives
\begin{align}
 \mathcal R_{A,k}^{\rm DG,H,d}
 &=\mathcal R_{A,k}^{\rm raw,H,d}
 +\left(\mathcal R_{g,k}^{H,d}
 -\w_k:\mathcal R_{A,k}^{\rm raw,H,d}\right)P_k \nonumber\\
 &=\widetilde\Gamma_{A,k-1/2}^{+,H,d}
  +\widetilde\Gamma_{A,k+1/2}^{-,H,d},
 \label{eq:GC-DG-original-equals-subcell-A}
\end{align}
where the second equality is precisely
\eqref{eq:GC-DG-projection-commutes}.  Thus the original DGSEM and projected
subcell constructions have identical directional distortion residuals.
Multiplication by the same geometric and quadrature factors, summation of
the two coordinate directions, and addition of the same source term prove
equivalence at every nodal point.
\end{proof}

\begin{theorem}[Equivalence and companion law of the complete subcell form]
\label{thm:GC-DG-subcell-equivalence}
The complete subcell scheme
\eqref{eq:GC-DG-subcell-form} is identical at every nodal point to the high-order
DGSEM scheme \eqref{eq:GC-DGSEM}.  Its distortion equation induces
the two-dimensional nodal density equation
\begin{equation}
 \boxed{
 \frac{\dd\rho_{ij}}{\dd t}
 =-\frac{\lambda_{K,x}}{\omega_i}
 (\overline F_{\rho,i+1/2,j}^{H,x}
 -\overline F_{\rho,i-1/2,j}^{H,x})
 -\frac{\lambda_{K,y}}{\omega_j}
 (\overline F_{\rho,i,j+1/2}^{H,y}
 -\overline F_{\rho,i,j-1/2}^{H,y}).}
 \label{eq:GC-DG-nodal-mass}
\end{equation}
Equation \eqref{eq:GC-DG-nodal-mass} is the subcell-flux representation of
the high-order density equation \eqref{eq:GC-DG-density}.  Its global
semidiscrete consequence is stated later in
Theorem~\ref{thm:global-semidiscrete-mass}.
\end{theorem}

\begin{proof}
The conservative equivalence has already been established by
\eqref{eq:GC-DG-conservative-telescope-x}--
\eqref{eq:GC-DG-conservative-telescope-y}; the $J$ equivalence is the RRG
identity \eqref{eq:GC-DG-RRG-telescope}; and the $A$ equivalence is
Lemma~\ref{lem:GC-DG-A-subcell-equivalence}.  It remains to derive the
companion equation.  Contract the $A$ equation in
\eqref{eq:GC-DG-subcell-form} with $\rho_0\w_{ij}$ and use
$\dd\rho/\dd t=\rho_0\w:\dd\A/\dd t$,
\eqref{eq:GC-DG-staggered-companion}, and $\w:S_A=0$.  In each direction,
\begin{align}
 \rho_0\w_k:\left(
 \widetilde\Gamma_{A,k-1/2}^{+,H,d}
 +\widetilde\Gamma_{A,k+1/2}^{-,H,d}\right)
 &=\left(f_{\rho,k}^d-\overline F_{\rho,k-1/2}^{H,d}\right)
  +\left(\overline F_{\rho,k+1/2}^{H,d}-f_{\rho,k}^d\right)
 \nonumber\\
 &=\overline F_{\rho,k+1/2}^{H,d}
  -\overline F_{\rho,k-1/2}^{H,d}.
 \label{eq:GC-DG-subcell-density-telescope}
\end{align}
Substitution into the distortion equation gives precisely
\eqref{eq:GC-DG-nodal-mass}, which is equivalent at every nodal point to
\eqref{eq:GC-DG-density} by the already established operator equivalence.
\end{proof}

\subsection{The GC-FV scheme in low-order subcell form}
\label{sec:GC-FV-low-subcell}

Unlike the high-order GC-DG operator, GC-FV requires no algebraic rewrite:
it already consists of differences of single-valued conservative fluxes
and sums of directed $A$ and $J$ fluctuations.  To place the two schemes
in a common notation, regard every LGL node $(i,j)$ as an algebraic
subcell with quadrature mass
\begin{equation}
 m_{ij}^K=J_K\omega_i\omega_j,
 \qquad J_K=\frac{h_{K,x}h_{K,y}}4.
 \label{eq:SC-subcell-mass}
\end{equation}
No additional grid or polynomial reconstruction is introduced.

Suppress the transverse index and consider two adjacent nodal states
$U_k$ and $U_{k+1}$ in direction $d\in\{x,y\}$.  With
$\ell_{k+1/2}^d=a_{k+1/2}^d/2$, define the low-order conservative and
induced-density fluxes by
\begin{align}
 \overline F_{Y,k+1/2}^{L,d}
 &:=F_Y^0(U_k,U_{k+1};e_d)
    -\ell_{k+1/2}^d(Y_{k+1}-Y_k),
 \label{eq:GC-FV-low-subcell-Y-flux}\\
 \overline F_{\rho,k+1/2}^{L,d}
 &:=F_\rho^0(U_k,U_{k+1};e_d)
    -\ell_{k+1/2}^d(\rho_{k+1}-\rho_k).
 \label{eq:GC-FV-low-subcell-rho-flux}
\end{align}
Here and below, the superscript $L$ denotes a low-order subcell quantity,
in direct contrast to the high-order quantities marked by $H$ in
Section~\ref{sec:GC-DG-subcell-equivalent}.  The corresponding directed
low-order fluctuations are
\begin{align}
 \widetilde\Gamma_{A,k+1/2}^{-,L,d}
 &:=\mathcal H_A^{-,G}(U_k,U_{k+1};e_d)
    +D_A^-(U_k,U_{k+1};e_d),\notag\\
 \widetilde\Gamma_{A,k+1/2}^{+,L,d}
 &:=\mathcal H_A^{+,G}(U_k,U_{k+1};e_d)
    +D_A^+(U_k,U_{k+1};e_d),
 \label{eq:GC-FV-low-subcell-A-fluctuations}\\
 \Gamma_{J,k+1/2}^{-,L,d}
 &:=\mathcal H_{J,0}^-(U_k,U_{k+1};e_d)
    +D_J^-(U_k,U_{k+1};e_d),\notag\\
 \Gamma_{J,k+1/2}^{+,L,d}
 &:=\mathcal H_{J,0}^+(U_k,U_{k+1};e_d)
    +D_J^+(U_k,U_{k+1};e_d).
 \label{eq:GC-FV-low-subcell-J-fluctuations}
\end{align}
The tilde again distinguishes the projected distortion fluctuations from
the unmodified thermal-impulse fluctuations.  At an element boundary, the
low- and high-order subcell quantities are both set equal to the
corresponding block of the same complete GC-FV face operator
$\widehat{\mathcal K}_G^\pm$ in
\eqref{eq:GC-FV-left}--\eqref{eq:GC-FV-right}.

With transverse indices restored, the low-order GC-FV scheme on the LGL
subcells is therefore
\begin{equation}
 \boxed{\begin{aligned}
 \frac{\dd Y_{ij}}{\dd t}
 &=-\frac{\lambda_{K,x}}{\omega_i}
 (\overline F_{Y,i+1/2,j}^{L,x}-\overline F_{Y,i-1/2,j}^{L,x})
 -\frac{\lambda_{K,y}}{\omega_j}
 (\overline F_{Y,i,j+1/2}^{L,y}-\overline F_{Y,i,j-1/2}^{L,y}),\\
 \frac{\dd\A_{ij}}{\dd t}
 &=-\frac{\lambda_{K,x}}{\omega_i}
 (\widetilde\Gamma_{A,i-1/2,j}^{+,L,x}
  +\widetilde\Gamma_{A,i+1/2,j}^{-,L,x})
 -\frac{\lambda_{K,y}}{\omega_j}
 (\widetilde\Gamma_{A,i,j-1/2}^{+,L,y}
  +\widetilde\Gamma_{A,i,j+1/2}^{-,L,y})+S_{A,ij},\\
 \frac{\dd\Jv_{ij}}{\dd t}
 &=-\frac{\lambda_{K,x}}{\omega_i}
 (\Gamma_{J,i-1/2,j}^{+,L,x}+\Gamma_{J,i+1/2,j}^{-,L,x})
 -\frac{\lambda_{K,y}}{\omega_j}
 (\Gamma_{J,i,j-1/2}^{+,L,y}+\Gamma_{J,i,j+1/2}^{-,L,y})
 -\frac{(\betaJ)_{ij}}{\theta_{2,ij}}.
 \end{aligned}}
 \label{eq:GC-FV-low-subcell-form}
\end{equation}
By the GC-FV companion law, its distortion block induces
\begin{equation}
 \frac{\dd\rho_{ij}}{\dd t}
 =-\frac{\lambda_{K,x}}{\omega_i}
 (\overline F_{\rho,i+1/2,j}^{L,x}
  -\overline F_{\rho,i-1/2,j}^{L,x})
 -\frac{\lambda_{K,y}}{\omega_j}
 (\overline F_{\rho,i,j+1/2}^{L,y}
  -\overline F_{\rho,i,j-1/2}^{L,y}).
 \label{eq:GC-FV-low-subcell-density}
\end{equation}
We denote the spatial flux--fluctuation operators in
\eqref{eq:GC-DG-subcell-form} and \eqref{eq:GC-FV-low-subcell-form}, excluding
the point-local source, by $\mathcal L_K^H$ and $\mathcal L_K^L$,
respectively.  Thus the superscripts $H$ and $L$ consistently distinguish
the high- and low-order subcell operators and all of their interface
quantities.

\subsection{Global semidiscrete mass balance}
\label{sec:global-semidiscrete-mass}

The companion laws established above are local: GC-FV induces a density
balance in every finite-volume cell, while the high- and low-order subcell
operators induce density balances at every LGL nodal point.  Because all
three formulations use a single-valued density flux at shared interfaces,
their internal contributions cancel under global summation.  The following
theorem records the resulting exact semidiscrete mass balance, which is the
common global structure used by the blended method in
Section~\ref{sec:gpr-subcell}.  This statement is deliberately
semidiscrete: a finite Runge--Kutta update of $\A$ does not satisfy the
nonlinear determinant chain rule exactly.  The corresponding fully discrete
mass drift is defined and analyzed in
Remark~\ref{rem:fully-discrete-mass-drift}.

\begin{theorem}[Global semidiscrete mass balance]
\label{thm:global-semidiscrete-mass}
Assume the hypotheses of Theorems~\ref{thm:GC-FV} and
\ref{thm:GC-DG-compatibility}.  The GC-FV and GC-DG spatial operators satisfy,
respectively,
\begin{align}
 \frac{\dd\mathcal M_h^{\rm FV}}{\dd t}
 +\mathcal F_{\partial\Omega}^{\rm FV}&=0,
 &
 \mathcal M_h^{\rm FV}
 &:=\sum_{K\in\mathcal T_h}|K|\rho_K,
 &
 \mathcal F_{\partial\Omega}^{\rm FV}
 &:=\sum_{f\subset\partial\Omega}|f|\widehat F_{\rho,f},
 \label{eq:global-FV-mass}\\
 \frac{\dd\mathcal M_h^{\rm DG}}{\dd t}
 +\mathcal F_{\partial\Omega}^{\rm DG}&=0,
 &
 \mathcal M_h^{\rm DG}
 &:=\sum_KJ_K\sum_{i,j=0}^N\omega_i\omega_j\rho_{ij}^K,
 &
 \mathcal F_{\partial\Omega}^{\rm DG}
 &:=\sum_{f\subset\partial\Omega}\sum_{q=0}^N
 J_f\omega_q\widehat F_{\rho,fq}.
 \label{eq:global-DG-mass}
\end{align}
The low-order GC-FV operator on the LGL subcells satisfies
\eqref{eq:global-DG-mass} with the same quadrature mass and boundary flux.
Consequently, under periodic boundary conditions,
\begin{equation}
 \frac{\dd\mathcal M_h^{\rm FV}}{\dd t}
 =\frac{\dd\mathcal M_h^{\rm DG}}{\dd t}=0.
 \label{eq:global-periodic-mass}
\end{equation}
\end{theorem}

\begin{proof}
Summing the cellwise companion equation \eqref{eq:GC-FV-density} over the
GC-FV mesh cancels every internal single-valued density flux and gives
\eqref{eq:global-FV-mass}.  For GC-DG, multiply
\eqref{eq:GC-DG-density} by $J_K\omega_i\omega_j$ and sum over the nodal
points of $K$.  In the $x$ direction, symmetry of $F_\rho^0$ and the SBP
identity give
\begin{align}
 2\sum_{i,m=0}^NQ_{im}F_\rho^0(U_{ij},U_{mj};e_1)
 &=\sum_{i,m=0}^N(Q_{im}+Q_{mi})
   F_\rho^0(U_{ij},U_{mj};e_1)\nonumber\\
 &=f_\rho(U_{Nj};e_1)-f_\rho(U_{0j};e_1).
 \label{eq:global-DG-SBP-telescope}
\end{align}
The corresponding terms in the density SAT
\eqref{eq:GC-DG-density-SAT} replace these physical endpoint fluxes by
$\widehat F_\rho$.  Repeating the same calculation in the $y$ direction
therefore yields
\begin{equation}
 \frac{\dd}{\dd t}\left(
 J_K\sum_{i,j=0}^N\omega_i\omega_j\rho_{ij}^K\right)
 +\sum_{f\subset\partial K}J_f\sum_{q=0}^N
 \omega_q\widehat F_{\rho,fq}^{K,\rm out}=0.
 \label{eq:global-DG-element-mass}
\end{equation}
The low-order subcell result follows directly by telescoping
\eqref{eq:GC-FV-low-subcell-density}.  Finally, summation over the elements
cancels every internal face contribution and proves
\eqref{eq:global-DG-mass}; periodicity gives
\eqref{eq:global-periodic-mass}.
\end{proof}

\section{Geometric-compatible global subcell limiter}
\label{sec:gpr-subcell}

The unlimited GC-DG scheme is high-order accurate but may develop Gibbs oscillations
near shocks.  Standard postprocessing procedures, including TVD-type
limiters~\cite{CockburnShu2001} and oscillation-eliminating (OE)
filters~\cite{PengSunWu2025,YangFuOEESDG}, are designed to preserve the
cell average of each evolved variable.  Applied directly to the distortion,
they therefore enforce
\begin{equation}
 \overline{\A_h^{\rm lim}}{}_K=\overline{\A_h}{}_K,
 \qquad
 \overline{\det\A_h^{\rm lim}}{}_K
 \ne \overline{\det\A_h}{}_K
 \quad\hbox{in general}.
 \label{eq:SC-generic-limiter-obstruction}
\end{equation}
Here the overbar denotes the LGL cell average.  Notice that the second
quantity in \eqref{eq:SC-generic-limiter-obstruction} is the average of the
determinant, not the determinant of the average.  Since
$\rho_h=\rho_0\det\A_h$, preserving $\overline\A_h$ is consequently
insufficient to preserve the derived cell mass.  Although the underlying
semidiscrete GC-DG operator remains companion compatible, inserting such a
generic limiting map between Runge--Kutta stages generally breaks the fully
discrete companion structure.  The resulting limiter-induced mass drift is
not controlled by the temporal truncation estimate in
Remark~\ref{rem:fully-discrete-mass-drift}; consequently, the
$O(\Delta t^p)$ accumulated mass-drift estimate derived there no longer
applies.

We instead blend two spatial residuals that separately satisfy the same
discrete companion law: the high-order GC-DG operator and the robust GC-FV
operator on the LGL nodal subcells.  This follows the elementwise subcell
strategy of Hennemann et al.~\cite{Hennemann2021}, adapted here to the
nonconservative GPR system and its determinant companion law.

\subsection{Elementwise high--low blending}

Fix an element $K$ and its tensor-product LGL nodal points
$\{(\xi_i,\xi_j)\}_{i,j=0}^N$.  The blend is performed within $K$ by
combining, at every one of these nodal points, the high- and low-order
subcell residuals introduced at the end of
Section~\ref{sec:GC-FV-low-subcell}.  These operators have the same subcell
layout.  To make the blend precise, let $r\in\{H,L\}$ and write their three
spatial blocks, excluding the point-local relaxation sources, as
\begin{equation}
 \boxed{\begin{aligned}
 (\mathcal L_{K,Y}^{r}U)_{ij}
 &=-\frac{\lambda_{K,x}}{\omega_i}
  (\overline F_{Y,i+1/2,j}^{r,x}
  -\overline F_{Y,i-1/2,j}^{r,x})
  -\frac{\lambda_{K,y}}{\omega_j}
  (\overline F_{Y,i,j+1/2}^{r,y}
  -\overline F_{Y,i,j-1/2}^{r,y}),\\
 (\mathcal L_{K,A}^{r}U)_{ij}
 &=-\frac{\lambda_{K,x}}{\omega_i}
  (\widetilde\Gamma_{A,i-1/2,j}^{+,r,x}
  +\widetilde\Gamma_{A,i+1/2,j}^{-,r,x})
  -\frac{\lambda_{K,y}}{\omega_j}
  (\widetilde\Gamma_{A,i,j-1/2}^{+,r,y}
  +\widetilde\Gamma_{A,i,j+1/2}^{-,r,y}),\\
 (\mathcal L_{K,J}^{r}U)_{ij}
 &=-\frac{\lambda_{K,x}}{\omega_i}
  (\Gamma_{J,i-1/2,j}^{+,r,x}
  +\Gamma_{J,i+1/2,j}^{-,r,x})
  -\frac{\lambda_{K,y}}{\omega_j}
  (\Gamma_{J,i,j-1/2}^{+,r,y}
  +\Gamma_{J,i,j+1/2}^{-,r,y}).
 \end{aligned}}
 \label{eq:SC-high-low-operators}
\end{equation}
For $r=H$, the fluxes and directed fluctuations in
\eqref{eq:SC-high-low-operators} are those of
\eqref{eq:GC-DG-subcell-form}; for $r=L$, they are those of
\eqref{eq:GC-FV-low-subcell-form}.  Thus no new spatial operator is
introduced here.  Moreover, the two operators use the same complete GC-FV
face operator at $\partial K$; only their internal LGL-subcell quantities
differ.

Let $\alpha_K\in[0,1]$ be one coefficient common to every nodal point,
equation block, and internal subcell interface of $K$ during one residual
evaluation.  We define
\begin{equation}
 \boxed{\frac{\dd U_{ij}^K}{\dd t}
 =(1-\alpha_K)(\mathcal L_K^HU)_{ij}
  +\alpha_K(\mathcal L_K^LU)_{ij}+S_U(U_{ij}^K).}
 \label{eq:SC-blended-operator}
\end{equation}
The coefficient $\alpha_K$ is an oscillation indicator and is recomputed
at every Runge--Kutta stage.  In a smooth, well-resolved element,
$\alpha_K\to0$ and \eqref{eq:SC-blended-operator} recovers the accurate,
low-dissipation GC-DG operator.  Near an under-resolved discontinuity,
$\alpha_K\to1$ and the method approaches the more dissipative GC-FV
operator on the LGL subcells.  Intermediate values provide a continuous
transition between these limits.

The crucial point is that this construction blends spatial residuals rather
than postprocessing the distortion polynomial.  Since the high- and
low-order distortion operators each satisfy their own discrete density
companion law, cofactor contraction commutes with their scalar convex
combination.  We now show that the blend therefore preserves the companion
law.  More specifically, the induced density residual is blended between
its high- and low-order forms by the same elementwise coefficient
$\alpha_K$ used for the evolved variables.  Thus density undergoes the same
limiting action even though it is recovered from $\det\A$ rather than
evolved independently.

\begin{theorem}[Geometric-compatible companion law under global subcell blending]
\label{thm:SC-companion-blending}
Assume the hypotheses of Theorem~\ref{thm:GC-DG-compatibility}.  Let
$\alpha_K\in[0,1]$ be constant over the nodal subcells of $K$ during one
residual evaluation, and let the high- and low-order operators use the
same element-boundary GC-FV face operator.  Define
\begin{align}
 (\mathcal D_{\rho,K}^{r}U)_{ij}
 &:=-\frac{\lambda_{K,x}}{\omega_i}
 (\overline F_{\rho,i+1/2,j}^{r,x}
 -\overline F_{\rho,i-1/2,j}^{r,x})
 -\frac{\lambda_{K,y}}{\omega_j}
 (\overline F_{\rho,i,j+1/2}^{r,y}
 -\overline F_{\rho,i,j-1/2}^{r,y}),
 \qquad r\in\{H,L\},
 \label{eq:SC-density-residual-r}\\
 \overline F_{\rho}^{B,d}
 &:=(1-\alpha_K)\overline F_{\rho}^{H,d}
   +\alpha_K\overline F_{\rho}^{L,d}.
 \label{eq:SC-blended-density-flux}
\end{align}
Then the distortion block of the blended scheme
\eqref{eq:SC-blended-operator} induces, at every LGL node,
\begin{equation}
 \boxed{\begin{aligned}
 \frac{\dd\rho_{ij}^K}{\dd t}
 &=(1-\alpha_K)(\mathcal D_{\rho,K}^{H}U)_{ij}
   +\alpha_K(\mathcal D_{\rho,K}^{L}U)_{ij}\\
 &=-\frac{\lambda_{K,x}}{\omega_i}
 (\overline F_{\rho,i+1/2,j}^{B,x}
 -\overline F_{\rho,i-1/2,j}^{B,x})
 -\frac{\lambda_{K,y}}{\omega_j}
 (\overline F_{\rho,i,j+1/2}^{B,y}
 -\overline F_{\rho,i,j-1/2}^{B,y}).
 \end{aligned}}
 \label{eq:SC-blended-nodal-density}
\end{equation}
Thus the global blend satisfies the semidiscrete mass balance
\eqref{eq:global-DG-mass}; in particular,
$\dd\mathcal M_h^{\rm DG}/\dd t=0$ for periodic boundary conditions.
It also preserves the conservative momentum and total-energy balances,
preserves constant states in the absence of sources, and recovers the
unlimited GC-DG operator wherever $\alpha_K=0$.
\end{theorem}

\begin{proof}
Equations \eqref{eq:GC-DG-nodal-mass} and
\eqref{eq:GC-FV-low-subcell-density} give
\begin{equation}
 \rho_0\w_{ij}:(\mathcal L_{K,A}^{r}U)_{ij}
 =(\mathcal D_{\rho,K}^{r}U)_{ij},
 \qquad r\in\{H,L\}.
 \label{eq:SC-endpoint-companion-identities}
\end{equation}
Contracting the distortion block of \eqref{eq:SC-blended-operator} with
$\rho_0\w_{ij}$, using the determinant chain rule and
$\w_{ij}:S_{A,ij}=0$, proves
\eqref{eq:SC-blended-nodal-density}.  Because the high- and low-order
boundary fluxes are both $\widehat F_\rho$, the blended boundary flux is
also $\widehat F_\rho$.  The global conclusion then follows directly from
the telescoping argument in
Theorem~\ref{thm:global-semidiscrete-mass}; no additional mass-balance
proof is required.
\end{proof}

\subsection{Complementary modal and face-jump indicators}

The blending coefficient must distinguish an unresolved physical wave from
a smooth, well-resolved polynomial.  We combine an element-interior modal
indicator with a face-jump indicator.  Their construction is given below
in the order in which it is evaluated.

\subsubsection{Elementwise modal indicator}

We first adapt the Persson--Peraire modal sensor~\cite{PerssonPeraire2006}
and its smooth blending map used by Hennemann et
al.~\cite[Sec.~4]{Hennemann2021}.  On the reference element, expand every
sensor group $Z$ in the tensor-product Legendre basis,
\begin{equation}
 Z_h(\xi_1,\xi_2)=\sum_{r,s=0}^N\widehat Z_{rs}
 P_r(\xi_1)P_s(\xi_2),
 \qquad
 e_{rs}^Z:=\gamma_r\gamma_s\|\widehat Z_{rs}\|_Z^2,
 \qquad \gamma_r:=\frac{2}{2r+1}.
 \label{eq:SC-modal-expansion}
\end{equation}
Here $\|\cdot\|_Z$ is the absolute value for the scalar group $Z=\rho p$,
the Frobenius norm for the embedded $3\times3$ distortion tensor $Z=\A$, and the
Euclidean norm for the thermal-impulse group $Z=\Jv$.  Thus all components
of a vector or tensor share one modal energy and are not assigned unrelated
thresholds.  Hennemann et al. use $\rho p$ alone for the Euler equations:
it reacts to both density and pressure jumps without separately sensing
velocity.  We retain this thermodynamic group and add $\A$ and $\Jv$ to
detect elastic and thermal waves.  A separate modal velocity group is not
used because it can strongly activate on resolved contacts, shear layers,
and vortical motion; interelement velocity information is nevertheless
present through the momentum group in the face-jump indicator below.

For each modal group, introduce its element magnitude and a scale-aware
modal-energy floor,
\begin{equation}
 \mathcal A_K^Z:=\max_{0\le i,j\le N}\|Z_{ij}^K\|_Z,
 \qquad
 \varepsilon_{{\rm mod},K}^Z
 :=\left[\max\left\{100\epsilon_{\rm mach}s_Z,
              \varepsilon_{\rm rel}\mathcal A_K^Z\right\}\right]^2,
 \qquad \varepsilon_{\rm rel}=10^{-6}.
 \label{eq:SC-modal-roundoff-floor}
\end{equation}
Here $s_Z>0$ is a fixed guard scale for the group $Z$; we use $s_Z=1$ for
all groups in the computations below.  The first term
protects an identically zero group, whereas the second prevents
roundoff-sized high modes of a large-magnitude, nearly constant field from
producing an $O(1)$ modal ratio.  Both are fixed numerical safeguards, not
shock-capturing parameters.

For $N\ge3$, we compare the energy in the highest and penultimate
directional modal bands with the corresponding total energy:
\begin{align}
 E_{Z,K}^{x}
 &=\max\left\{
 \frac{\sum_{s=0}^Ne_{Ns}^Z}
      {\max\{\sum_{r,s=0}^Ne_{rs}^Z,
              \varepsilon_{{\rm mod},K}^Z\}},
 \frac{\sum_{s=0}^Ne_{N-1,s}^Z}
      {\max\{\sum_{r=0}^{N-1}\sum_{s=0}^Ne_{rs}^Z,
              \varepsilon_{{\rm mod},K}^Z\}}
 \right\},\\
 E_{Z,K}^{y}
 &=\max\left\{
 \frac{\sum_{r=0}^Ne_{rN}^Z}
      {\max\{\sum_{r,s=0}^Ne_{rs}^Z,
              \varepsilon_{{\rm mod},K}^Z\}},
 \frac{\sum_{r=0}^Ne_{r,N-1}^Z}
      {\max\{\sum_{s=0}^{N-1}\sum_{r=0}^Ne_{rs}^Z,
              \varepsilon_{{\rm mod},K}^Z\}}
 \right\}.
 \label{eq:SC-modal-directional-energy}
\end{align}
For $N=2$, only the first ratio in each direction is retained; the
degree-dependent low-order cases are summarized in
Remark~\ref{rem:SC-degree-dependent-indicator}.  We then set
\begin{equation}
 E_K:=\max_{Z\in\{\rho p,\A,\Jv\}}
       \max\{E_{Z,K}^x,E_{Z,K}^y\},
 \qquad
 \Theta_N^{\rm mod}:=0.5\,10^{-1.8(N+1)^{1/4}}.
 \label{eq:SC-modal-sensor}
\end{equation}
If the heat-wave subsystem is inactive, the $\Jv$ group is omitted.  The
increasing logistic map of Hennemann et al. is
\begin{equation}
 \Lambda(r)=\frac{1}{1+\exp[-(\log9999)(r-1)]},
 \label{eq:SC-logistic-alpha}
\end{equation}
which satisfies $\Lambda(1)=1/2$ and changes rapidly from zero to one around
the prescribed threshold.  To remove numerically insignificant endpoint
values, define
\begin{equation}
 \mathcal C_{\rm clip}(a)=
 \begin{cases}
  0,&a<\epsilon_{\rm clip},\\
  a,&\epsilon_{\rm clip}\le a\le1-\epsilon_{\rm clip},\\
  1,&a>1-\epsilon_{\rm clip},
 \end{cases}
 \qquad
 \Lambda_{\rm clip}:=\mathcal C_{\rm clip}\circ\Lambda,
 \qquad \epsilon_{\rm clip}=10^{-3}.
 \label{eq:SC-indicator-clipping}
\end{equation}
The implemented modal activation is therefore
$\alpha_{M,K}=\Lambda_{\rm clip}(E_K/\Theta_N^{\rm mod})$.

\subsubsection{Face-jump indicator}

The modal indicator measures the loss of resolution inside one element,
but it can under-detect oscillations driven by a strong mismatch between
neighboring elements.  We therefore add a face-jump channel inspired by
the OE indicators of~\cite{PengSunWu2025,YangFuOEESDG}.  This channel uses
the groups
\begin{equation}
 \mathcal Z_F=\{\m,\mathcal E,\det\A,
                  \dev(\A^T\A),\Jv\},
 \label{eq:SC-face-groups}
\end{equation}
with the last group omitted when heat waves are inactive.  As defined in
\eqref{eq:GPR-shear-invariants}, $G=\A^T\A$ is the distortion metric and
\begin{equation}
 \dev(\A^T\A)
 =\A^T\A-\tfrac13\tr(\A^T\A)I_3
 =D
 \label{eq:SC-deviatoric-distortion-sensor}
\end{equation}
is its trace-free, or distortional, part.  This is the same tensor that
enters the GPR shear energy and shear stress through
\eqref{eq:GPR-shear-invariants}--\eqref{eq:J-thermal-closure}.  It is
therefore a direct indicator of elastic and shear waves.  At the same time,
$\det\A$ senses the complementary volumetric deformation and the induced
density.  Using these two groups separates volumetric and distortional
changes and avoids relying on individual entries of $\A$, whose jumps need
not directly measure the constitutively relevant shear deformation.

Euclidean and Frobenius group norms ensure that the components of $\m$,
$\dev(\A^T\A)$, and $\Jv$ are detected collectively.  For every
$Z\in\mathcal Z_F$, the modal energies $e_{rs}^Z$ used below are defined by
the same expansion \eqref{eq:SC-modal-expansion}, using the absolute value
for the scalar groups $\mathcal E$ and $\det\A$, the Euclidean norm for
$\m$ and $\Jv$, and the Frobenius norm for $\dev(\A^T\A)$.
The corresponding $\mathcal A_K^Z$ and
$\varepsilon_{{\rm mod},K}^Z$ are defined by
\eqref{eq:SC-modal-roundoff-floor} with the same group norm.

For $d\in\{x,y\}$, let $\mathcal P_{d,K}$ consist of $K$ and its available
$d$-direction neighbors.  Introduce the unit-interval reference coordinate
$\widehat\xi_d=(\xi_d+1)/2$ and normalized LGL weights
$\widehat\omega_q=\omega_q/2$, so that
$\sum_{q=0}^N\widehat\omega_q=1$.  The LGL-weighted patch mean is
\begin{equation}
 \langle Z_h\rangle_{\mathcal P_{d,K}}
 :=\frac{1}{|\mathcal P_{d,K}|}
 \sum_{K'\in\mathcal P_{d,K}}\sum_{i,j=0}^N
 \widehat\omega_i\widehat\omega_j Z_{ij}^{K'}.
 \label{eq:SC-face-patch-mean}
\end{equation}
We normalize every jump by the nodal patch amplitude
\begin{equation}
 \mathcal A_{d,K}^Z
 :=\max_{K'\in\mathcal P_{d,K}}\max_{0\le i,j\le N}
       \|Z_{ij}^{K'}\|_Z,
 \qquad
 \varepsilon_{d,K}^Z
 :=\max\left\{100\epsilon_{\rm mach}s_Z,
               \varepsilon_{\rm rel}\mathcal A_{d,K}^Z\right\},
 \label{eq:SC-face-magnitude-floor}
\end{equation}
and define
\begin{equation}
 \Delta_{d,K}^Z
 :=\max\left\{
 \max_{K'\in\mathcal P_{d,K}}\max_{0\le i,j\le N}
 \left\|Z_{ij}^{K'}-\langle Z_h\rangle_{\mathcal P_{d,K}}\right\|_Z,
 \varepsilon_{d,K}^Z\right\}.
 \label{eq:SC-face-patch-scale}
\end{equation}
The magnitude-relative term in \eqref{eq:SC-face-magnitude-floor} is
essential when a nearly constant group has a large background value: its
interface jump and patch variation can otherwise both be roundoff-sized, so
their quotient is spuriously $O(1)$.  For a resolved physical variation,
the patch amplitude in \eqref{eq:SC-face-patch-scale} remains the active
normalization.
For derivative order $m\in\{0,N\}$, define the normalized jump across the
interior $d$-normal faces of $K$ by
\begin{equation}
 \sigma_{m,d,K}^Z
 :=\frac{2m+1}{2(2N-1)m!}
 \frac{
 \left(\displaystyle\sum_{f\in\partial_d^{\rm int}K}\sum_{q=0}^N
 \widehat\omega_q
 \left\|\jump{\partial_{\widehat\xi_d}^mZ_h}_{f,q}\right\|_Z^2
 \right)^{1/2}}
 {\Delta_{d,K}^Z}.
 \label{eq:SC-normalized-face-jump}
\end{equation}
Here $\partial_d^{\rm int}K$ contains the interior faces of $K$ normal to
direction $d$; physical boundary faces are not included in the sensor.  In
the implementation, $\sigma_{m,d,K}^Z$ is set to zero whenever its unscaled
jump norm does not exceed the scale-aware floor $\varepsilon_{d,K}^Z$.
The $m=0$ quantity detects a jump of the state itself.  The $m=N$
quantity is a high-order OE channel that detects a mismatch in the
highest resolved derivative.  The two resulting activations are
\begin{align}
 I_{0,K}&:=\sum_{d\in\{x,y\}}\max_{Z\in\mathcal Z_F}
 \sigma_{0,d,K}^Z,
 &\alpha_{0,K}&:=\Lambda_{\rm clip}(I_{0,K}/C_{\rm OE}),
 \label{eq:SC-zero-face-indicator}\\
 \alpha_{N,d,K}&:=\max_{Z\in\mathcal Z_F}
 \left\{\vartheta_{N,K}^Z
 \Lambda_{\rm clip}(\sigma_{N,d,K}^Z/C_{\rm OE})\right\},
 &\alpha_{N,K}&:=1-\prod_{d\in\{x,y\}}(1-\alpha_{N,d,K}).
 \label{eq:SC-high-face-indicator}
\end{align}
Here $C_{\rm OE}>0$ is the face-sensor threshold and
\begin{equation}
 \vartheta_{N,K}^Z
 :=\begin{cases}
 \displaystyle
 \frac{\sum_{\max(r,s)=N}e_{rs}^Z}
      {\sum_{(r,s)\ne(0,0)}e_{rs}^Z},
 &\displaystyle \sum_{(r,s)\ne(0,0)}e_{rs}^Z>
                 \varepsilon_{{\rm mod},K}^Z,\\[8pt]
 0,&\text{otherwise},
 \end{cases}
 \qquad \vartheta_{N,K}^Z\in[0,1],
 \label{eq:SC-high-shell-weight}
\end{equation}
is the fraction of nonconstant modal energy in the highest tensor shell.
This factor prevents a high-derivative face jump from activating strong
dissipation when the corresponding highest-order content in $K$ is
negligible.  Finally, the two face channels are combined by
\begin{equation}
 \alpha_{F,K}:=1-(1-\alpha_{0,K})(1-\alpha_{N,K}).
 \label{eq:SC-face-union}
\end{equation}

\subsubsection{Final indicator}

The modal and face indicators contain complementary information.  Our
numerical experiments show that the modal indicator suppresses oscillations
effectively but can introduce excessive dissipation and erase fine
structures.  The face-jump indicator resolves some of these structures more
sharply, but can leave visible numerical oscillations near strong isolated
shocks.  For $N\ge3$, we therefore combine the two channels to balance
oscillation control and fine-structure resolution:
\begin{equation}
 \boxed{\alpha_K:=1-(1-\alpha_{F,K})(1-\alpha_{M,K})}.
 \label{eq:SC-alpha-overview}
\end{equation}
The reduced channel sets used for $N=1$ and $N=2$ are stated explicitly in
Remark~\ref{rem:SC-degree-dependent-indicator}.
The union lies in $[0,1]$ and is not smaller than either constituent
indicator.  Each logistic channel has already been mapped by
$\Lambda_{\rm clip}$; no additional clipping is applied after the unions.
Unlike a modal filter, the resulting coefficient only selects
the convex spatial-residual blend in \eqref{eq:SC-blended-operator}; it
does not postprocess the accepted polynomial and hence does not introduce
the incompatibility in \eqref{eq:SC-generic-limiter-obstruction}.

\begin{remark}[Polynomial-degree-dependent indicator channels]
\label{rem:SC-degree-dependent-indicator}
The indicator is specialized at the two lowest polynomial degrees.  To
state the three cases precisely, define the highest-band modal ratio
\begin{equation}
 E_K^{\rm top}:=\max_{Z\in\{\rho p,\A,\Jv\}}
 \max\left\{
 \frac{\sum_{s=0}^Ne_{Ns}^Z}
      {\max\{\sum_{r,s=0}^Ne_{rs}^Z,
              \varepsilon_{{\rm mod},K}^Z\}},
 \frac{\sum_{r=0}^Ne_{rN}^Z}
      {\max\{\sum_{r,s=0}^Ne_{rs}^Z,
              \varepsilon_{{\rm mod},K}^Z\}}
 \right\},
 \qquad
 \alpha_{M,K}^{\rm top}
 :=\Lambda_{\rm clip}(E_K^{\rm top}/\Theta_N^{\rm mod}),
 \label{eq:SC-top-modal-channel}
\end{equation}
where the $\Jv$ group is omitted when the thermal-impulse equation is
inactive.  With
$\alpha_{M,K}=\Lambda_{\rm clip}(E_K/\Theta_N^{\rm mod})$ denoting the full
modal channel defined after \eqref{eq:SC-indicator-clipping}, the implemented
indicator is
\begin{equation}
 \alpha_K=
 \begin{cases}
  \alpha_{0,K}, & N=1,\\[2pt]
  1-(1-\alpha_{0,K})(1-\alpha_{N,K})
    (1-\alpha_{M,K}^{\rm top}), & N=2,\\[2pt]
  1-(1-\alpha_{0,K})(1-\alpha_{N,K})
    (1-\alpha_{M,K}), & N\ge3,
 \end{cases}
 \label{eq:SC-degree-dependent-indicator}
\end{equation}
Thus, for $N=1$, the element-interior modal channel and the
highest-derivative face channel are disabled because the only nonconstant
mode is the resolved linear slope.  For $N=2$, the highest quadratic modal
band is retained, whereas the penultimate band is omitted because it is
again linear.  The complete highest/penultimate modal construction is used
only for $N\ge3$.  This prevents resolved low-order slopes from being
misidentified as under-resolved oscillations.
\end{remark}

\begin{remark}[The only tunable limiter parameter]
\label{rem:SC-only-tunable-parameter}
The threshold $C_{\rm OE}$ is the only user-tuned parameter in the present
global subcell limiter; the modal threshold, logistic map, roundoff floors,
and clipping tolerance are fixed by
\eqref{eq:SC-modal-roundoff-floor}--\eqref{eq:SC-indicator-clipping}.
Its appropriate value is problem-dependent and balances shock control against
the resolution of smooth and small-scale structures.  Larger values reduce
limiter activation, whereas smaller values move more elements toward the
dissipative GC-FV operator.
\end{remark}

\subsection{Local subcell blending as a lower-dissipation extension}
\label{sec:SC-local-extension}

The equivalent subcell forms derived in
Sections~\ref{sec:GC-DG-subcell-equivalent} and~\ref{sec:GC-FV-low-subcell} expose
the high- and low-order operators at every LGL subcell interface.  They
therefore also provide the algebraic foundation for a local subcell
limiter, without changing either the GC-DG or GC-FV building block.
Instead of the single elementwise coefficient $\alpha_K$, such methods
assign $\alpha_{k+1/2}^d\in[0,1]$ to each internal subcell interface and
direction~\cite{RuedaRamirezPaznerGassner2022,RuedaRamirezGassner2024}.
For the distortion block, the corresponding blend is simply
\begin{equation}
 \widetilde\Gamma_{A,k+1/2}^{\pm,B,d}
 =(1-\alpha_{k+1/2}^d)
   \widetilde\Gamma_{A,k+1/2}^{\pm,H,d}
   +\alpha_{k+1/2}^d
   \widetilde\Gamma_{A,k+1/2}^{\pm,L,d}.
 \label{eq:SC-local-A-blend}
\end{equation}
The same $\alpha_{k+1/2}^d$ is used for the two directed fluctuations on
the two sides of the interface; the conservative blocks and the $\Jv$
block are blended analogously.  Since both the high- and low-order
distortion operators satisfy the
same cofactor-contraction identities, linearity gives the single-valued
induced density flux
\begin{equation}
 \overline F_{\rho,k+1/2}^{B,d}
 =(1-\alpha_{k+1/2}^d)\overline F_{\rho,k+1/2}^{H,d}
  +\alpha_{k+1/2}^d\overline F_{\rho,k+1/2}^{L,d}.
 \label{eq:SC-local-density-blend}
\end{equation}
Thus the nodal determinant companion law and the semidiscrete global mass
balance remain exact.  When the local blend is evaluated at every
Runge--Kutta stage, it also retains the same stagewise companion structure
as the global blend: the limiter introduces no additional mass drift, and
the remaining drift is solely the temporal determinant error
described in Remark~\ref{rem:fully-discrete-mass-drift}.

What cannot be transferred directly is the indicator.  The elementwise
$\alpha_K$ constructed above does not determine how much dissipation should
be assigned to each direction and each subcell interface.  A local method
requires a new interfacewise indicator that remains shared by the two
directed fluctuations, localizes shocks without erasing nearby fine
structure, and is economical enough to evaluate at every interface and RK
stage.  Designing and validating such an indicator is outside the scope of
this work.  We therefore use the global elementwise limiter in all
computations and leave the local extension for future investigation.

\section{Time integration}
\label{sec:time}

The computations use the third-order strong-stability-preserving
Runge--Kutta method
\begin{align}
 U^{(1)}&=U^n+\Delta t\,L(U^n),\\
 U^{(2)}&=\tfrac34U^n+\tfrac14\bigl(U^{(1)}+\Delta t\,L(U^{(1)})\bigr),\\
 U^{n+1}&=\tfrac13U^n+\tfrac23\bigl(U^{(2)}+\Delta t\,L(U^{(2)})\bigr),
 \label{eq:ssprk3}
\end{align}
where $L$ contains the selected spatial operator and the point-local
relaxation source.  The determinant companion laws established above are
semidiscrete; their global consequence was established in
Theorem~\ref{thm:global-semidiscrete-mass}.  We now quantify the additional
mass drift introduced by time discretization.

\begin{remark}[Fully discrete mass drift]
\label{rem:fully-discrete-mass-drift}
Theorem~\ref{thm:global-semidiscrete-mass} gives an exact semidiscrete mass
balance for GC-FV and GC-DG, and
Theorem~\ref{thm:SC-companion-blending} extends it to the globally blended
operator.  This identity is not exact after Runge--Kutta time integration,
because a finite update of $\A$ does not satisfy the nonlinear determinant
chain rule.  We therefore measure the discrepancy between the fully discrete
mass change and the boundary flux used by the same Runge--Kutta step.  We call
this discrepancy the \emph{mass drift}; it is the mass-conservation defect of
the fully discrete method relative to the exact semidiscrete companion
balance.

Set $\mathcal M_h^{{\rm DG},n}:=\mathcal M_h^{\rm DG}(U^n)$, with
$\mathcal M_h^{\rm DG}$ defined in \eqref{eq:global-DG-mass}, and denote the
three SSP--RK3 stage fluxes by
\begin{equation}
 \mathcal F_{\partial\Omega}^{n,0}
 :=\mathcal F_{\partial\Omega}^{\rm DG}(U^n),\qquad
 \mathcal F_{\partial\Omega}^{n,1}
 :=\mathcal F_{\partial\Omega}^{\rm DG}(U^{(1)}),\qquad
 \mathcal F_{\partial\Omega}^{n,2}
 :=\mathcal F_{\partial\Omega}^{\rm DG}(U^{(2)}).
 \label{eq:ssprk3-stage-boundary-fluxes}
\end{equation}
The accumulated discrete outward mass flux is updated by
\begin{equation}
 \mathcal B_{\partial\Omega}^{0}=0,\qquad
 \mathcal B_{\partial\Omega}^{n+1}
 =\mathcal B_{\partial\Omega}^{n}
 +\Delta t_n\left(
 \frac16\mathcal F_{\partial\Omega}^{n,0}
 +\frac16\mathcal F_{\partial\Omega}^{n,1}
 +\frac23\mathcal F_{\partial\Omega}^{n,2}\right).
 \label{eq:ssprk3-accumulated-boundary-flux}
\end{equation}
The normalized fully discrete mass drift is then
\begin{equation}
 \varepsilon_M^n:=
 \frac{\left|\mathcal M_h^{{\rm DG},n}
 -\mathcal M_h^{{\rm DG},0}
 +\mathcal B_{\partial\Omega}^{n}\right|}
 {\mathcal M_h^{{\rm DG},0}}.
 \label{eq:fully-discrete-normalized-mass-drift}
\end{equation}
Thus $\varepsilon_M^n=0$ is precisely the fully discrete counterpart of the
global balance \eqref{eq:global-DG-mass}; a nonzero value measures the
mass-conservation defect introduced by time integration.  The same definition
applies to GC-FV after replacing the DG quantities by those in
\eqref{eq:global-FV-mass}.

For a Runge--Kutta method of order $p$ and a sufficiently smooth
semidiscrete solution, the one-step discrepancy satisfies
\begin{equation}
 \mathcal M_h^{{\rm DG},n+1}-\mathcal M_h^{{\rm DG},n}
 +\mathcal B_{\partial\Omega}^{n+1}
 -\mathcal B_{\partial\Omega}^{n}
 =O(\Delta t_n^{p+1}).
 \label{eq:fully-discrete-local-mass-drift}
\end{equation}
Consequently, the numerator of
\eqref{eq:fully-discrete-normalized-mass-drift} accumulates as
$O(\Delta t^p)$ over a fixed time interval.  Under periodic boundary
conditions, $\mathcal B_{\partial\Omega}^{n}=0$, and hence
\begin{equation}
 \varepsilon_M^n
 =\frac{\left|\mathcal M_h^{{\rm DG},n}
 -\mathcal M_h^{{\rm DG},0}\right|}
 {\mathcal M_h^{{\rm DG},0}}
 =O(\Delta t^p).
 \label{eq:fully-discrete-mass-drift}
\end{equation}
For SSP--RK3 the one-step discrepancy is $O(\Delta t^4)$, and the accumulated
mass drift is $O(\Delta t^3)$.  This estimate requires a sufficiently smooth
semidiscrete trajectory with uniformly bounded time derivatives.  It does not
apply during the initial transient generated when a nodally sampled
discontinuous initial condition produces $O(1)$ jumps between neighboring
nodes.  Any larger drift created in those first steps is a fully discrete
determinant-chain-rule defect, not an additional defect of the spatial
subcell blend, which satisfies the semidiscrete companion identity at every
stage.
\end{remark}
\section{Numerical experiments}
\label{sec:numerics}

The experiments assess smooth-solution accuracy, resolution of fluid,
elastic, and thermal-wave structures, discrete mass behavior relative to the
semidiscrete companion law, and robustness in two dimensions.  Unless stated
otherwise, all tests use SSP--RK3 and the global subcell blend
\eqref{eq:SC-blended-operator} with the combined modal/face-jump indicator of
Section~\ref{sec:gpr-subcell}, including its degree-dependent form
\eqref{eq:SC-degree-dependent-indicator}.  Density is always recovered as
$\rho_h=\rho_0\det\A_h$; it is not independently evolved.
For every test, the prescribed initial condition is evaluated directly at
the LGL nodal points and used as the corresponding nodal degrees of freedom.
All experiments are chosen so that the relaxation sources are not severely
stiff on the simulated time scale.  Much stiffer relaxation regimes arise in
other GPR computations~\cite{Dumbser2016,JacksonNikiforakis2019,
PeshkovEtAl2019}, but their treatment generally requires time integration
designed specifically for stiff sources.  Because the present study uses
explicit SSP--RK3 throughout, such time discretizations and the associated
strong-relaxation limit are beyond its scope.
Whenever the limiter is active in the experiments below, we use
$C_{\rm OE}=0.1$ for $N=2$ and $C_{\rm OE}=0.02$ for $N\ge3$.  These are the
fixed choices for the present test suite rather than universal recommended
values; as emphasized in Remark~\ref{rem:SC-only-tunable-parameter}, the
appropriate threshold is problem-dependent.  All other indicator constants
are kept fixed at the values stated in Section~\ref{sec:gpr-subcell}.

\subsection{Translating isentropic vortex}
\label{sec:isentropic-vortex}

The translating isentropic vortex is a smooth benchmark for measuring
spatial convergence without a manufactured forcing term.  Variants of this
test are standard in high-order Euler and GPR calculations
\cite{Dumbser2016,ChenShu2017}.  We use it to assess the high-order accuracy of
the unlimited GC-DG scheme and, in a separate temporal-refinement study, the
fully discrete mass-drift estimate in
Remark~\ref{rem:fully-discrete-mass-drift}.  The shear and thermal-impulse
sectors are inactive, so the Euler vortex supplies exact density, momentum,
and energy fields for the reduced GPR system.

The computational domain is $\Omega=[-10,10]^2$, partitioned into a uniform
$M\times M$ Cartesian mesh and equipped with periodic boundary conditions in
both directions.  The enlarged box reduces the exponentially small mismatch
of the standard vortex profile at the periodic cuts.  We take
$\rho_0=R_{\rm gas}=1$, $c_v=2.5$,
$\gamma=1+R_{\rm gas}/c_v=1.4$, $p_\infty=0$, vortex strength
$\beta_{\rm v}=5$, and background velocity $\mathbf V=(1,1)^T$.  The vortex
center is initially at the origin.  With
$\boldsymbol\xi=(x,y)^T-\mathbf Vt$, $r^2=|\boldsymbol\xi|^2$, define
\begin{equation}
 \omega(r)=\frac{\beta_{\rm v}}{2\pi}e^{(1-r^2)/2},\qquad
 \mathcal T_{\rm v}(r)=1-\frac{(\gamma-1)\beta_{\rm v}^2}{8\gamma\pi^2}e^{1-r^2}.
 \label{eq:vortex-omega-theta}
\end{equation}
The exact Euler fields are
\begin{equation}
 \rho_e=\mathcal T_{\rm v}^{1/(\gamma-1)},\qquad
 p_e=\mathcal T_{\rm v}^{\gamma/(\gamma-1)},\qquad
 \mathbf v_e=\mathbf V+\omega(-\xi_2,\xi_1)^T,
 \label{eq:vortex-euler-fields}
\end{equation}
with $\m_e=\rho_e\mathbf v_e$ and
$\mathcal E_e=p_e/(\gamma-1)+\rho_e|\mathbf v_e|^2/2$.  The initial
distortion is
\begin{equation}
 \A(x,y,0)=\left(\frac{\rho_e(x,y,0)}{\rho_0}\right)^{1/3}I_3,
 \label{eq:vortex-initial-distortion}
\end{equation}
so that $\rho_0\det\A=\rho_e$ initially.  We set $c_s=0$ and
$S_A=0$, and omit the decoupled thermal-impulse equation.  In evaluating the
periodic exact solution, each component of $\boldsymbol\xi$ is replaced by
its shortest periodic displacement.

The evolved variables are initialized pointwise at the LGL nodes and advanced
by SSP--RK3 to $t_f=0.5$ with CFL number $0.10$.  The subcell limiter is
disabled, $\alpha_K=0$, because this is a smooth accuracy test.  Density is
not independently evolved; throughout the computation it is recovered as
$\rho_h=\rho_0\det\A_h$.  For
$a\in\{\rho,v_1,v_2\}$, the discrete error and the successive refinement
rate are
\begin{equation}
 E_{a,2}(M,N)=\left[\sum_K\sum_{i,j=0}^{N}
 |J_K|\omega_i\omega_j\bigl(a_h(\boldsymbol x_{ij},t_f)-a_e(\boldsymbol x_{ij},t_f)\bigr)^2\right]^{1/2},
 \qquad
 r_a=\log_2\frac{E_{a,2}(M,N)}{E_{a,2}(2M,N)},
 \label{eq:vortex-error-order}
\end{equation}
where $J_K$ is the Cartesian reference-to-physical Jacobian, $\omega_i$ are the
one-dimensional LGL weights, and $\boldsymbol x_{ij}$ are the element nodes.  The reported
velocity error is
$E_{\mathbf v,2}=(E_{v_1,2}^2+E_{v_2,2}^2)^{1/2}$, with the analogous rate
$r_{\mathbf v}$.

\begin{table}[H]
\centering
\footnotesize
\setlength{\tabcolsep}{7pt}
\begin{tabular}{|c|r|rr|rr|}
\hline
$N$ & $M$ & $E_{\rho,2}$ & $r_\rho$ & $E_{\mathbf v,2}$ & $r_{\mathbf v}$\\
\hline
\multirow{4}{*}{$1$}
 & 16  & $2.9\times10^{-1}$ & --   & $8.2\times10^{-1}$ & --\\
 & 32  & $1.1\times10^{-1}$ & 1.36 & $3.0\times10^{-1}$ & 1.47\\
 & 64  & $3.4\times10^{-2}$ & 1.73 & $9.0\times10^{-2}$ & 1.73\\
 & 128 & $9.1\times10^{-3}$ & 1.90 & $2.4\times10^{-2}$ & 1.88\\
\hline
\multirow{4}{*}{$2$}
 & 16  & $7.7\times10^{-2}$ & --   & $2.3\times10^{-1}$ & --\\
 & 32  & $1.6\times10^{-2}$ & 2.30 & $4.3\times10^{-2}$ & 2.38\\
 & 64  & $3.4\times10^{-3}$ & 2.20 & $7.7\times10^{-3}$ & 2.48\\
 & 128 & $7.3\times10^{-4}$ & 2.23 & $1.4\times10^{-3}$ & 2.42\\
\hline
\multirow{4}{*}{$3$}
 & 16  & $1.8\times10^{-2}$ & --   & $6.3\times10^{-2}$ & --\\
 & 32  & $2.6\times10^{-3}$ & 2.77 & $7.1\times10^{-3}$ & 3.13\\
 & 64  & $2.3\times10^{-4}$ & 3.53 & $5.9\times10^{-4}$ & 3.60\\
 & 128 & $1.5\times10^{-5}$ & 3.88 & $3.1\times10^{-5}$ & 4.26\\
\hline
\multirow{4}{*}{$4$}
 & 16  & $5.2\times10^{-3}$ & --   & $1.3\times10^{-2}$ & --\\
 & 32  & $2.6\times10^{-4}$ & 4.35 & $6.1\times10^{-4}$ & 4.46\\
 & 64  & $1.2\times10^{-5}$ & 4.38 & $2.5\times10^{-5}$ & 4.58\\
 & 128 & $6.9\times10^{-7}$ & 4.16 & $1.3\times10^{-6}$ & 4.29\\
\hline
\multirow{4}{*}{$5$}
 & 16  & $1.2\times10^{-3}$ & --   & $3.3\times10^{-3}$ & --\\
 & 32  & $3.8\times10^{-5}$ & 5.00 & $1.1\times10^{-4}$ & 4.95\\
 & 64  & $1.1\times10^{-6}$ & 5.10 & $2.1\times10^{-6}$ & 5.65\\
 & 128 & $1.9\times10^{-8}$ & 5.86 & $3.1\times10^{-8}$ & 6.09\\
\hline
\end{tabular}
\caption{\textbf{Translating isentropic vortex:} density and velocity-vector
$L^2$ errors of the unlimited GC-DG scheme at $t_f=0.5$.  Here $M$ is the
number of elements in each direction and $N$ is the polynomial degree.
Errors are displayed to two significant figures; rates are evaluated before
rounding.}
\label{tab:vortex-convergence}
\end{table}

Both density and velocity errors decrease systematically under refinement.
For $N=1,3,5$, the finer meshes approach the nominal order $N+1$.  For
$N=2,4$, the observed rates approach $N+1$ less clearly over the tested
mesh sequence.  This pattern resembles, and may be related to, odd--even
convergence effects reported for DGSEM discretizations of the Euler
equations with local LF interface dissipation
\cite{HindenlangGassner2020}; the present data alone do not establish such
a mechanism.  Nevertheless, irrespective of the detailed odd--even
mechanism, this experiment demonstrates the high-order convergence of GC-DG
and provides numerical evidence that the projected two-state construction and
its SBP lift retain the intended accuracy.

To test the fully discrete estimate in
Remark~\ref{rem:fully-discrete-mass-drift}, we use the same periodic vortex
but hold the spatial discretization fixed at $(N,M)=(3,8)$ and prescribe the
time steps
\begin{equation}
 \Delta t_\ell=2^{-\ell}(2.0\times10^{-2}),\qquad
 \ell=0,1,2,3,
 \label{eq:vortex-mass-time-steps}
\end{equation}
up to $t_f=0.5$, corresponding to $25$, $50$, $100$, and $200$ SSP--RK3
steps.  The coarse fixed mesh deliberately keeps the temporal mass drift
above the roundoff floor; spatial convergence is irrelevant for this test
because every member of the semidiscrete family already satisfies the same
global companion law.  Since the boundary is periodic,
$\mathcal B_{\partial\Omega}^{n}=0$ at every time level, and
\eqref{eq:fully-discrete-normalized-mass-drift} reduces to the normalized
mass drift.  We measure both its final value $\varepsilon_M^{n_f}$ and
\begin{equation}
 \varepsilon_M^{\max}:=\max_{0\le n\le n_f}\varepsilon_M^n,
 \qquad
 r_M^\star=\log_2\frac{\varepsilon_M^\star(\Delta t)}
                         {\varepsilon_M^\star(\Delta t/2)},
 \qquad \star\in\{n_f,\max\}.
 \label{eq:vortex-mass-drift-rate}
\end{equation}
The maximum-in-time value prevents an accidental cancellation at $t_f$ from
masking the temporal rate.  The calculation is run without subcell blending
and with a fixed step size.  A single MPI rank is used so that changes in the
parallel reduction order do not contaminate this small defect.

\begin{table}[H]
\centering
\footnotesize
\setlength{\tabcolsep}{7pt}
\begin{tabular}{|c|r|rr|rr|}
\hline
$\Delta t$ & steps
 & $\varepsilon_M^{n_f}$ & $r_M^{n_f}$
 & $\varepsilon_M^{\max}$ & $r_M^{\max}$\\
\hline
$2.0\times10^{-2}$ & 25
 & $5.0186\times10^{-9}$ & --
 & $5.0186\times10^{-9}$ & --\\
$1.0\times10^{-2}$ & 50
 & $6.2812\times10^{-10}$ & 2.998
 & $6.2812\times10^{-10}$ & 2.998\\
$5.0\times10^{-3}$ & 100
 & $7.8524\times10^{-11}$ & 3.000
 & $7.8524\times10^{-11}$ & 3.000\\
$2.5\times10^{-3}$ & 200
 & $9.7855\times10^{-12}$ & 3.004
 & $9.7855\times10^{-12}$ & 3.004\\
\hline
\end{tabular}
\caption{\textbf{Translating isentropic vortex:} temporal convergence of
the fully discrete normalized mass drift on the fixed $(N,M)=(3,8)$ mesh.
The boundary is periodic, the GC-DG scheme is unlimited, and SSP--RK3 uses
the prescribed constant time step.}
\label{tab:vortex-mass-drift}
\end{table}

The mass drift decreases by almost exactly a factor of eight whenever the time
step is halved, giving the observed orders $2.998$, $3.000$, and $3.004$.
Moreover, it grows monotonically in magnitude in these runs, so its
final and maximum-in-time values coincide to the displayed digits.  By
contrast, the density $L^2$ error remains approximately
$1.1886\times10^{-1}$ for all four time steps because it is already dominated
by the deliberately coarse spatial mesh.  Thus the temporal experiment in
Table~\ref{tab:vortex-mass-drift} verifies
$\varepsilon_M^{n_f},\varepsilon_M^{\max}=O(\Delta t^3)$ without conflating
this mass drift with the ordinary spatial error of the numerical
solution.

\subsection{One-dimensional Riemann problems}
\label{sec:riemann-tests}

Two planar Riemann problems isolate complementary wave systems: nonlinear
elastic waves governed by a general solid EOS and the homogeneous GPR system
with active thermal impulse.  The problems are computed on thin
two-dimensional Cartesian strips to exercise the
same multidimensional implementation used below.  The data are independent of
$y$, the $y$ boundaries are periodic, and the $x$ boundaries are
transmissive.  In addition to the solution profiles, we report the final
elementwise blending coefficient $\alpha_K$ and the fully discrete,
boundary-compensated mass drift $\varepsilon_M^n$ defined in
\eqref{eq:fully-discrete-normalized-mass-drift}.  Its accumulated boundary
contribution is computed from the same discrete flux
$\mathcal F_{\partial\Omega}^{\rm DG}$ as in
Theorem~\ref{thm:global-semidiscrete-mass}, using the SSP--RK3 stage update
\eqref{eq:ssprk3-accumulated-boundary-flux}.  The global subcell blend has
the same boundary flux by Theorem~\ref{thm:SC-companion-blending}.
For the reported final times, both tests have zero far-field normal velocity
and no wave reaches the boundary, so the accumulated boundary contribution
vanishes.

\subsubsection{Five-wave nonlinear elastic Riemann problem}
\label{sec:rp2}

The first Riemann test is the five-wave copper shock tube of Zhang et
al.~\cite{Zhang2017}, based on the nonlinear elastic model and Riemann-wave
analysis of Barton et al.~\cite{BartonEtAl2009}.  The published data are given
in terms of the deformation gradient $F$ and specific entropy $s$; the GPR
variable is $\A=F^{-1}$.  With $x_d=0.5$, the left and right states are
\begin{equation}
 \begin{aligned}
 \vct_L&=(0,1,0)^T,&
 F_L&=\begin{pmatrix}0.95&0&0\\0.05&1&0\\0&0&1\end{pmatrix},&
 s_L&=10^{-3},\\
 \vct_R&=(0,0,0)^T,&
 F_R&=I_3,&s_R&=0.
 \end{aligned}
 \label{eq:RP2-data}
\end{equation}
The density is initialized consistently as
$\rho=\rho_0\det\A$.  We use the Barton--Romenski copper closure of
Section~\ref{sec:model}, with
\begin{equation}
 \rho_0=8.93,\quad c_{L,0}=4.6,\quad b_0=2.1,\quad
 c_v=3.9\times10^{-4},\quad T_0=300,
 \qquad (\alpha_B,\beta_B,\gamma_B)=(1,3,2),
 \label{eq:RP2-material}
\end{equation}
in the units used in~\cite{Zhang2017}.  Distortion relaxation is inactive,
and the thermal-impulse subsystem is omitted.  This choice is useful because it
tests that the geometric construction and its limiter do not rely on an
ideal- or stiffened-gas closure.

We compute on $[0,1]\times[0,0.01]$ with $400\times4$ square elements,
$N=3$, CFL number $0.15$, $C_{\rm LF}=1.2$, and $t_f=0.06$.  The
reference solution is obtained by composing the five Lax curves as in
\cite{BartonEtAl2009}: rarefaction curves are integrated along the relevant
eigenvectors and shocks satisfy the full Rankine--Hugoniot relations.  Its
wave pattern is, from left to right, two rarefactions, a contact, a third
rarefaction, and a shock~\cite{Zhang2017}.

Figure~\ref{fig:RP2-profiles} compares the density, normal velocity, pressure,
shear flux-stress, and distortion with this iterative Riemann solution.  The
method resolves the five distinct waves and the intervening plateaus, including
the coupled changes in $\sigma_{21}$ and $A_{21}$, without visible Gibbs
oscillations.  Figure~\ref{fig:RP2-diagnostics} reports the final elementwise
blending coefficient $\alpha_K$ and the mass-drift history.  The coefficient
is concentrated near the wave fronts; because the far-field normal
velocities vanish and no wave reaches the boundary, the boundary mass flux is
zero and the final normalized fully discrete mass drift
$\varepsilon_M^{n_f}$ in \eqref{eq:fully-discrete-normalized-mass-drift}
is about $9.4\times10^{-13}$.

\begin{figure}[htbp]
\centering
\plotsubfigure{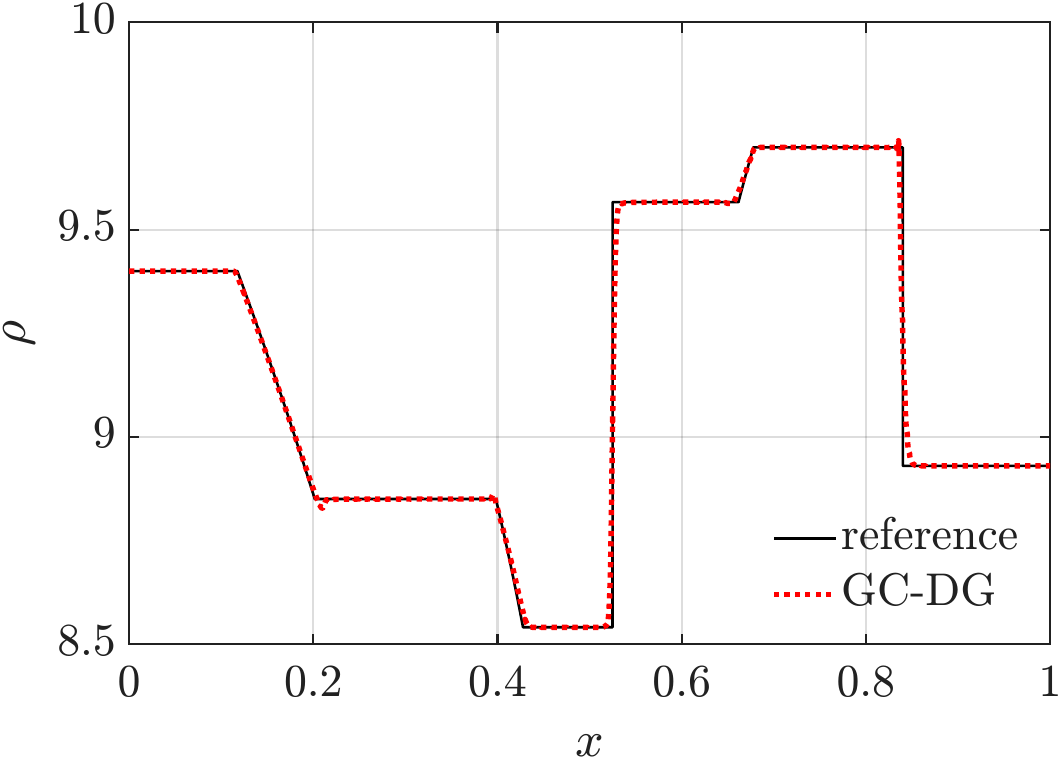}{Density $\rho$.}\hfill
\plotsubfigure{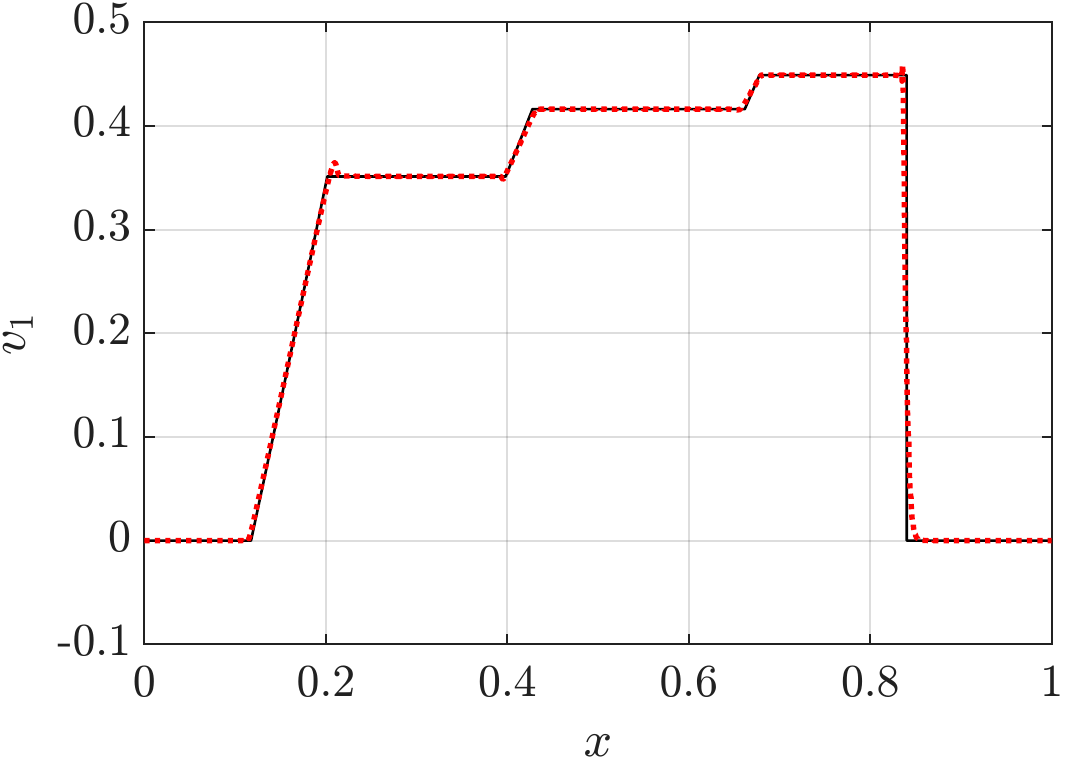}{Normal velocity $v_1$.}\hfill
\plotsubfigure{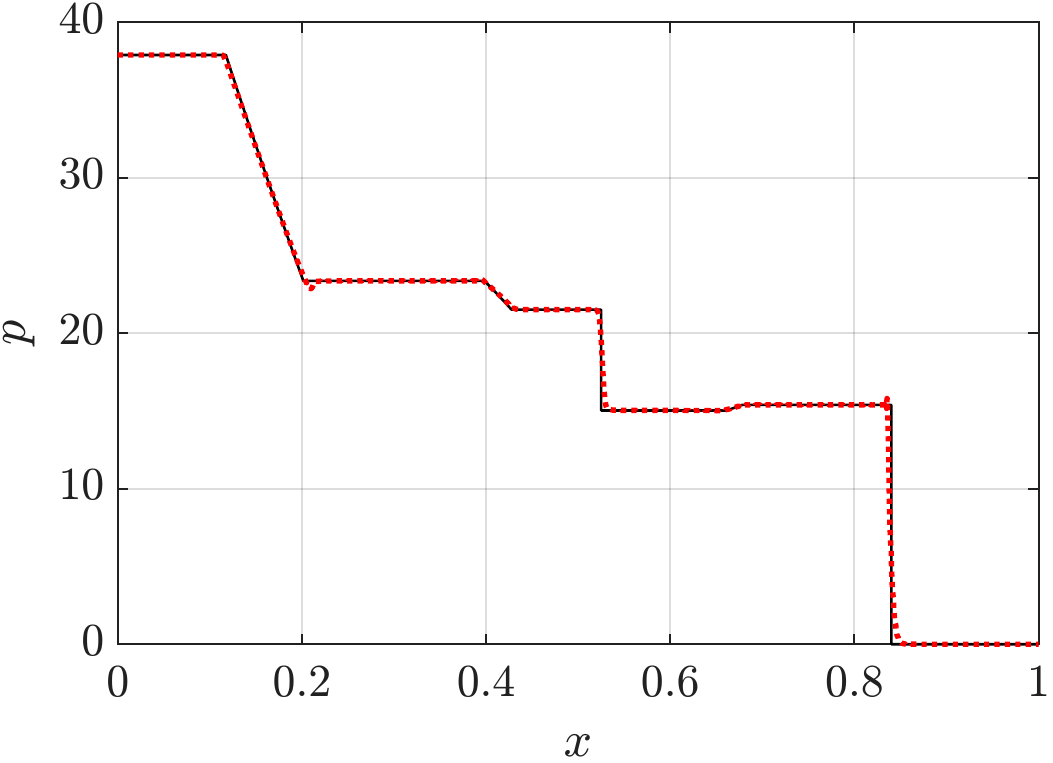}{Pressure $p$.}
\par\medskip
\plotsubfigure[0.40\textwidth]{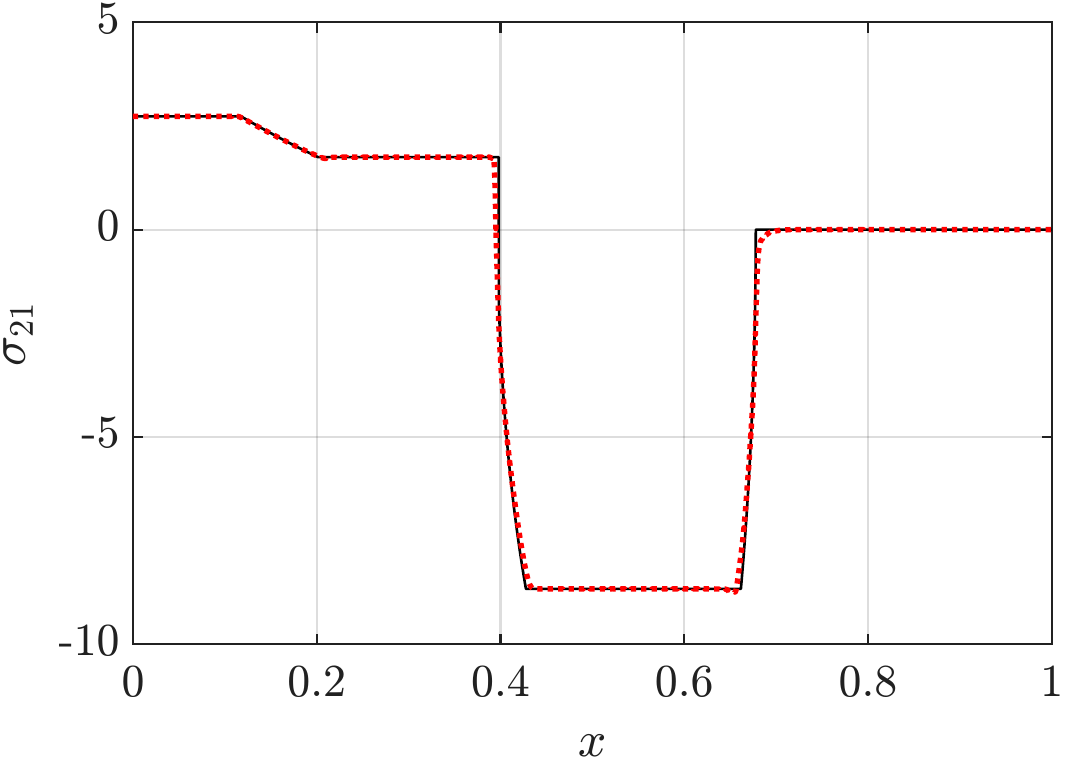}{Shear flux-stress $\sigma_{21}$; the
corresponding Cauchy component is $-\sigma_{21}$.}\hspace{0.05\textwidth}
\plotsubfigure[0.40\textwidth]{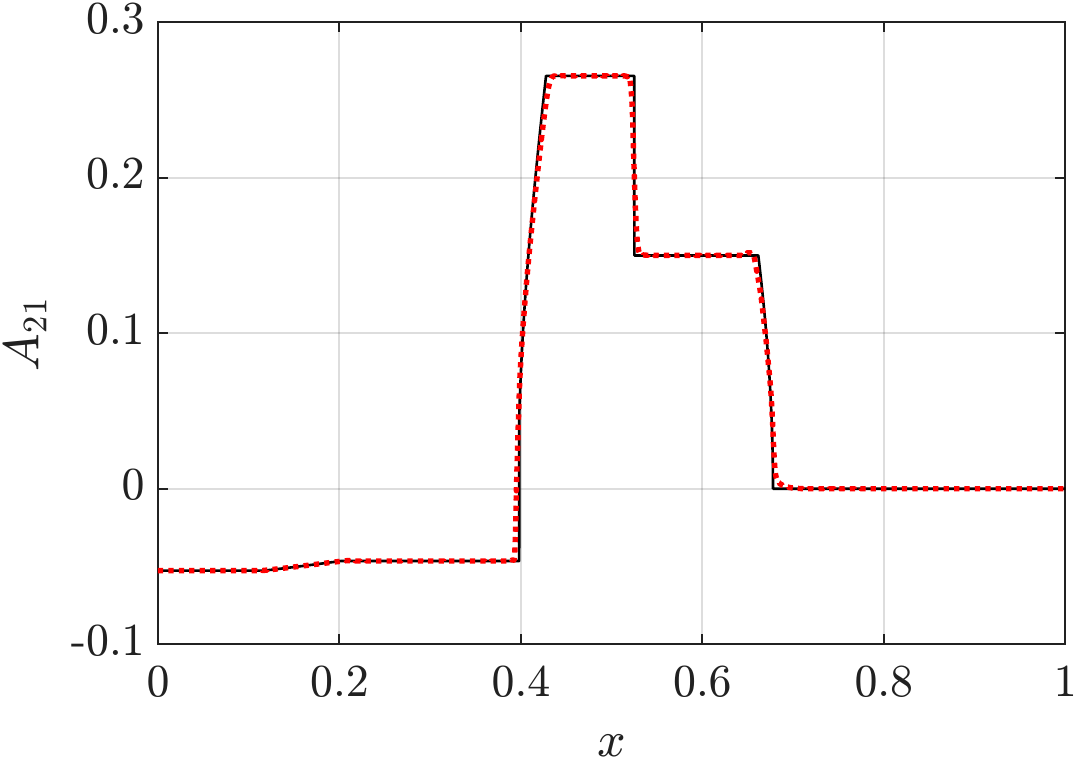}{Distortion component $A_{21}$.}
\caption{\textbf{Five-wave nonlinear elastic Riemann problem:} GC-DG profiles
at $t_f=0.06$ compared with the iterative nonlinear-elastic Riemann solution.
The red dotted curves are the GC-DG results, and the black solid curves are
the reference solution.  Panels show (a) density, (b) normal velocity,
(c) pressure, (d) shear flux-stress, and (e) distortion.}
\label{fig:RP2-profiles}
\end{figure}

\begin{figure}[htbp]
\centering
\plotsubfigure[0.47\textwidth]{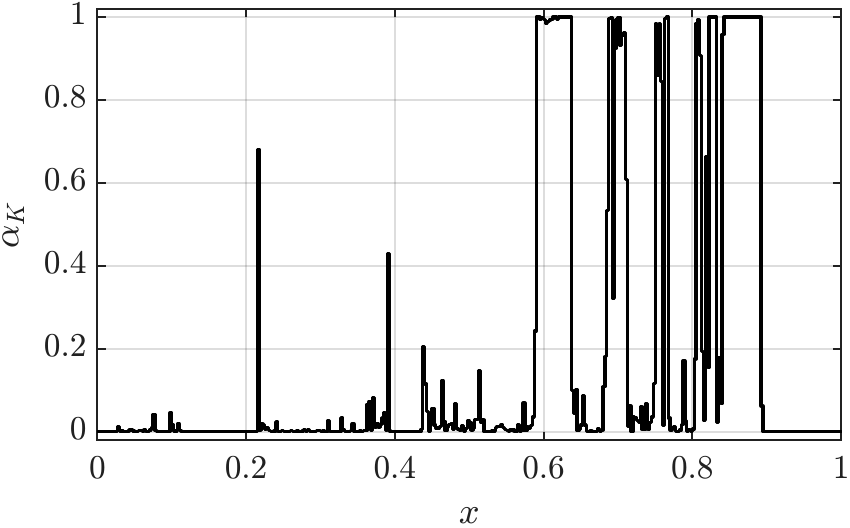}{Final elementwise blending coefficient $\alpha_K$.}\hfill
\plotsubfigure[0.47\textwidth]{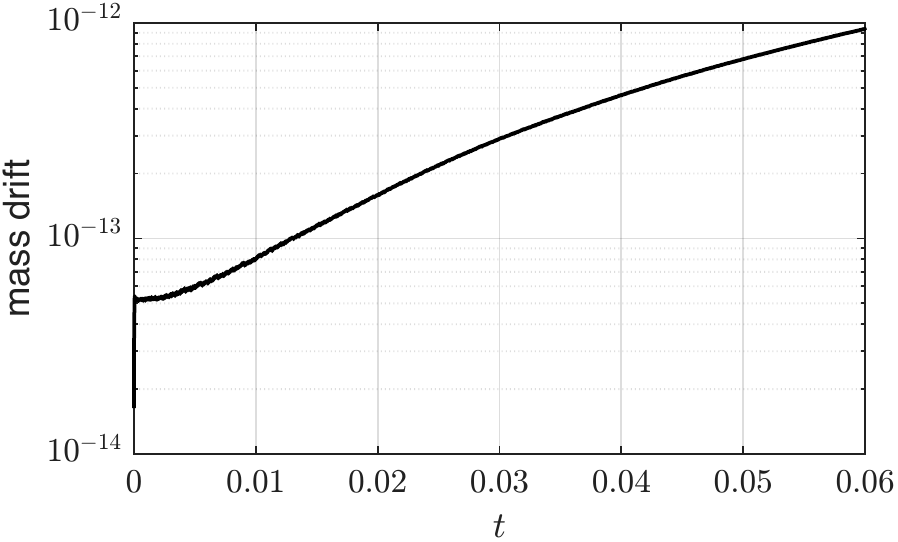}{Mass drift $\varepsilon_M^n$.}
\caption{\textbf{Five-wave nonlinear elastic Riemann problem:} Final
elementwise blending coefficient at $t_f=0.06$ and the fully discrete
mass-drift history: (a) $\alpha_K$ and (b) boundary-compensated mass drift.}
\label{fig:RP2-diagnostics}
\end{figure}

\subsubsection{Seven-wave homogeneous GPR Riemann problem}
\label{sec:rp3}

The final one-dimensional test is based on the third Riemann problem of
Dumbser et al.~\cite{Dumbser2026}, which exhibits the seven-wave structure of
the homogeneous GPR system.  It activates both elastic shear and thermal waves,
and is therefore substantially richer than an Euler shock tube.  The
macroscopic left and right data at $x_d=0$ are
\begin{equation}
 (\rho,v_1,v_2,p)_L=(1,0,-0.2,1),\qquad
 (\rho,v_1,v_2,p)_R=(0.5,0,0.2,0.5).
 \label{eq:RP3-data}
\end{equation}
Because density is not an independent variable in the reduced formulation,
we initialize
\begin{equation}
 \A_a=\left(\frac{\rho_a}{\rho_0}\right)^{1/3}I_3,
 \qquad \Jv_a=\mathbf0,\qquad a\in\{L,R\},
 \label{eq:RP3-distortion-data}
\end{equation}
which enforces $\rho_a=\rho_0\det\A_a$ exactly.  We take an ideal gas with
$\rho_0=1$, $\gamma=1.4$, $c_v=1$, and $R_{\rm gas}=0.4$, together with
$c_s=c_h=1$.  The relaxation times are
$\tau_1=\tau_2=10^{20}$, so the calculation probes the homogeneous
hyperbolic dynamics rather than a stiff relaxation limit.

The domain is $[-0.5,0.5]\times[-0.005,0.005]$, discretized by
$400\times4$ square elements with $N=3$.  We use CFL number $0.15$,
$C_{\rm LF}=1$, the thermal-wave estimate of
Section~\ref{sec:GC-FV-wave-speed}, and evolve to $t_f=0.2$.  For comparison,
we compute a self-similar reference by composing rarefaction curves and
straight-path DLM shock curves for the same reduced system; this is the
nonconservative analogue of the iterative construction in
\cite{BartonEtAl2009}, with the generalized jump condition of
\cite{DalMaso1995,Pares2006} imposed across discontinuities.

The six profiles in Fig.~\ref{fig:RP3-profiles} display the seven-wave
structure emphasized in~\cite{Dumbser2026}.  GC-DG separates all seven waves and
closely reproduces the small intermediate plateaus in $\rho$, $A_{11}$,
$A_{21}$, and $J_1$.  Figure~\ref{fig:RP3-diagnostics} reports the final
elementwise blending coefficient $\alpha_K$ and the mass-drift history.  The
coefficient is localized near the wave fronts and remains small on the
resolved plateaus; since $v_1=0$ in both far fields, there is no net mass flux
through the $x$ boundaries and the final normalized fully discrete mass drift is
$\varepsilon_M^{n_f}=6.6\times10^{-9}$.

\begin{figure}[htbp]
\centering
\plotsubfigure{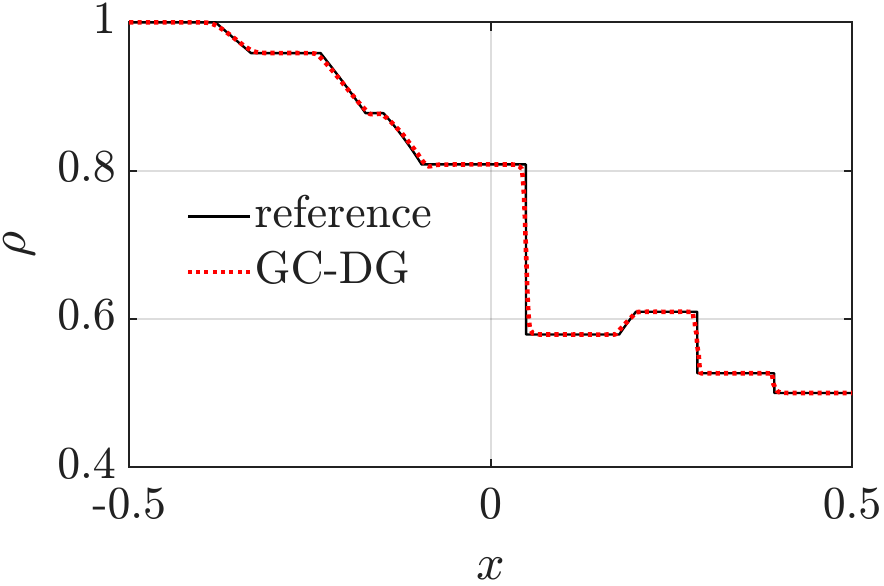}{Density $\rho$.}\hfill
\plotsubfigure{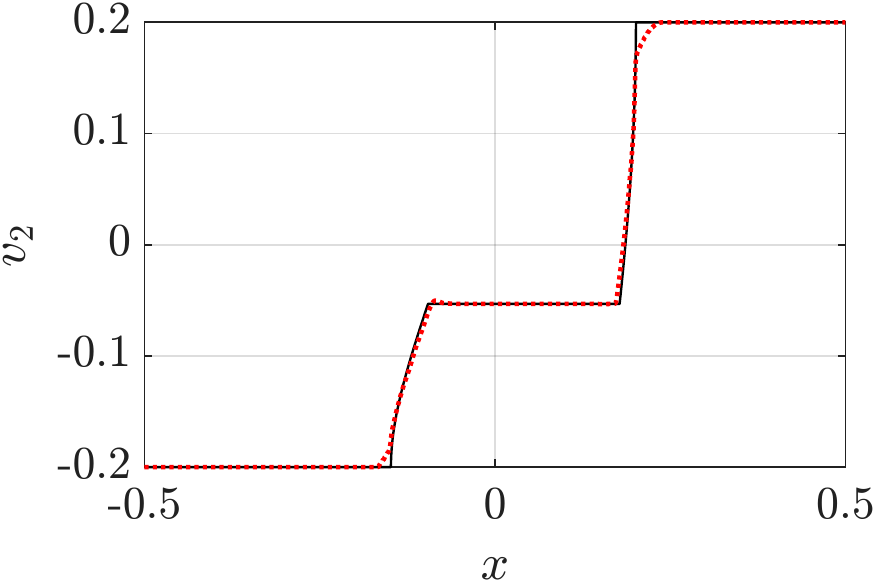}{Transverse velocity $v_2$.}\hfill
\plotsubfigure{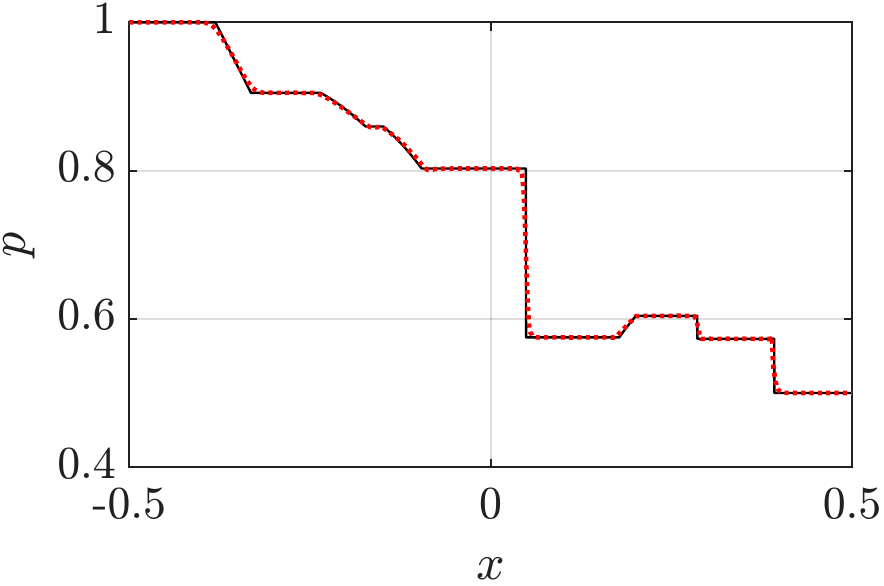}{Pressure $p$.}
\par\medskip
\plotsubfigure{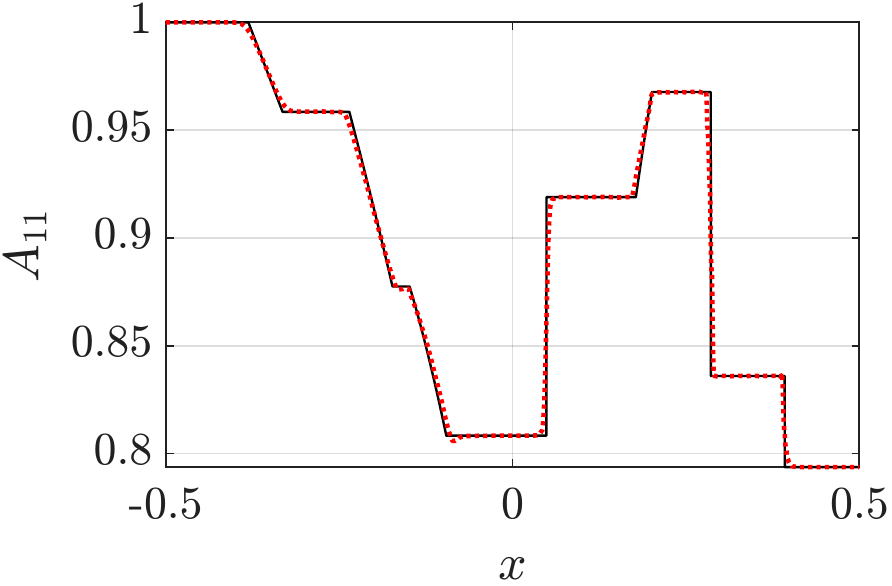}{Distortion component $A_{11}$.}\hfill
\plotsubfigure{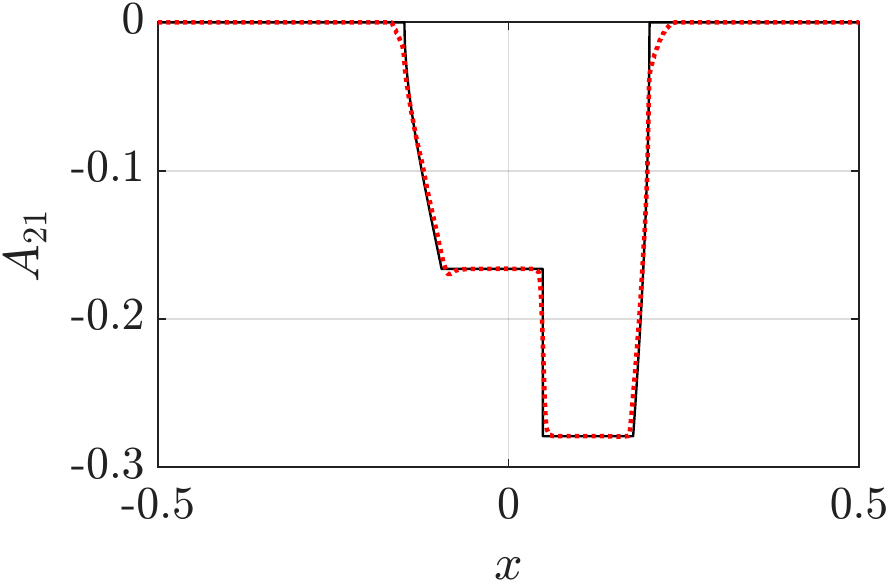}{Distortion component $A_{21}$.}\hfill
\plotsubfigure{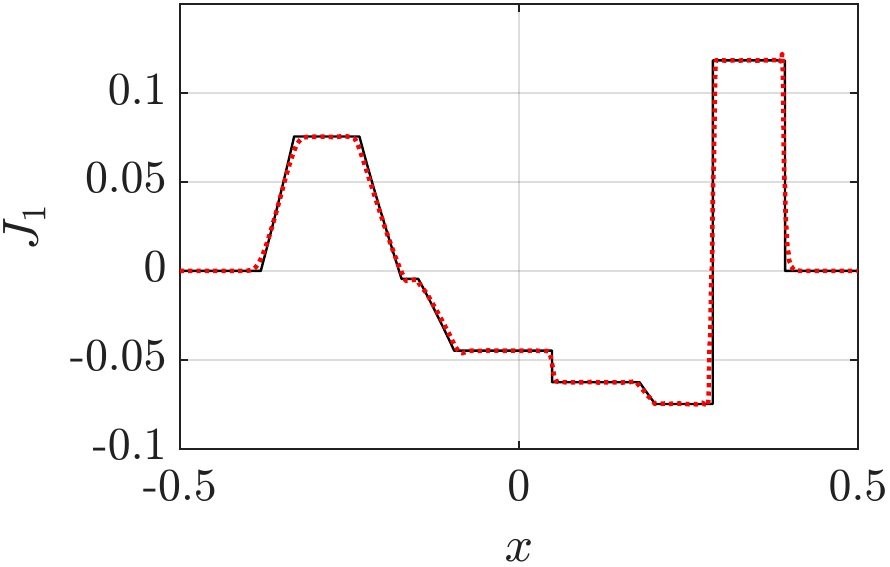}{Thermal impulse $J_1$.}
\caption{\textbf{Seven-wave homogeneous GPR Riemann problem:} GC-DG
profiles at $t_f=0.2$ compared with the iterative straight-path DLM Riemann
solution.  The red dotted curves are the GC-DG results, and the black solid
curves are the reference solution.  Panels show (a) density, (b) transverse
velocity, (c) pressure, (d)--(e) distortion components, and (f) thermal
impulse.}
\label{fig:RP3-profiles}
\end{figure}

\begin{figure}[htbp]
\centering
\plotsubfigure[0.47\textwidth]{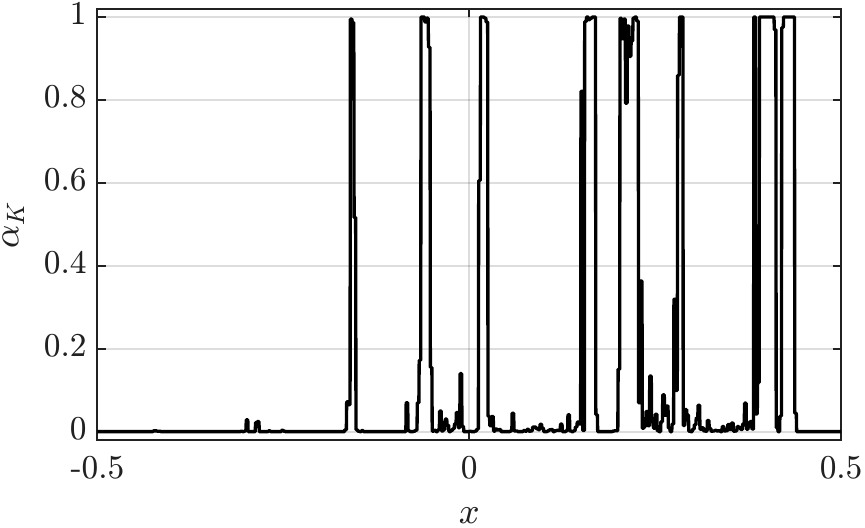}{Final elementwise blending coefficient $\alpha_K$.}\hfill
\plotsubfigure[0.47\textwidth]{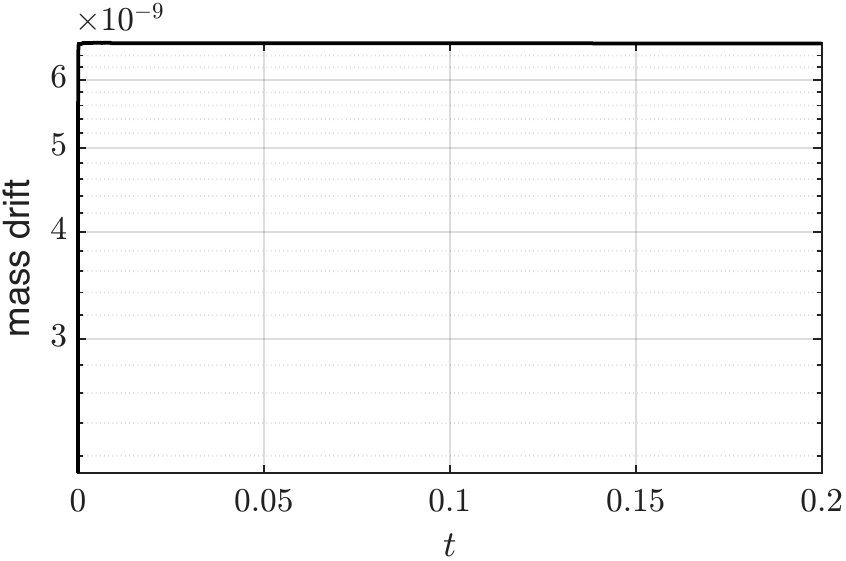}{Mass drift $\varepsilon_M^n$.}
\caption{\textbf{Seven-wave homogeneous GPR Riemann problem:} Final
elementwise blending coefficient at $t_f=0.2$ and the fully discrete
mass-drift history: (a) $\alpha_K$ and (b) boundary-compensated mass drift.}
\label{fig:RP3-diagnostics}
\end{figure}

\subsection{Circular explosion}
\label{sec:circular-explosion}

The circular explosion is a standard genuinely two-dimensional GPR
benchmark~\cite{Dumbser2016,Dumbser2026}.  A high-pressure disk expands
into a lower-pressure exterior and generates concentric waves whose
propagation tests multidimensional symmetry, shock capturing, and the
coupling between the conservative and nonconservative blocks.  Following the
EP1--EP2 comparison of Dumbser et al.~\cite{Dumbser2026}, we consider both a
fluid limit and a homogeneous elastic-solid limit.  Unlike a collection of
one-dimensional shock tubes, these calculations exercise the complete
two-dimensional volume and face operators, while radial symmetry provides a
convenient independent comparison solution.

Both tests use $\Omega=[-1,1]^2$.  Let
$r=(x^2+y^2)^{1/2}$ and $R_0=0.5$.  The common initial data are
\begin{equation}
 (\rho,p,\vct)(x,y,0)=
 \begin{cases}
  (1,1,\mathbf 0),&r\le R_0,\\
  (0.125,0.1,\mathbf 0),&r>R_0,
 \end{cases}
 \qquad
 \A(x,y,0)=
 \left(\frac{\rho(x,y,0)}{\rho_0}\right)^{1/3}I_3,
 \label{eq:circular-explosion-data}
\end{equation}
and $\Jv(x,y,0)=\mathbf0$ whenever the thermal-impulse equation is active.
Thus $\rho=\rho_0\det\A$ holds exactly at initialization.  The total energy
is recovered from these primitive data and the constitutive energy of
Section~\ref{sec:model}.  We take the ideal-gas member of
\eqref{eq:stiffened-gas-EOS}, with
\begin{equation}
 \rho_0=1,\qquad R_{\rm gas}=1,\qquad c_v=2.5,
 \qquad \gamma=1.4,\qquad p_\infty=0.
 \label{eq:circular-explosion-material}
\end{equation}
Transmissive conditions are imposed on all four boundaries, and the
discontinuous initial condition is evaluated directly at every LGL nodal
point, consistently with the initialization convention stated at the
beginning of this section.  Both runs use CFL number $0.15$ and
$C_{\rm LF}=1$.  We use $C_{\rm OE}=0.1$ for the $N=2$
fluid run and $C_{\rm OE}=0.02$ for the $N=4$ solid run.  The polynomial
degrees are intentionally different:
the fluid solution contains the familiar outgoing shock, contact, and
rarefaction, for which $N=2$ is sufficient, whereas the solid problem
contains a much richer sequence of elastic and thermal waves separated by
narrow intermediate structures.  We therefore use $N=4$ in the latter case
to resolve these fine features.  Apart from this resolution choice and the
different final times, no problem-specific modification of the limiter is
used.

\subsubsection{Fluid circular explosion}
\label{sec:circular-explosion-fluid}

For the fluid test we set $c_s=c_h=0$ and omit the thermal-impulse equation.
The resulting mechanical system has the Euler pressure law, while density is
still recovered from $\det\A$ and therefore continues to test the discrete
determinant companion relation.  This differs from the stiff-relaxation EP1 of
\cite{Dumbser2026}: here the inactive shear and thermal mechanisms are
removed rather than approximated with very small relaxation times.  Hence a
radially symmetric Euler computation provides a reference for precisely the
same limiting problem.

We use a uniform $600\times600$ Cartesian mesh, polynomial degree $N=2$, and
evolve to $t_f=0.2$.  An independent comparison solution is computed from
the radial Euler equations, including the cylindrical geometric source, using a
characteristic fifth-order WENO-Z discretization
\cite{JiangShu1996,Borges2008} on 2000 radial cells.  The comparison is made
on the positive $x$ axis at $y=0^+$.

Figure~\ref{fig:circular-explosion-fluid-profiles} compares the GC-DG density,
radial velocity, and pressure with the radial WENO-Z solution.  The profiles
reproduce the positions and amplitudes of the outgoing shock, contact wave,
and inward rarefaction; the contact and narrow inner features are smoother,
as expected because the elementwise blend activates the more dissipative
GC-FV contribution near discontinuities.  The fluid density surface in
Fig.~\ref{fig:circular-explosion-density-surfaces}(a) remains circular and
free of visible grid-aligned distortion.  The final boundary-compensated
fully discrete mass drift is
$\varepsilon_M^{n_f}=1.00\times10^{-7}$; its temporal distribution is
discussed with the diagnostics below.

\begin{figure}[htbp]
\centering
\plotsubfigure{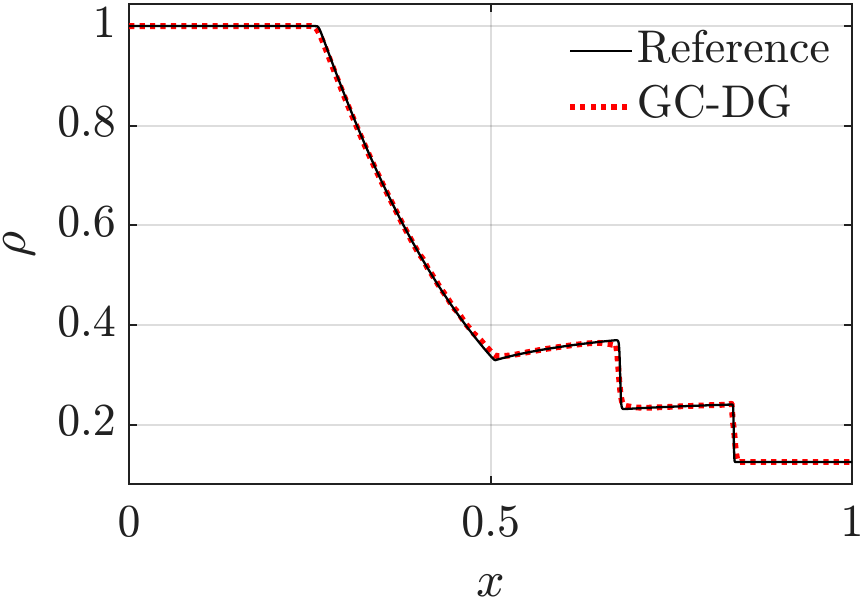}{Density $\rho$.}\hfill
\plotsubfigure{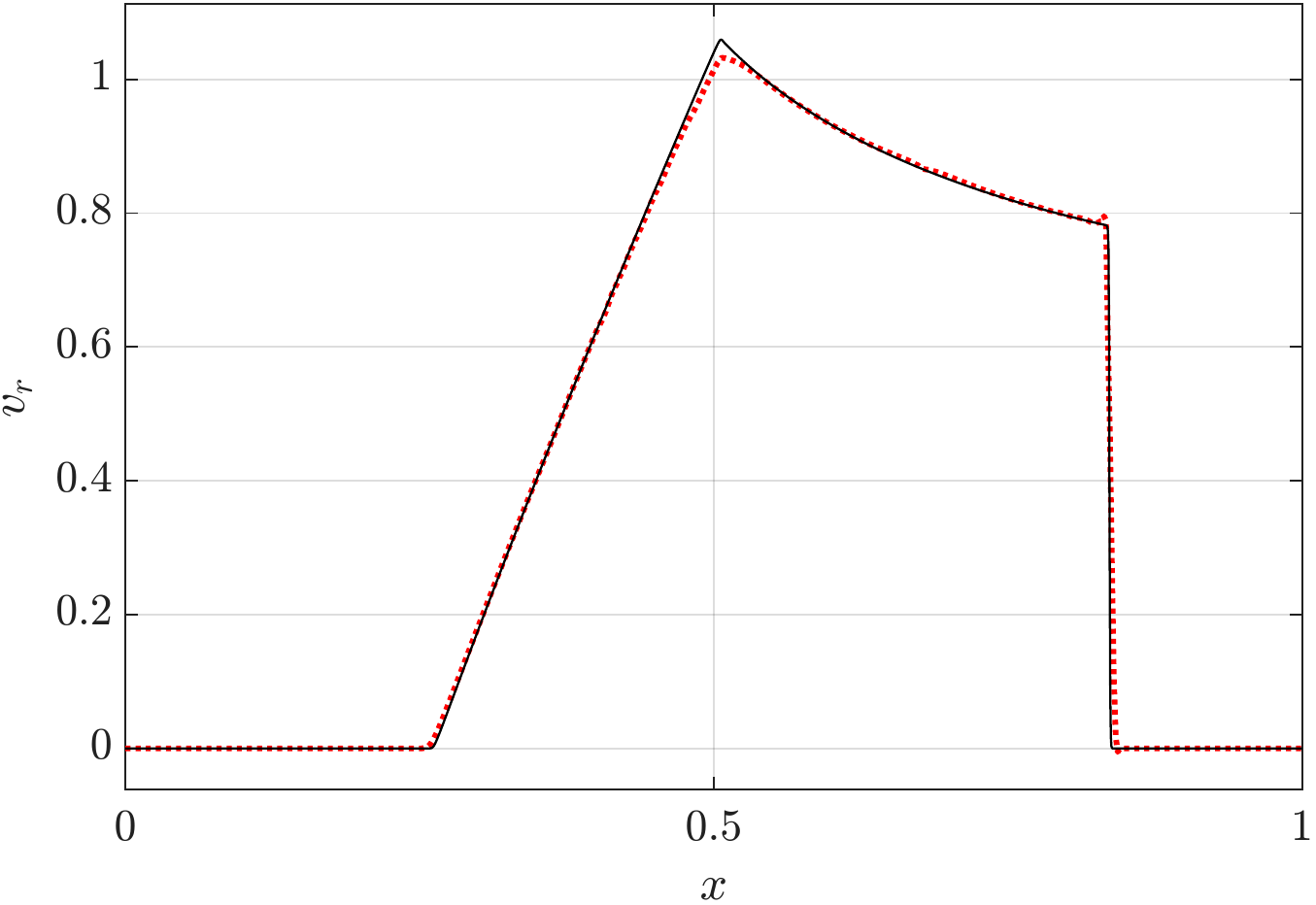}{Radial velocity $v_r$.}\hfill
\plotsubfigure{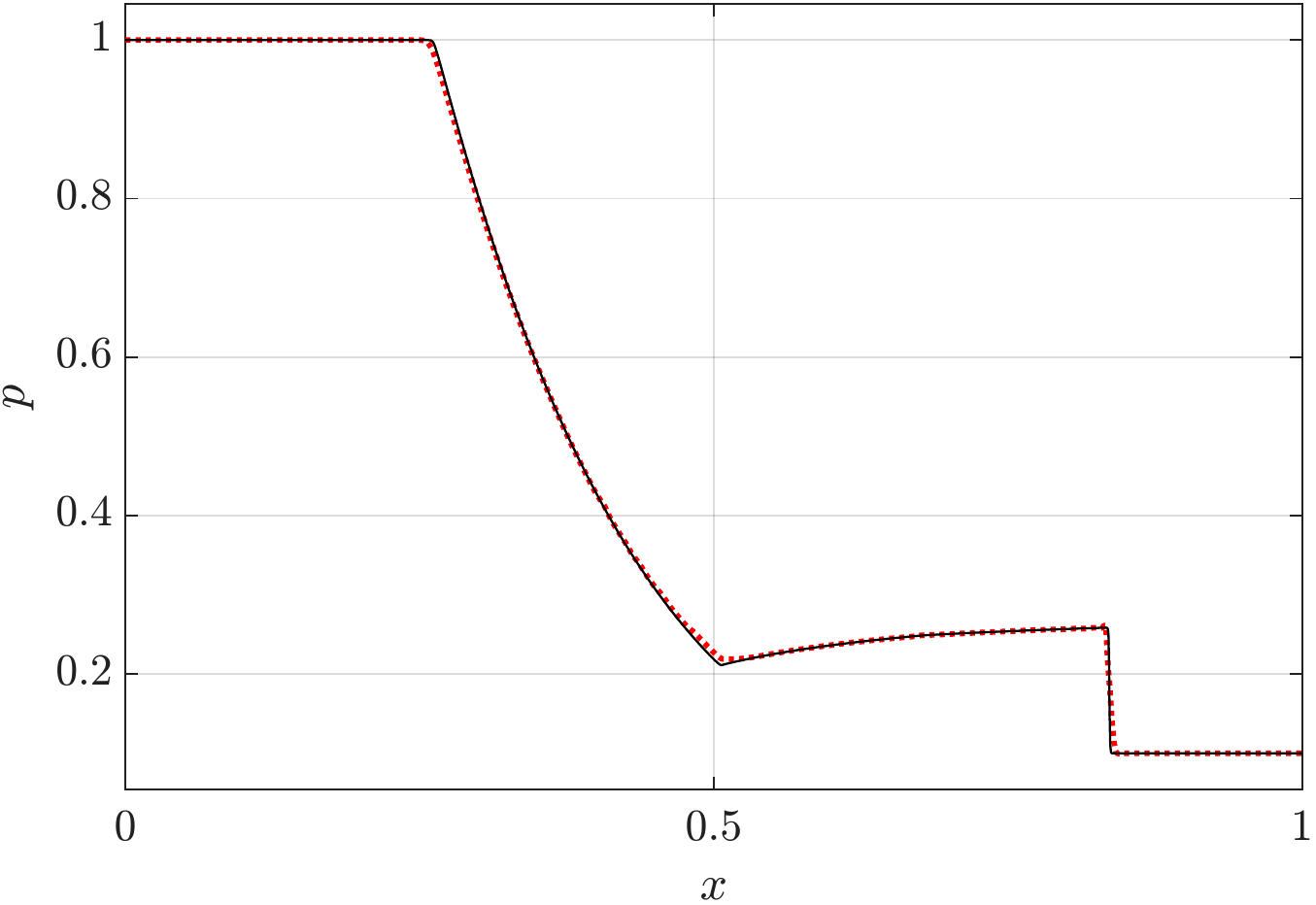}{Pressure $p$.}
\caption{\textbf{Fluid circular explosion:} GC-DG profiles at $t_f=0.2$
on $0\le x\le1$, $y=0^+$, compared with the radially symmetric Euler
WENO-Z comparison solution.  The red dotted curves are the GC-DG results,
and the black solid curves are the comparison solution.  Panels show
(a) density, (b) radial velocity, and (c) pressure.}
\label{fig:circular-explosion-fluid-profiles}
\end{figure}

\subsubsection{Solid circular explosion}
\label{sec:circular-explosion-solid}

The solid test retains the same initial states but activates both shear and
thermal waves.  We set
\begin{equation}
 c_s=c_h=1,\qquad \tau_1=\tau_2=10^{20},
 \label{eq:circular-explosion-solid-parameters}
\end{equation}
so that the computation represents the homogeneous elastic limit rather
than a stiff relaxation regime.  The thermal impulse is evolved, initially
with $\Jv=\mathbf0$, and the thermal-wave contribution to the LF estimate is
included with $\kappa_J=10$ as described in
Section~\ref{sec:GC-FV-wave-speed}.  We use a $600\times600$ mesh with
$N=4$ and evolve to $t_f=0.15$.

For comparison, an independent solution is obtained by solving the radially
symmetric original GPR--$J$ system, in which density is evolved explicitly,
with a monolithic characteristic fifth-order path-conservative WENO-Z
discretization \cite{Pares2006,JiangShu1996,Borges2008} on 2000 radial cells.
The initial data satisfy the determinant--density relation, so this provides
an independent comparison between the reduced GC-DG formulation and the
full GPR system.  In addition to $\rho$, $v_r$, and $p$, we compare $A_{11}$
and the radial thermal impulse $J_r=\Jv\cdot\mathbf e_r$.

The solid response contains a substantially richer train of elastic and
thermal waves than the fluid calculation.  Figure~\ref{fig:circular-explosion-solid-profiles}
shows that GC-DG reproduces the comparison plateaus, rarefactions, and sharp
fronts in all five quantities.  In particular, the agreement in $A_{11}$
assesses the distortion update underlying the hidden density law, while that
in $J_r$ tests the straight-path thermal-impulse discretization.  The solid
density surface in Fig.~\ref{fig:circular-explosion-density-surfaces}(b)
remains radially symmetric.  Only weak localized oscillations occur near a
few closely spaced solid-wave fronts; they neither contaminate the large
constant regions nor alter the wave locations and plateau structure.  The
final boundary-compensated fully discrete mass drift is
$\varepsilon_M^{n_f}=1.25\times10^{-7}$, which is small compared with the
$O(1)$ density variation; its temporal distribution is discussed with the
diagnostics below.

\begin{figure}[htbp]
\centering
\plotsubfigure{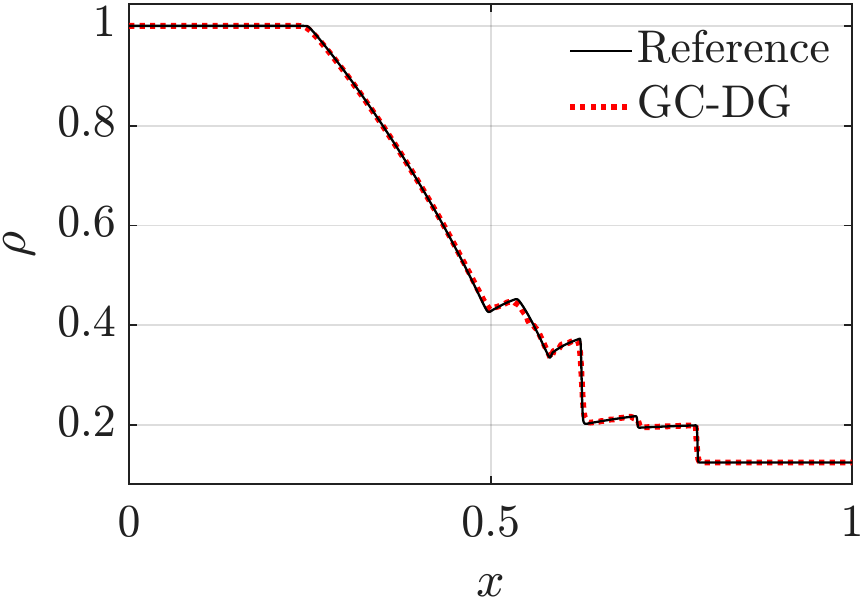}{Density $\rho$.}\hfill
\plotsubfigure{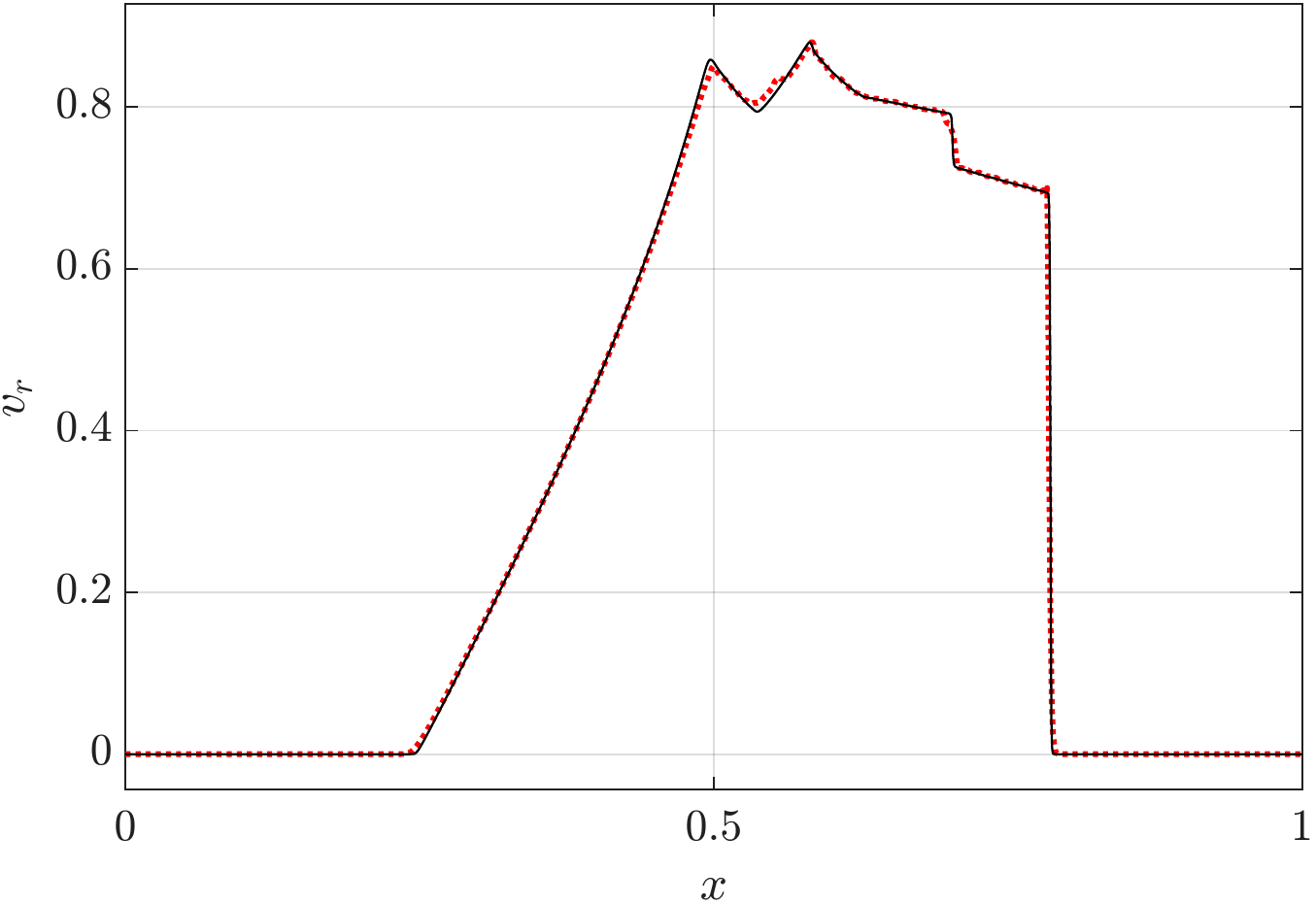}{Radial velocity $v_r$.}\hfill
\plotsubfigure{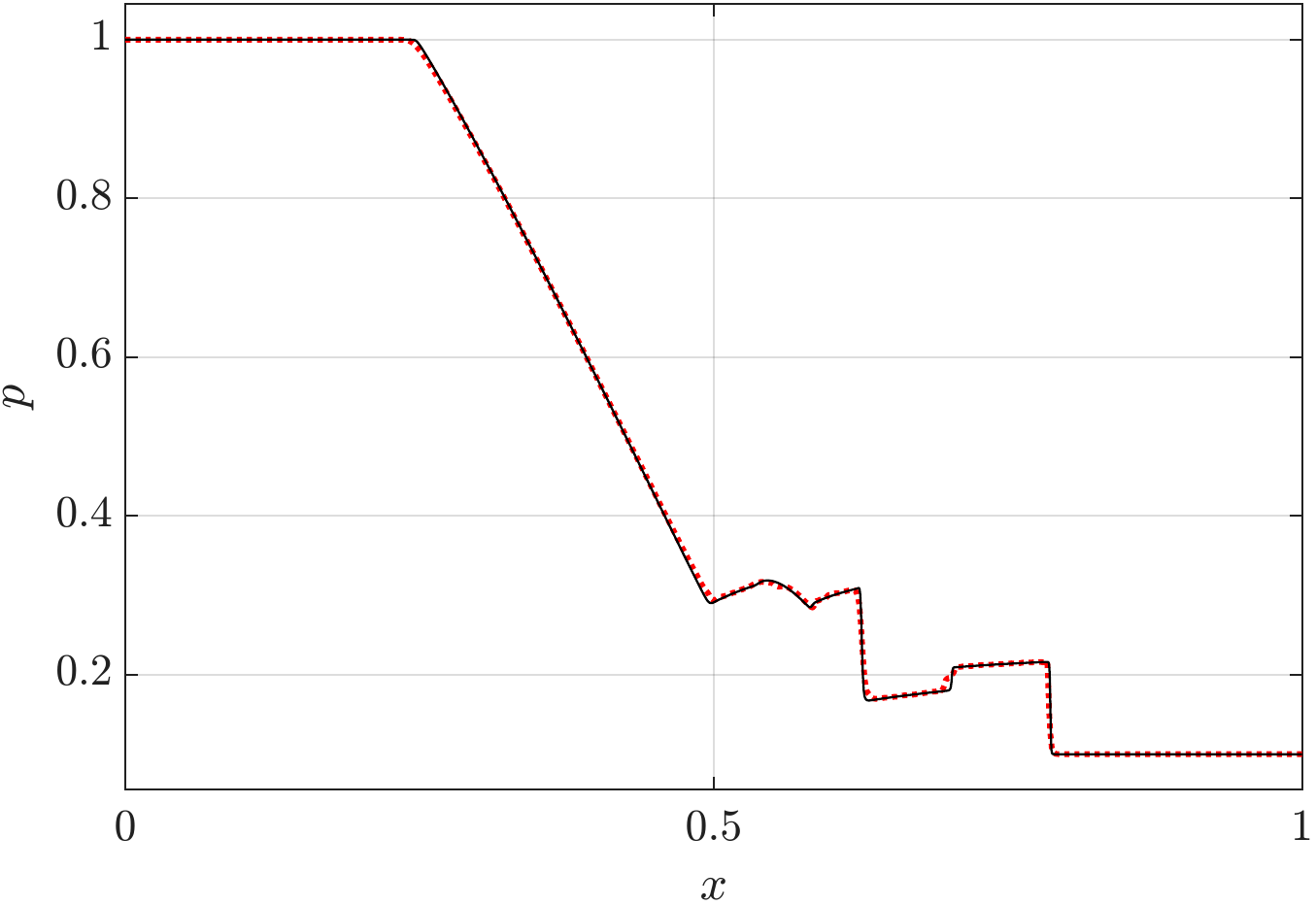}{Pressure $p$.}
\par\medskip
\plotsubfigure[0.40\textwidth]{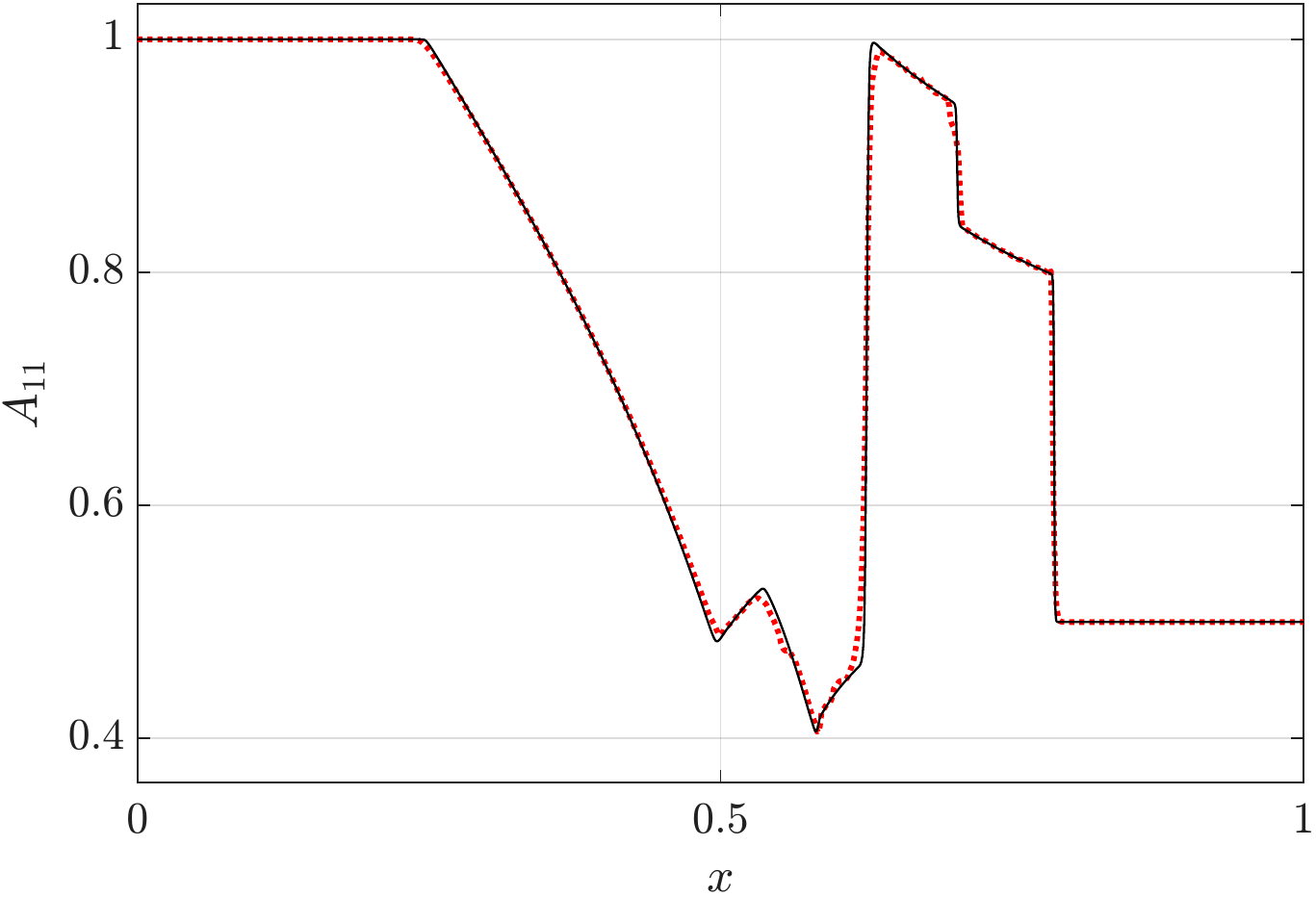}{Distortion component $A_{11}$.}\hspace{0.05\textwidth}
\plotsubfigure[0.40\textwidth]{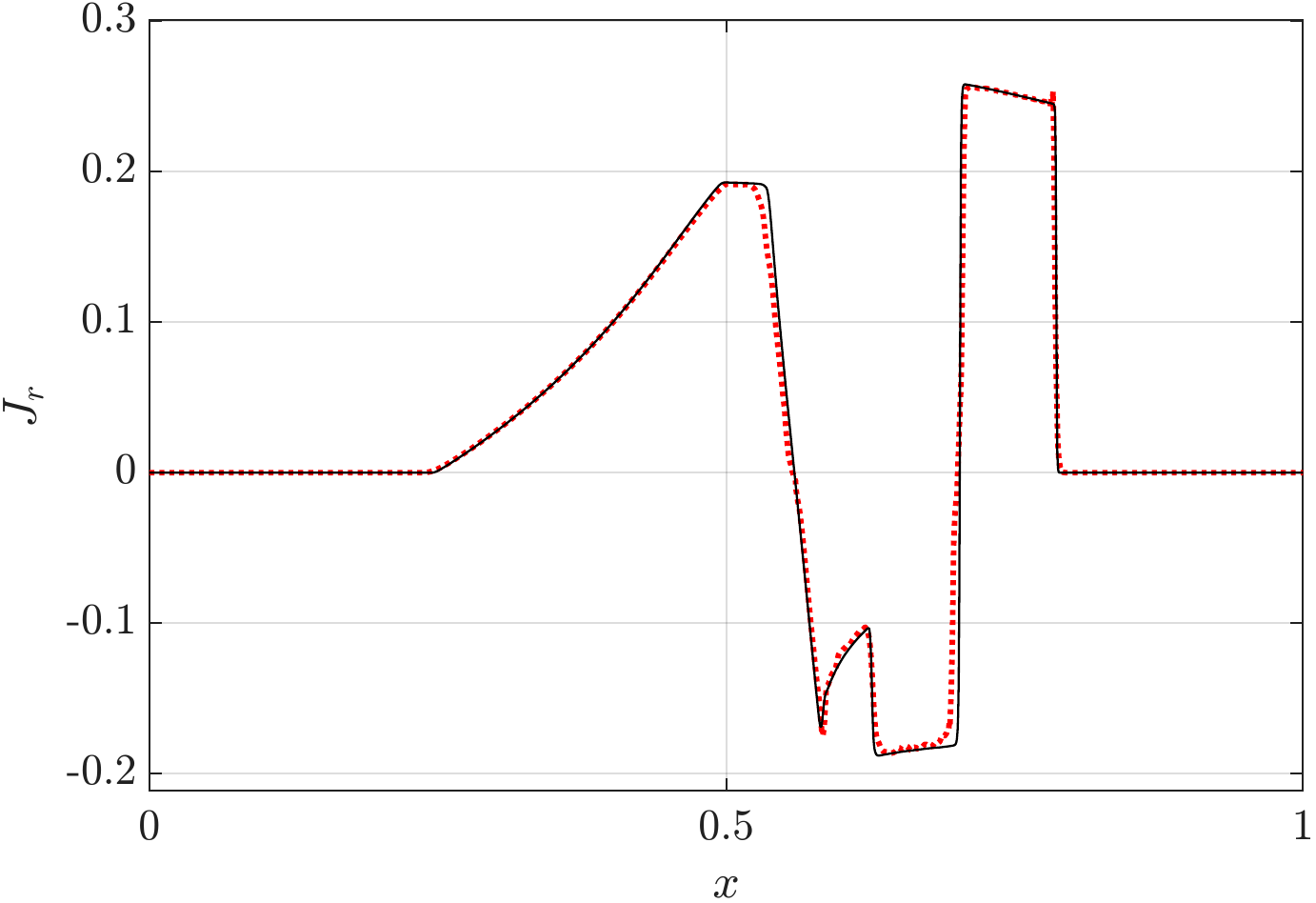}{Radial thermal impulse $J_r$.}
\caption{\textbf{Solid circular explosion:} GC-DG profiles at $t_f=0.15$
on $0\le x\le1$, $y=0^+$, compared with the radially symmetric original
GPR--$J$ WENO-Z comparison solution.  The red dotted curves are the GC-DG
results, and the black solid curves are the comparison solution.  Panels show
(a) density, (b) radial velocity, (c) pressure, (d) distortion, and
(e) radial thermal impulse.}
\label{fig:circular-explosion-solid-profiles}
\end{figure}

\begin{figure}[H]
\centering
\plotsubfigure[0.47\textwidth]{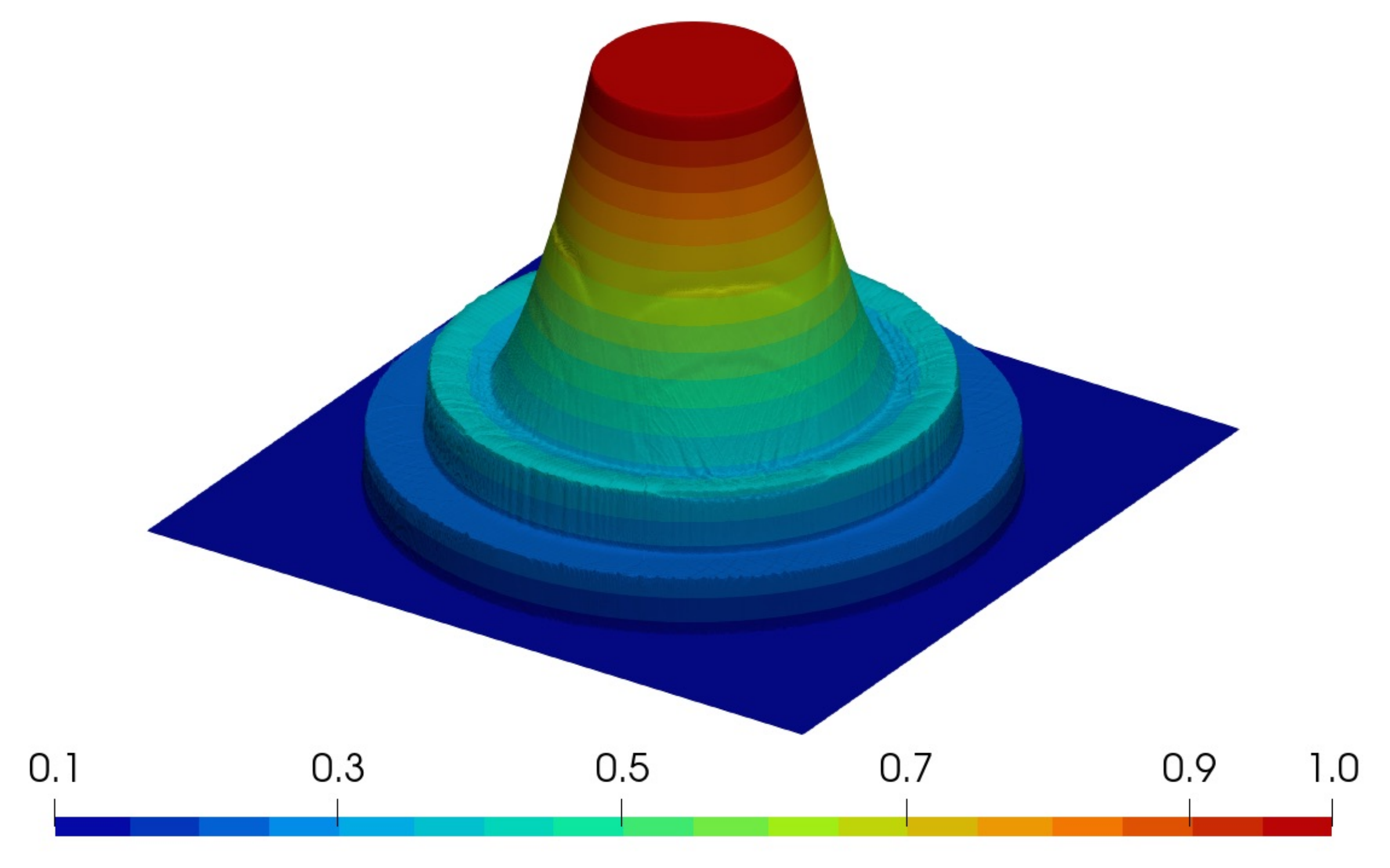}{Fluid density at $t_f=0.2$.}\hfill
\plotsubfigure[0.47\textwidth]{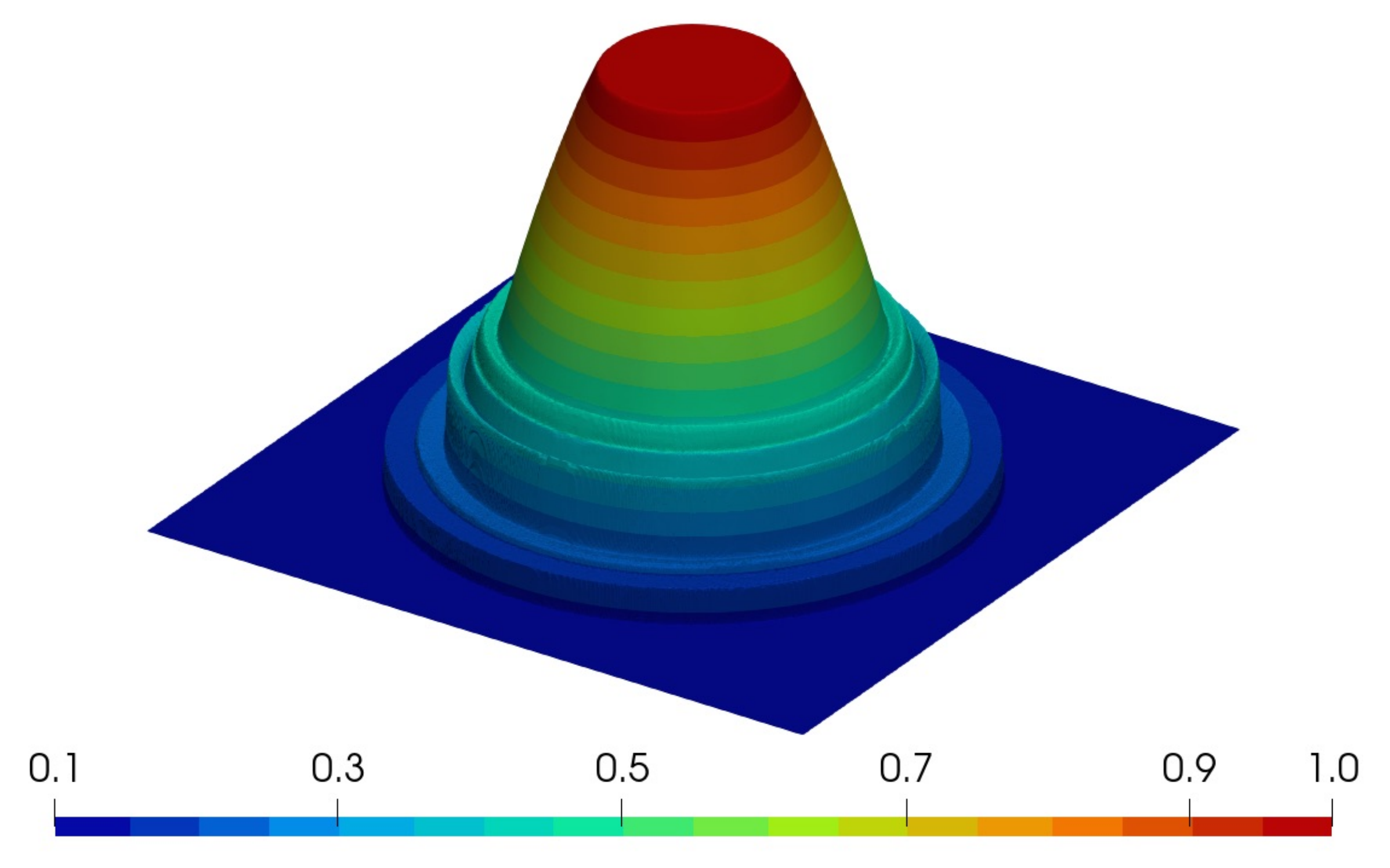}{Solid density at $t_f=0.15$.}
\caption{\textbf{Circular explosion:} Three-dimensional density surfaces
for (a) the fluid case at $t_f=0.2$ and (b) the elastic-solid case at
$t_f=0.15$.  The height and color both represent the computed density.}
\label{fig:circular-explosion-density-surfaces}
\end{figure}

The final elementwise blending coefficients $\alpha_K$ in
Fig.~\ref{fig:circular-explosion-diagnostics} respond to the circular wave
patterns in both regimes.  In the degree-$2$ fluid calculation, $\alpha_K$
extends across a substantial part of the rarefaction-to-shock fan rather
than being confined to a thin front; in the degree-$4$ solid calculation it
is more localized around the finer wave structures.  This contrast is
consistent with the lower polynomial degree: at $N=2$, fewer modal scales
are available to distinguish resolved from under-resolved content, whereas
higher-degree indicators generally localize activation more selectively
near sharp wave fronts.  The constant inner and outer states remain
predominantly governed by the high-order operator.  The
lower panels report the corresponding mass-drift histories on logarithmic
axes.  Their rapid initial rise is caused by the nodally sampled discontinuity:
the smooth-trajectory estimate of Remark~\ref{rem:fully-discrete-mass-drift}
does not apply while the first wave fronts form.  By $t=0.01t_f$, the
remaining changes to the final drift are $2.89\times10^{-9}$ in the fluid
case and $6.87\times10^{-11}$ in the solid case, so subsequent accumulation
is much smaller than the initial offset.

\begin{figure}[H]
\centering
\plotsubfigure[0.43\textwidth]{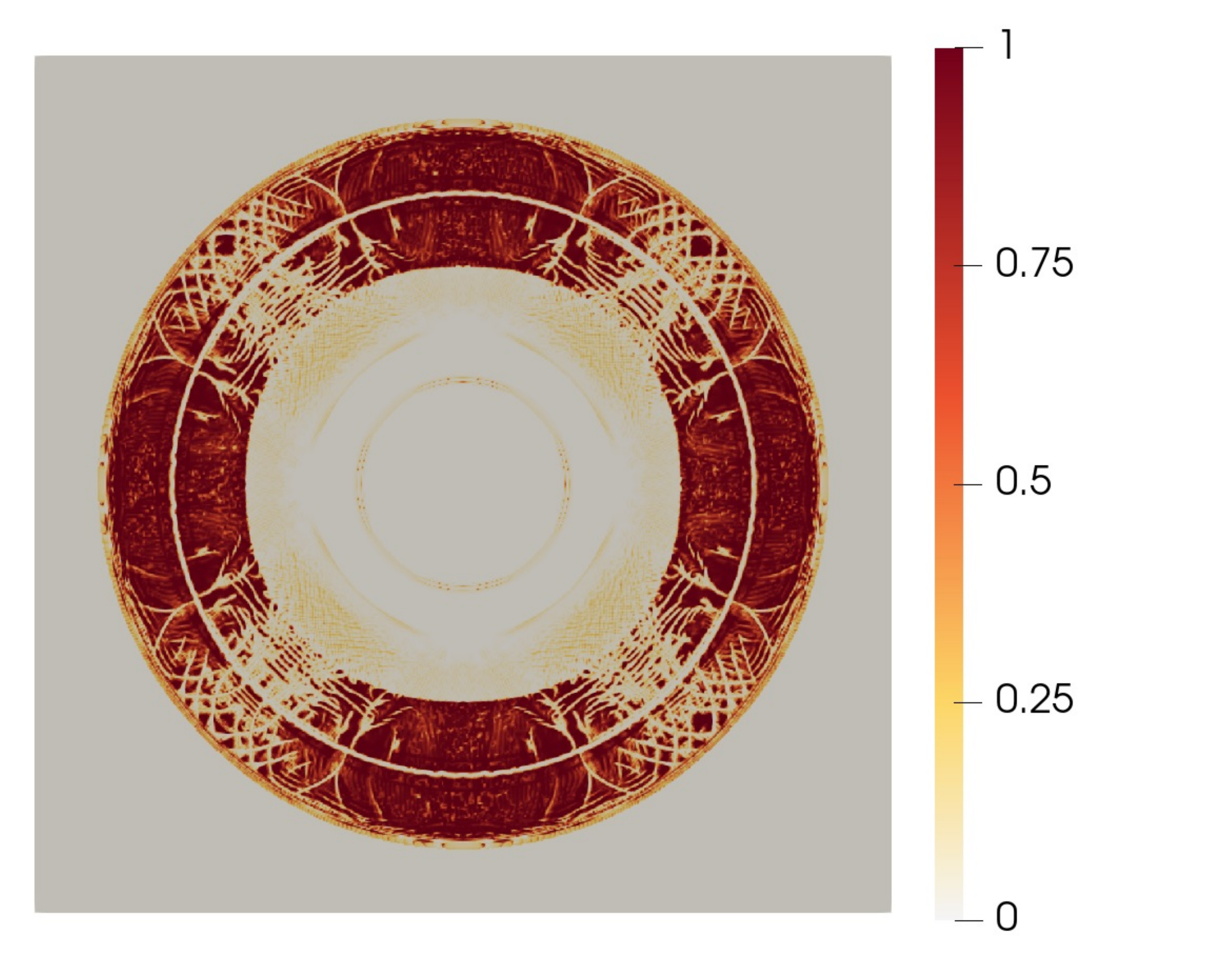}{Fluid elementwise blending coefficient $\alpha_K$.}\hfill
\plotsubfigure[0.43\textwidth]{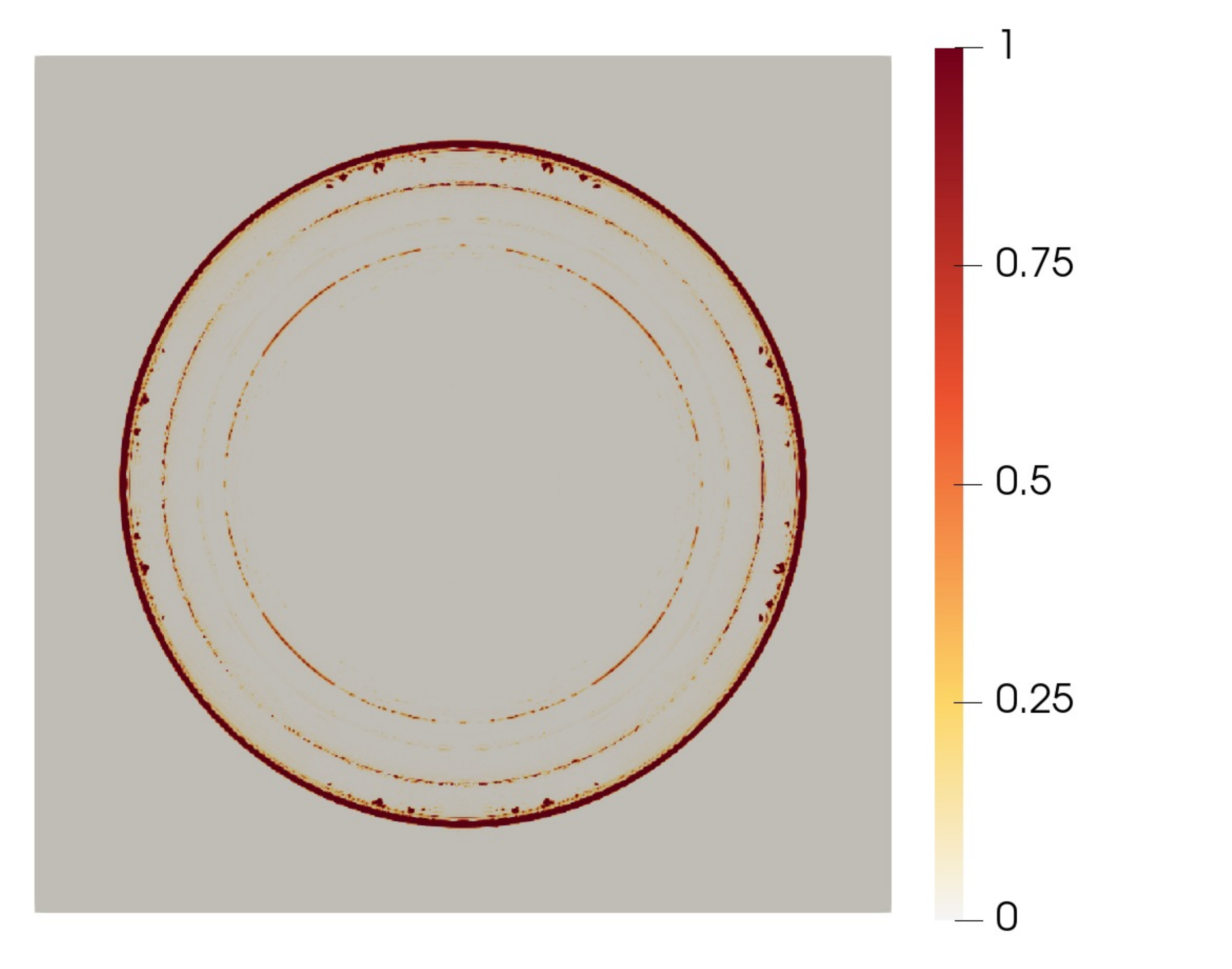}{Solid elementwise blending coefficient $\alpha_K$.}
\par\smallskip
\plotsubfigure[0.43\textwidth]{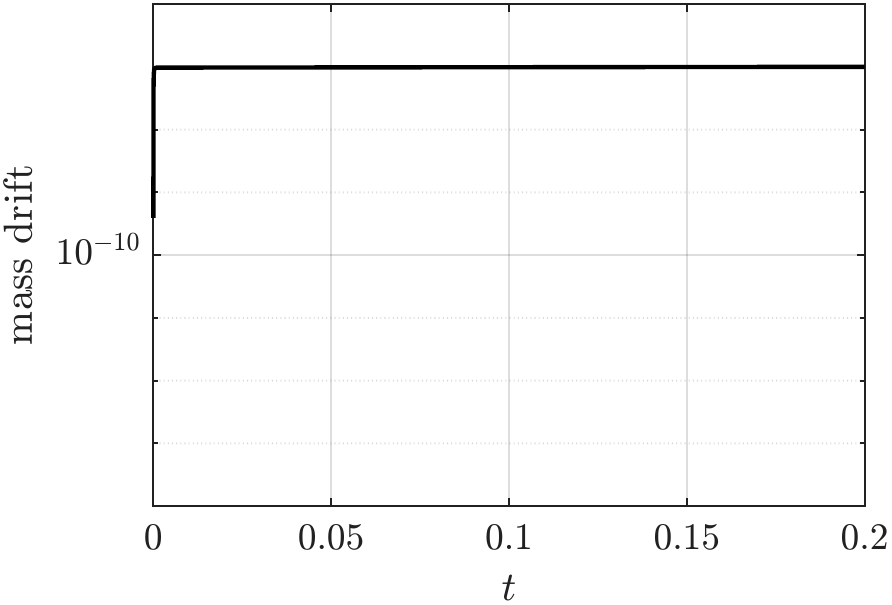}{Fluid mass drift.}\hfill
\plotsubfigure[0.43\textwidth]{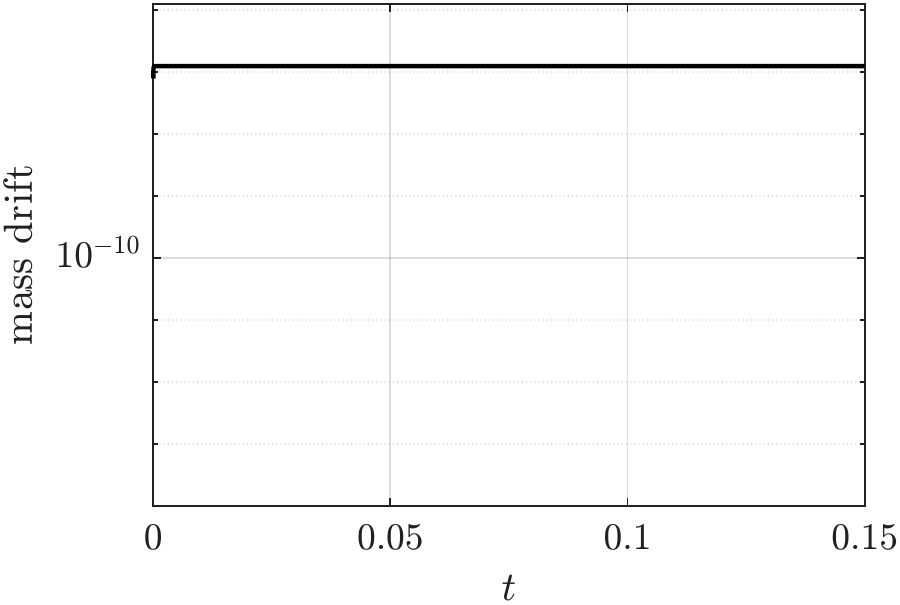}{Solid mass drift.}
\caption{\textbf{Circular explosion:} Final elementwise blending coefficients
and fully discrete mass-drift histories: (a)--(b) $\alpha_K$
and (c)--(d) logarithmic mass-drift histories for the fluid case at $t_f=0.2$ and the
elastic-solid case at $t_f=0.15$, respectively.}
\label{fig:circular-explosion-diagnostics}
\end{figure}

\subsection{Elastic solid rotor}
\label{sec:solid-rotor}

The elastic solid rotor is a genuinely two-dimensional benchmark for the
interaction of rotational motion with nonlinear elastic and thermal waves.
It was introduced for Eulerian nonlinear elasticity and has subsequently
been used to assess high-order GPR solvers and structure-preserving methods
\cite{PeshkovEtAl2019,Boscheri2024,Dumbser2026}.  The problem may be viewed as
an idealized bearing: an initially rotating circular core is instantaneously
joined to an initially stationary surrounding solid.  Shear, pressure, and
thermal waves then propagate through both regions.  The test is particularly
relevant here because the distortion tensor, the active thermal impulse, and
the global subcell limiter all participate in a fully two-dimensional flow.

We follow the periodic setup of Boscheri et al.~\cite{Boscheri2024} and
Dumbser et al.~\cite{Dumbser2026}.  On
$\Omega=[-1,1]^2$, let $r=(x^2+y^2)^{1/2}$ and $R=0.2$.  The initial data are
\begin{equation}
 (\rho,p,\vct)(x,y,0)=
 \begin{cases}
  (1,1,\omega_{\rm rot}(-y,x,0)^T),&r\le R,\\
  (1,1,\mathbf0),&r>R,
 \end{cases}
 \qquad
 \A(x,y,0)=I_3,\qquad \Jv(x,y,0)=\mathbf0,
 \label{eq:solid-rotor-data}
\end{equation}
with angular frequency $\omega_{\rm rot}=5$, so the initial tangential speed at the
rotor edge is $\omega_{\rm rot}R=1$.  Periodic conditions are imposed on all four
boundaries.  We use the ideal-gas member of
\eqref{eq:stiffened-gas-EOS}, with
\begin{equation}
 \rho_0=R_{\rm gas}=1,\qquad c_v=2.5,\qquad
 \gamma=1.4,\qquad p_\infty=0,\qquad c_s=c_h=1,
 \qquad \tau_1=\tau_2=10^{20}.
 \label{eq:solid-rotor-material}
\end{equation}
Thus both distortion and thermal relaxation are inactive on the simulated
time scale, while the elastic and thermal waves remain part of the
homogeneous GPR--$J$ dynamics.

The domain is partitioned into $400\times400$ square elements, and we use
polynomial degree $N=5$, CFL number $0.25$, $C_{\rm LF}=1$, and the
thermal-wave factor $\kappa_J=10$.  We take $C_{\rm OE}=0.02$.
The calculation is advanced to $t_f=0.3$.  For comparison, we compute an independent
second-order path-conservative MUSCL--Hancock solution of the complete
GPR--$J$ system on an $800\times800$ mesh
\cite{vanLeer1979,Toro2009,Pares2006}.  Both solutions are sampled on the
cut $y=0^+$.

Figure~\ref{fig:solid-rotor-profiles} compares density, transverse velocity,
distortion, and thermal impulse.  GC-DG agrees with the wave locations and
nested intermediate structures in the comparison solution.  Agreement in $A_{11}$ is
especially relevant to the determinant companion construction, whereas the
$J_1$ comparison tests the independently path-conservative thermal-impulse
block.  The corresponding two-dimensional fields in
Fig.~\ref{fig:solid-rotor-contours} retain the expected rotational structure
while resolving the spiral and annular wave pattern.  The final elementwise
blending coefficient $\alpha_K$ is concentrated along these wave fronts; large portions of
the smooth plateaus remain governed predominantly by the high-order
operator.  Figure~\ref{fig:solid-rotor-diagnostics} displays this map together
with the mass-drift history.  Because the boundary is periodic, the final normalized fully
discrete mass drift is $\varepsilon_M^{n_f}=1.97\times10^{-10}$, consistent
with the periodic semidiscrete companion balance and the fully discrete
discussion in Remark~\ref{rem:fully-discrete-mass-drift}.
The drift rises most rapidly during the initial rotor adjustment; after
$t=0.01$ it changes by only $8.16\times10^{-11}$, so its final level is set
mainly during that adjustment.

\begin{figure}[!t]
\centering
\plotsubfigure[0.47\textwidth]{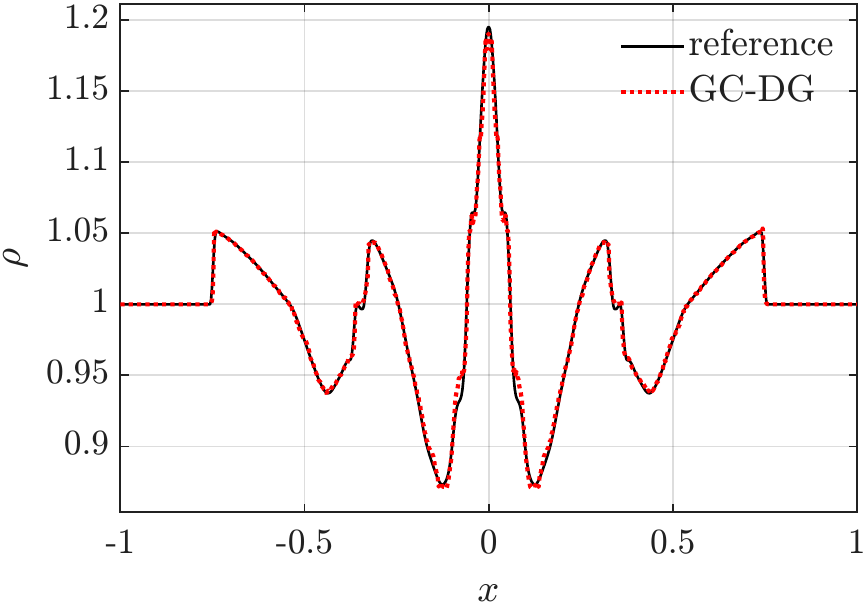}{Density $\rho$.}\hfill
\plotsubfigure[0.47\textwidth]{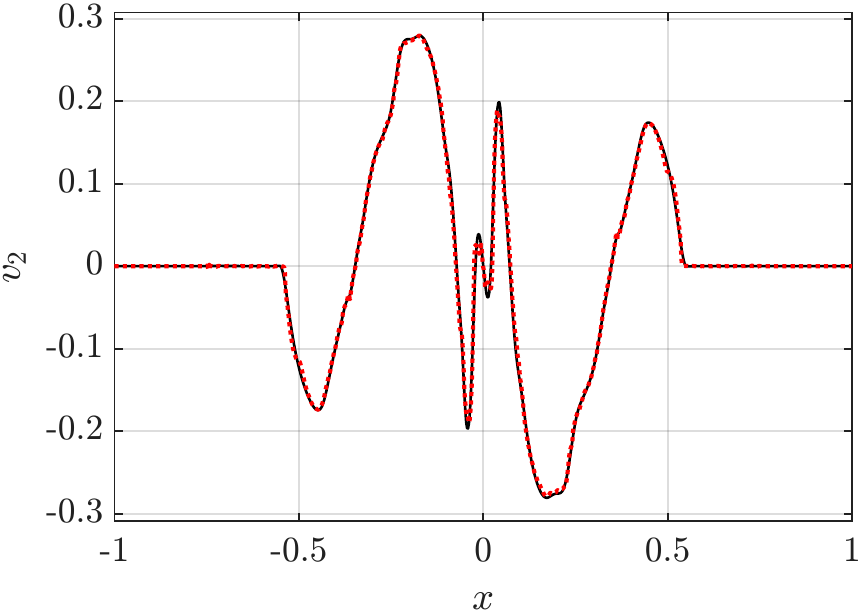}{Transverse velocity $v_2$.}
\par\medskip
\plotsubfigure[0.47\textwidth]{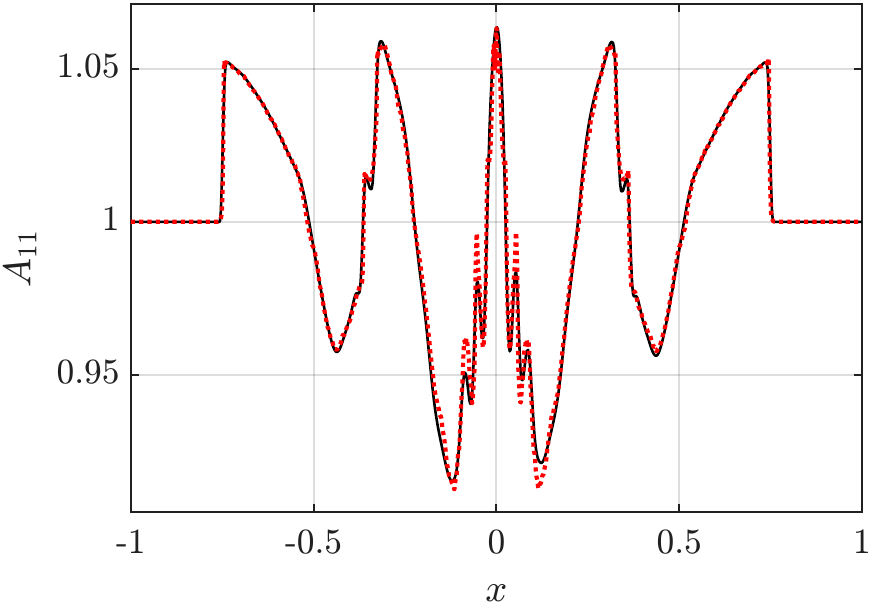}{Distortion component $A_{11}$.}\hfill
\plotsubfigure[0.47\textwidth]{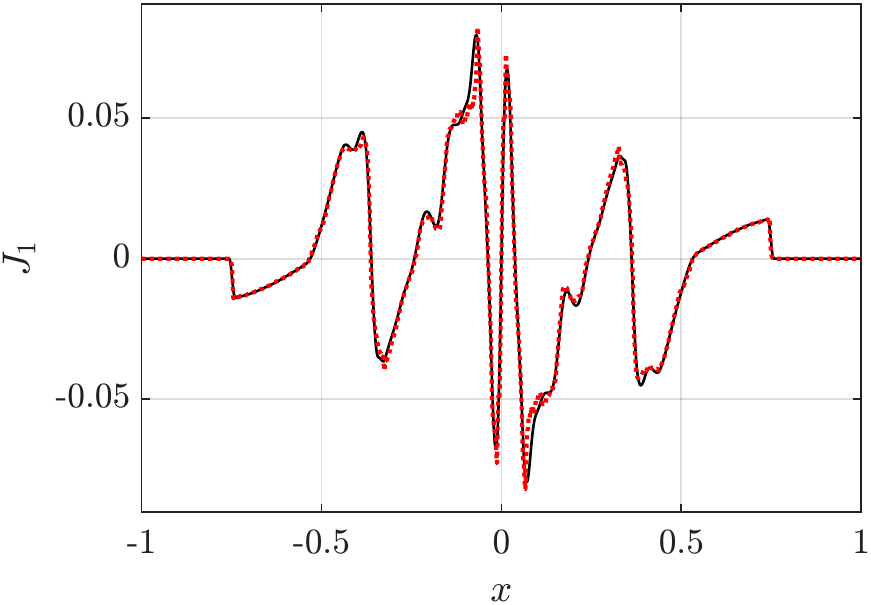}{Thermal impulse $J_1$.}
\caption{\textbf{Elastic solid rotor:} GC-DG profiles at $t_f=0.3$ on the
cut $y=0^+$, compared with the independent path-conservative
MUSCL--Hancock comparison solution.  The red dotted curves are the GC-DG
results, and the black solid curves are the comparison solution.  Panels show
(a) density, (b) transverse velocity, (c) distortion, and (d) thermal
impulse.}
\label{fig:solid-rotor-profiles}
\end{figure}

\begin{figure}[htbp]
\centering
\plotsubfigure[0.47\textwidth]{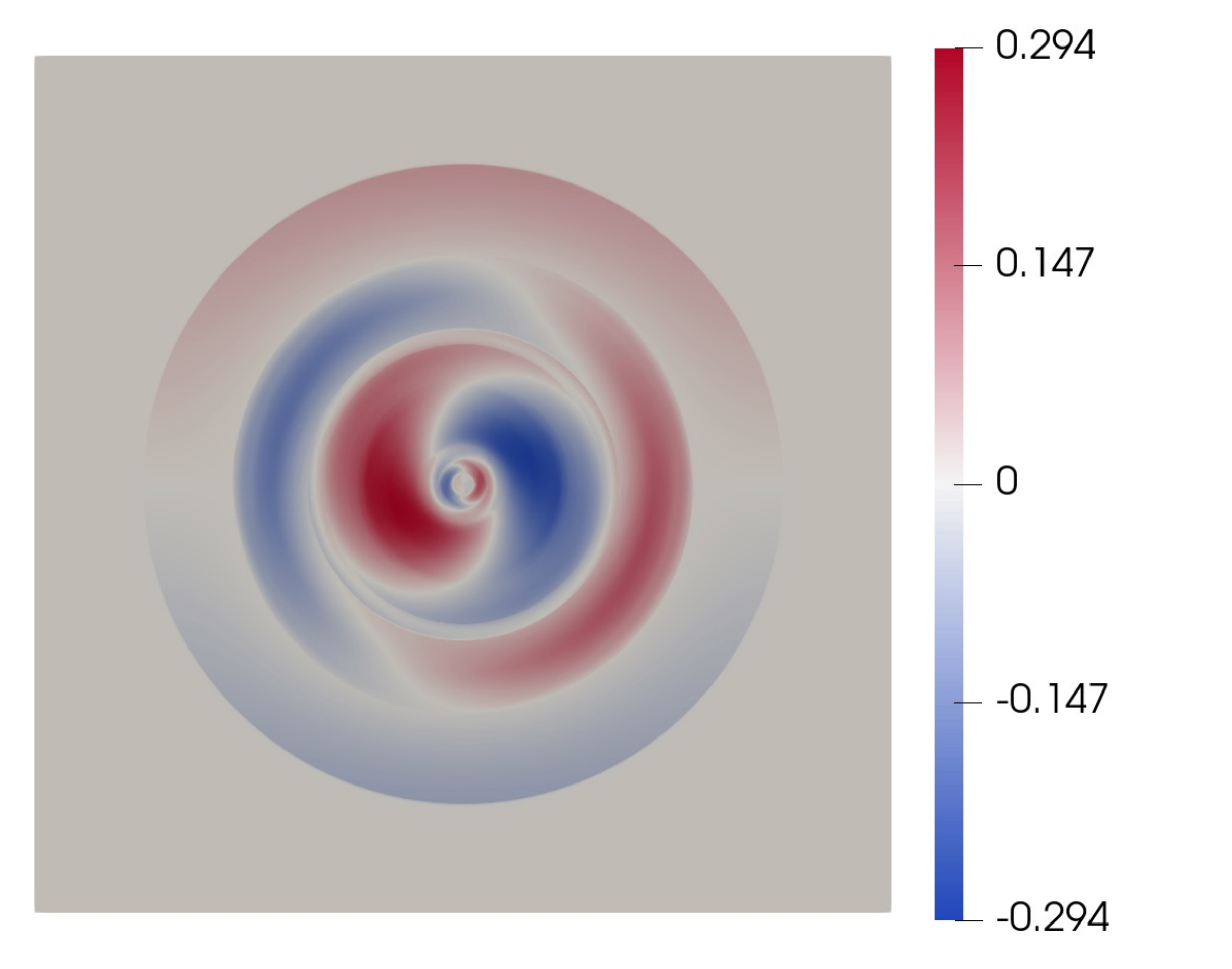}{Transverse velocity $v_2$.}\hfill
\plotsubfigure[0.47\textwidth]{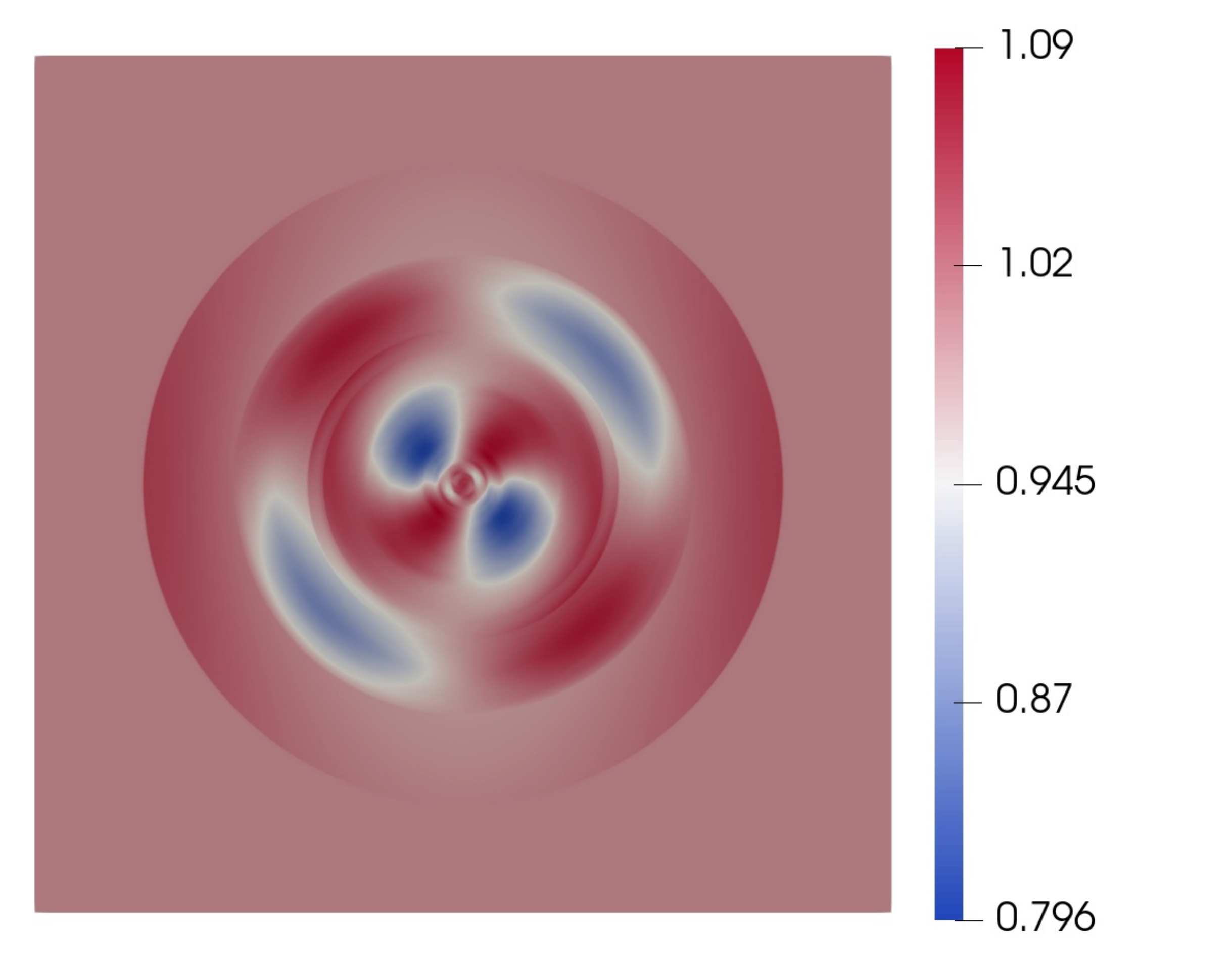}{Distortion component $A_{11}$.}
\par\medskip
\plotsubfigure[0.47\textwidth]{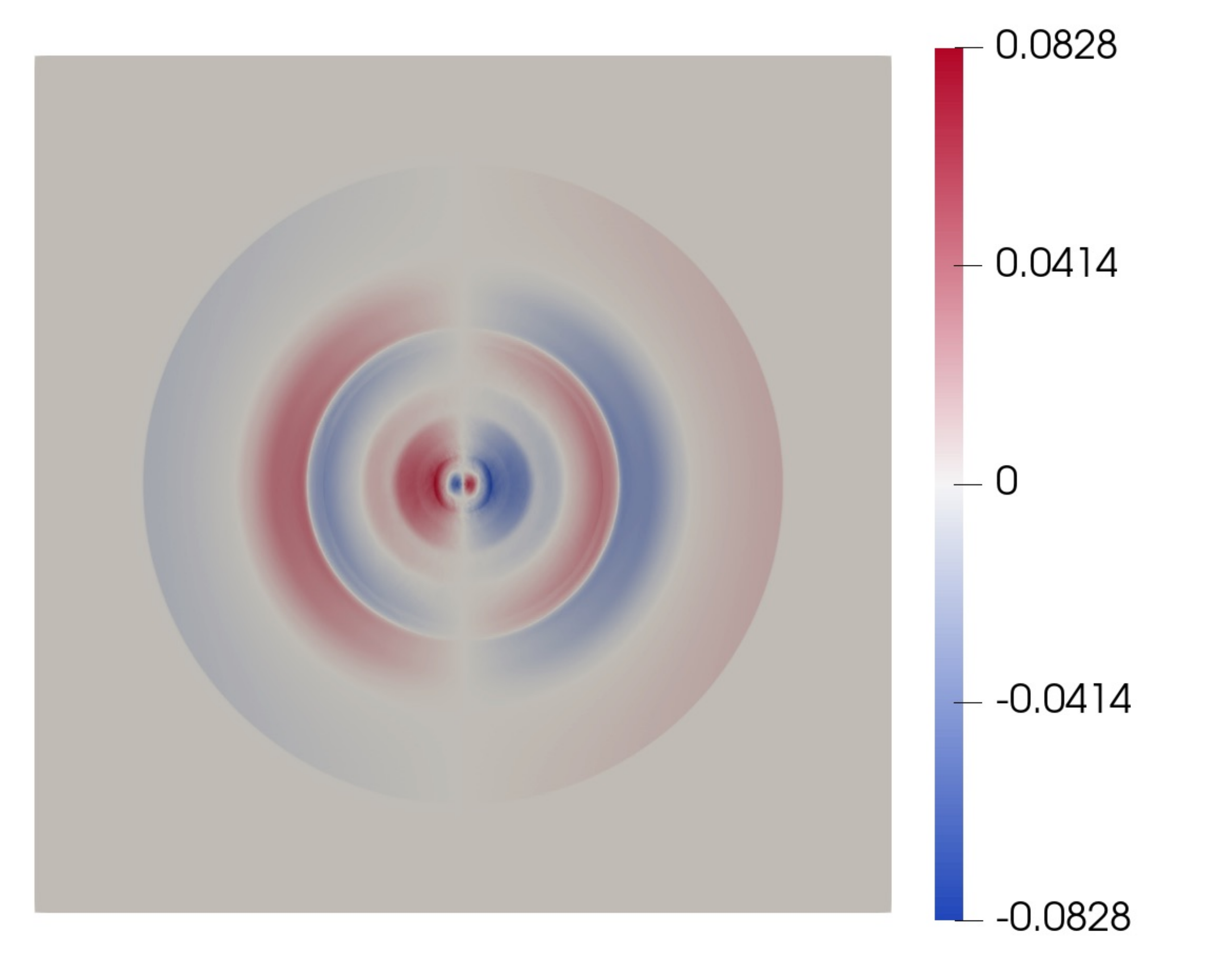}{Thermal impulse $J_1$.}
\caption{\textbf{Elastic solid rotor:} Final-time two-dimensional GC-DG
fields at $t_f=0.3$: (a) transverse velocity, (b) distortion, and
(c) thermal impulse.}
\label{fig:solid-rotor-contours}
\end{figure}

\begin{figure}[htbp]
\centering
\plotsubfigure[0.47\textwidth]{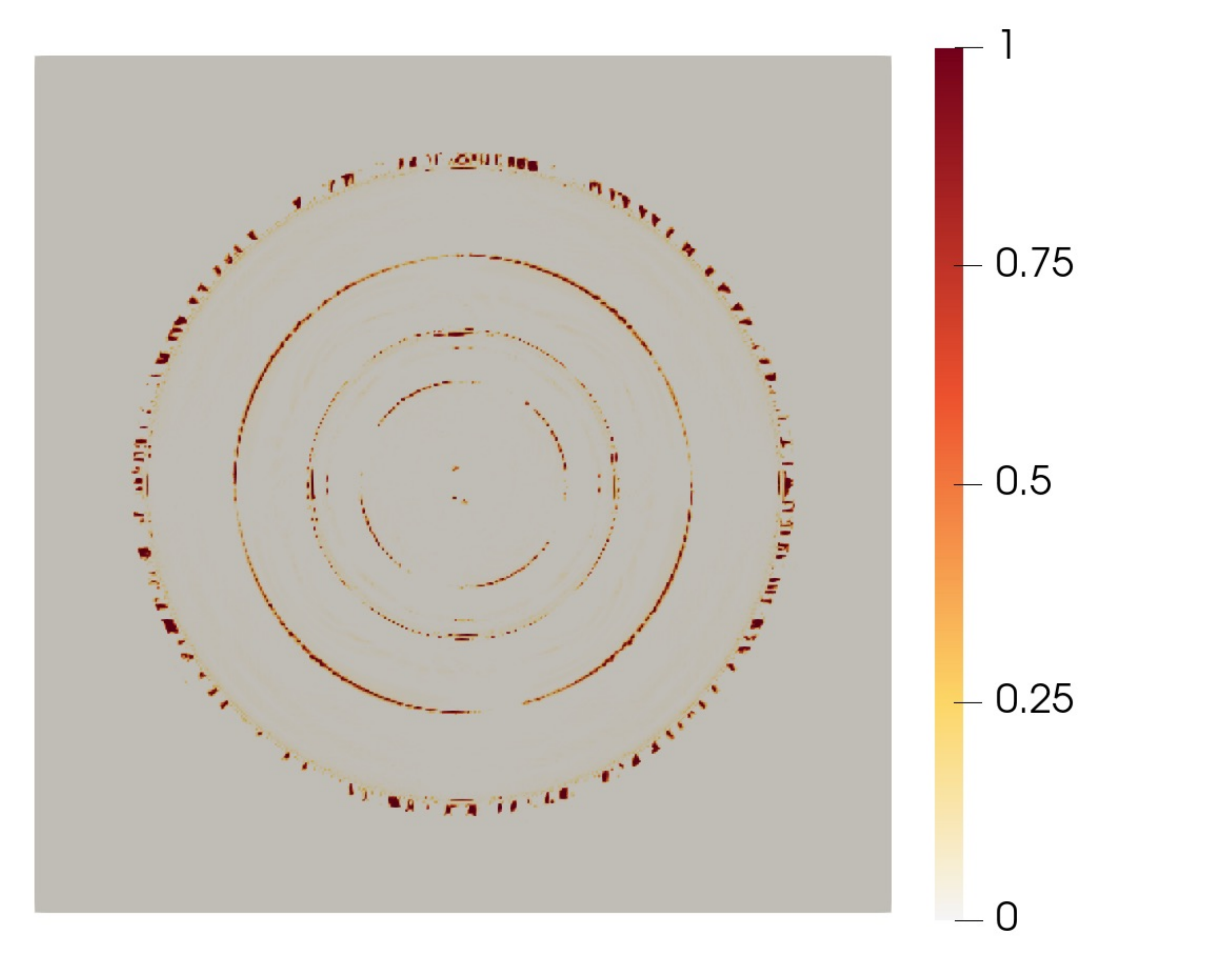}{Final elementwise blending coefficient $\alpha_K$.}\hfill
\plotsubfigure[0.47\textwidth]{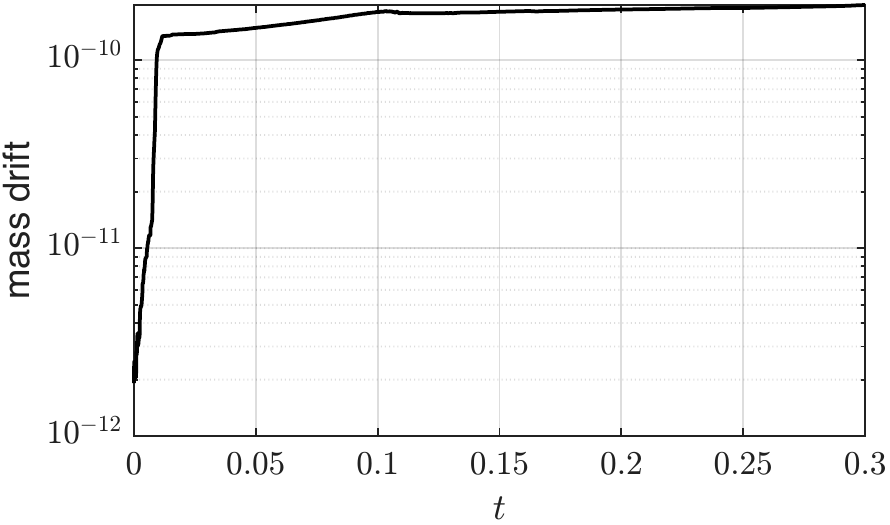}{Mass drift $\varepsilon_M^n$.}
\caption{\textbf{Elastic solid rotor:} Final elementwise blending coefficient
and fully discrete mass-drift history at $t_f=0.3$: (a) $\alpha_K$ and
(b) logarithmic mass drift.}
\label{fig:solid-rotor-diagnostics}
\end{figure}

\subsection{Cylindrical elastoplastic shock}
\label{sec:cylindrical-shock}

The cylindrical elastoplastic shock is a standard multidimensional copper
benchmark introduced by Barton et al.~\cite{Barton2010} and later computed
with the GPR model by Jackson and Nikiforakis
\cite{JacksonNikiforakis2019}.  A circular initial-state discontinuity separates an
inner region at ambient conditions from a hot, highly compressed exterior.
The initial radial imbalance produces coupled elastic and plastic waves.  In
addition to testing radial symmetry and shock resolution, this problem
places the geometric-compatible discretization in a general nonlinear
solid EOS with an active, stress-dependent distortion source.

We use the full square $\Omega=[-0.1,0.1]^2\,\mathrm m$ with transmissive
boundaries.  With $r=(x^2+y^2)^{1/2}$ and $R_0=0.02\,\mathrm m$, the material
is initially at rest and
\begin{equation}
 (p,T)(x,y,0)=
 \begin{cases}
  (0,300\,\mathrm K),&r\le R_0,\\
  (10^{10}\,\mathrm{Pa},600\,\mathrm K),&r>R_0.
 \end{cases}
 \label{eq:cylindrical-shock-data}
\end{equation}
For each state, the isotropic distortion is
$\A=\nu^{1/3}I_3$, where $\nu=\rho/\rho_0$ is determined from the
Barton--Romenski EOS at the prescribed $(p,T)$.  This gives
$\rho_{\rm in}=8900\,\mathrm{kg\,m^{-3}}$ and
$\rho_{\rm out}=9367.85\,\mathrm{kg\,m^{-3}}$, and enforces
$\rho=\rho_0\det\A$ initially.  The thermal-impulse equation is not evolved
in this test.

The Barton--Romenski energy
\eqref{eq:Barton-volumetric-energy}--\eqref{eq:Barton-shear-energy} uses
the copper parameters
\begin{equation}
 \begin{aligned}
 \rho_0&=8900\,\mathrm{kg\,m^{-3}},&
 c_{L,0}&=4651\,\mathrm{m\,s^{-1}},&
 b_0&=2141\,\mathrm{m\,s^{-1}},\\
 c_v&=390\,\mathrm{J\,kg^{-1}K^{-1}},&
 T_0&=300\,\mathrm K,&
 (\alpha_B,\beta_B,\gamma_B)&=(1,3,2).
 \end{aligned}
 \label{eq:cylindrical-shock-material}
\end{equation}
Plastic relaxation is modeled by
\begin{equation}
 \tau_1=\tau_0
 \left(\frac{\sigma_0}{\|\dev\bm\sigma\|_F}\right)^{n_p},\qquad
 \tau_0=3.13\times10^{-3}\,\mathrm s,\quad
 \sigma_0=4.5\times10^7\,\mathrm{Pa},\quad n_p=6.
 \label{eq:cylindrical-shock-relaxation}
\end{equation}
The EOS and yield-stress scale are those of
\cite{JacksonNikiforakis2019}.  The more severe relaxation law used there,
$(\tau_0,n_p)=(0.92\,\mathrm s,10.1)$, is softened here for an explicit
SSP--RK3 computation; the value in
\eqref{eq:cylindrical-shock-relaxation} is chosen so that the two laws agree
at $\|\dev\bm\sigma\|_F=4\sigma_0$.  Thus plastic relaxation remains active,
but the source stiffness is reduced without changing the underlying copper
energy or yield-stress scale.

We compute on a uniform $600\times600$ mesh with polynomial degree $N=4$
until $t_f=10^{-5}\,\mathrm s$.  SSP--RK3 uses hyperbolic CFL number $0.30$,
$C_{\rm LF}=1.5$, and the additional source restriction
\begin{equation}
 \Delta t\le
 \frac{C_{\rm src}}
 {\displaystyle\max_{K,i,j} 3(\det\A_{K,ij})^{5/3}(n_p+1)/\tau_{1,K,ij}},
 \qquad C_{\rm src}=0.25.
 \label{eq:cylindrical-shock-source-step}
\end{equation}
Here $\A_{K,ij}$ and $\tau_{1,K,ij}$ are the distortion and relaxation time
evaluated at LGL node $(i,j)$ of element $K$.  We take $C_{\rm OE}=0.02$.
For comparison, an independent radially
symmetric solution is computed on
$6000$ radial cells using the same constitutive law and relaxation source,
a characteristic fifth-order path-conservative WENO-Z discretization
\cite{Pares2006,JiangShu1996,Borges2008}, and SSP--RK3.

Figure~\ref{fig:cylindrical-shock-profiles} shows the radial density,
compressive deviatoric radial stress $-\sigma_{rr}$, and radial velocity.
Here $\bm\sigma$ is the flux-stress tensor in
\eqref{eq:full-GPR-stresses}, so $-\sigma_{rr}$ is the corresponding
deviatoric Cauchy radial stress under our sign convention.  The GC-DG profiles
closely follow the radial comparison solution through the principal elastic and
plastic fronts and reproduce the narrow stress structure generated near the
initial interface.  The density and velocity-magnitude maps in
Fig.~\ref{fig:cylindrical-shock-diagnostics} remain circular and free of
visible grid imprinting, despite the Cartesian mesh and strong thermodynamic
state jump.  The same figure shows that the final elementwise blending
coefficient $\alpha_K$ is confined to the physical wave regions rather than the
constant far field.  The final boundary-compensated fully discrete mass drift
is $\varepsilon_M^{n_f}=6.02\times10^{-11}$.  After
$t=10^{-7}\,\mathrm s$ it changes by only $1.77\times10^{-13}$, so its final
level is set almost entirely during the initial discontinuous transient.
These results demonstrate that the geometric
projection and global subcell stabilization remain robust for a general
solid energy and an active nonlinear relaxation mechanism; neither relies
on the ideal/stiffened-gas specialization used in several preceding tests.

\begin{figure}[htbp]
\centering
\plotsubfigure{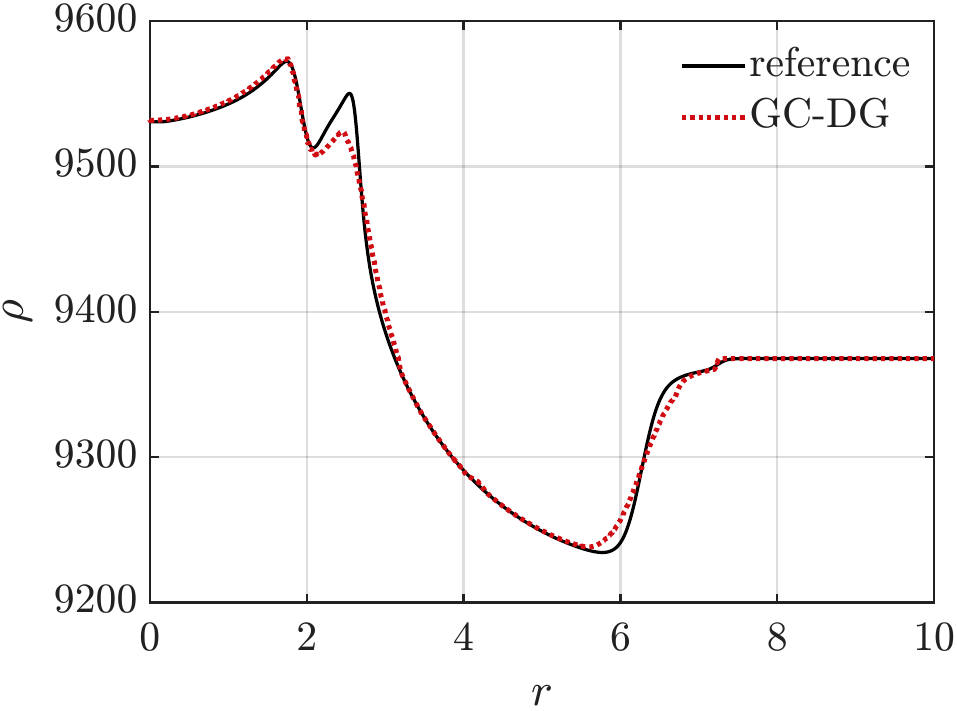}{Density $\rho$.}\hfill
\plotsubfigure{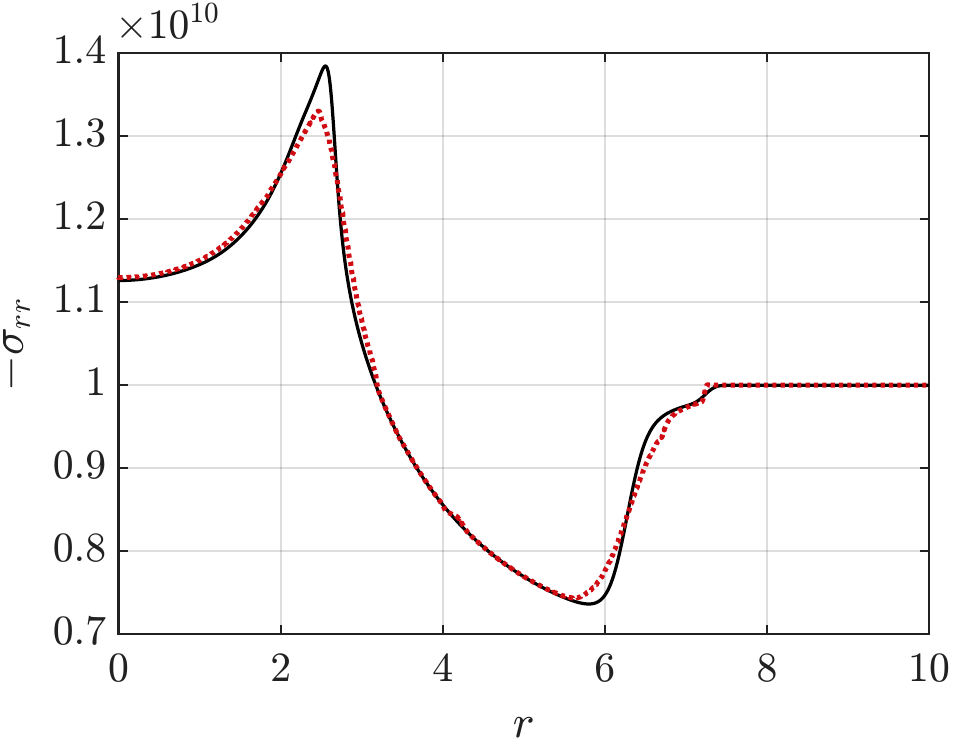}{Compressive deviatoric radial stress $-\sigma_{rr}$.}\hfill
\plotsubfigure{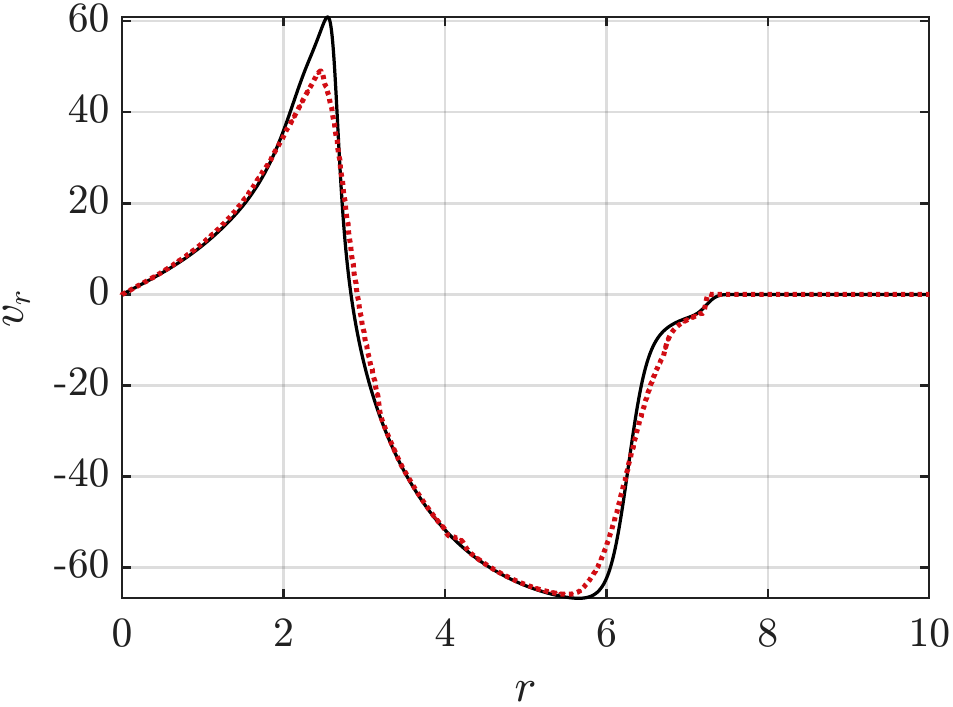}{Radial velocity $v_r$.}
\caption{\textbf{Cylindrical elastoplastic shock:} GC-DG profiles at
$t_f=10^{-5}\,\mathrm s$ compared with the characteristic radial
path-conservative WENO-Z comparison solution.  The red dotted curves are the
GC-DG results, and the black solid curves are the comparison solution.  Panels
show (a) density, (b) compressive deviatoric radial stress, and (c) radial
velocity.}
\label{fig:cylindrical-shock-profiles}
\end{figure}

\begin{figure}[H]
\centering
\plotsubfigure[0.47\textwidth]{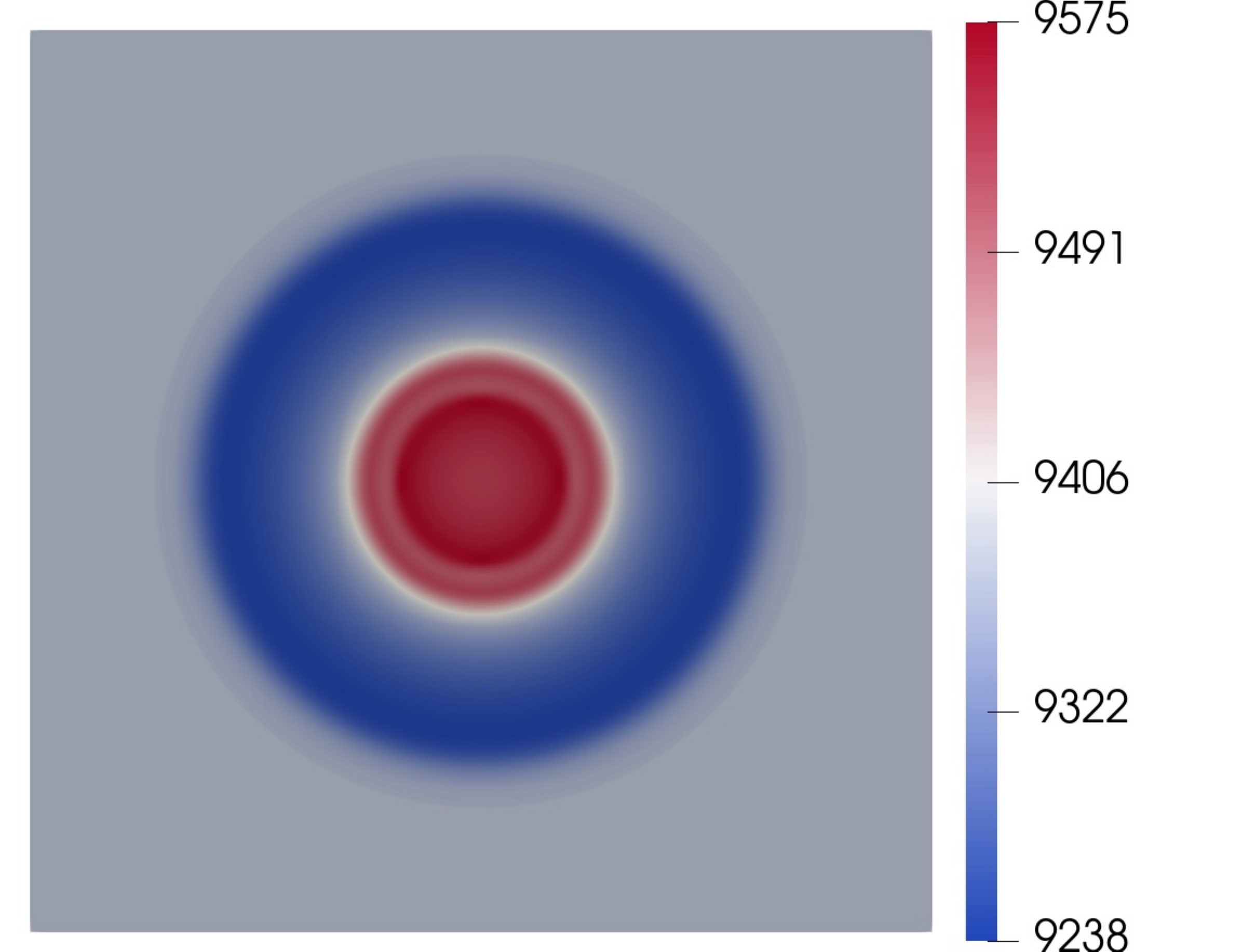}{Density $\rho$.}\hfill
\plotsubfigure[0.47\textwidth]{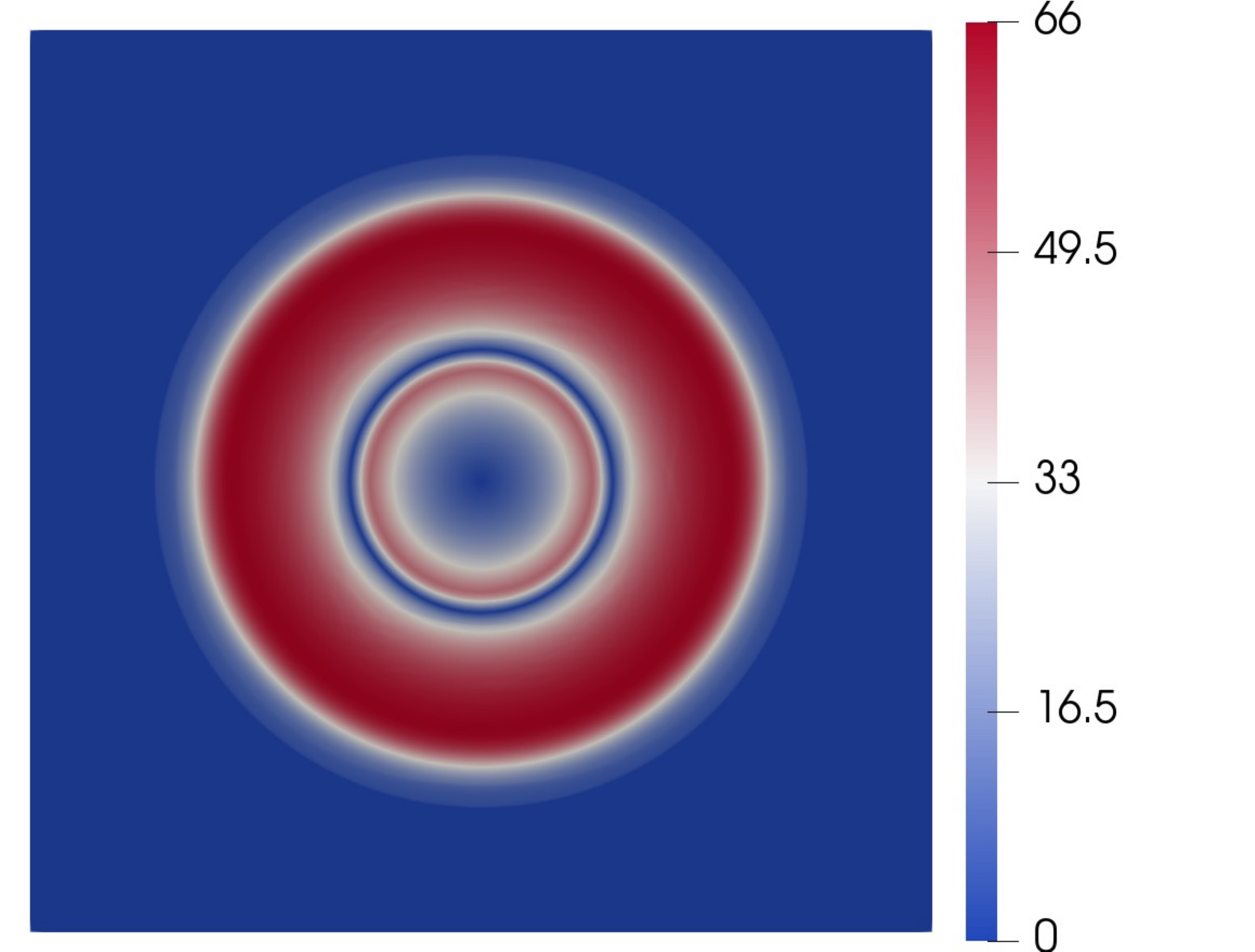}{Velocity magnitude $|\vct|$.}
\par\medskip
\plotsubfigure[0.47\textwidth]{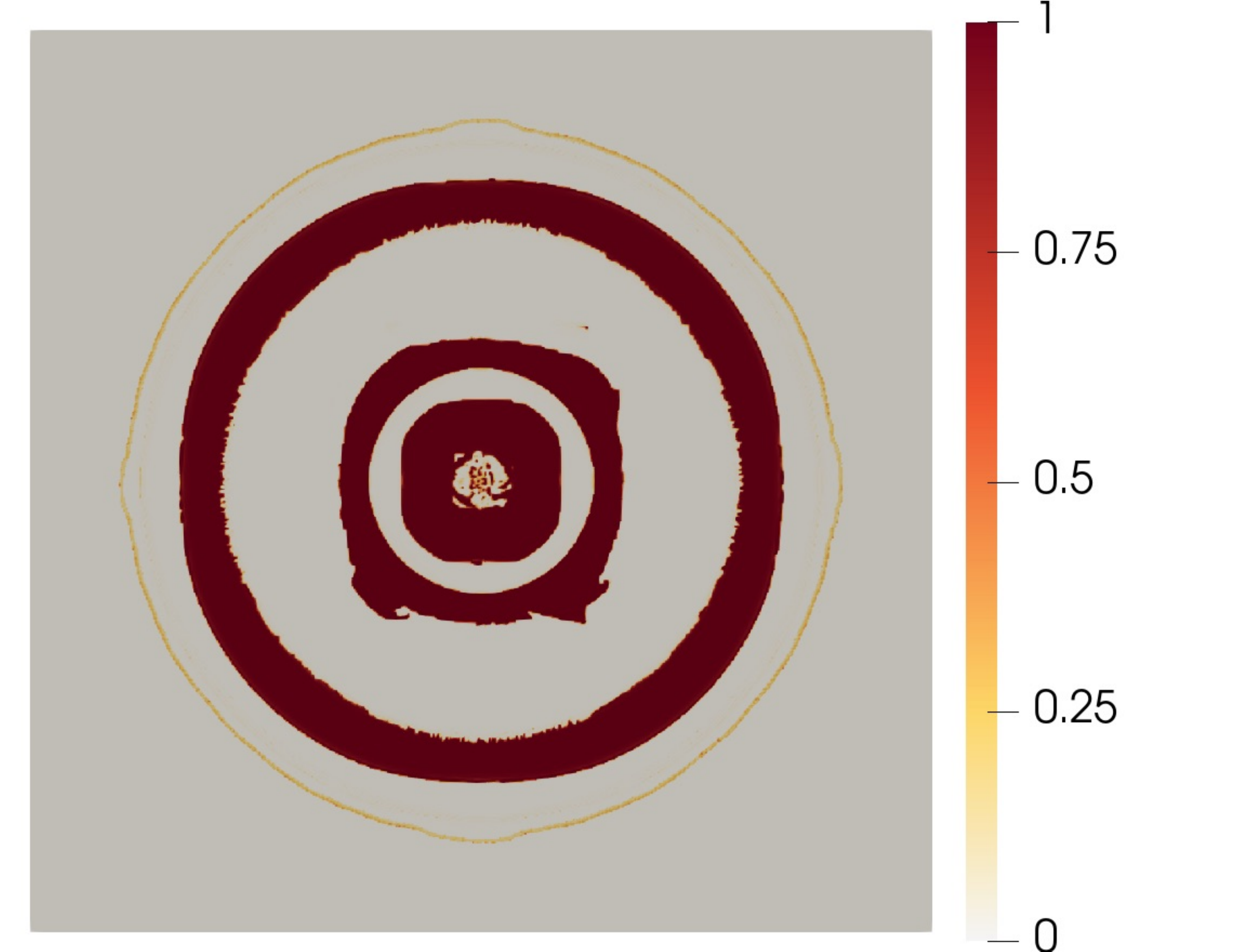}{Final elementwise blending coefficient $\alpha_K$.}\hfill
\plotsubfigure[0.47\textwidth]{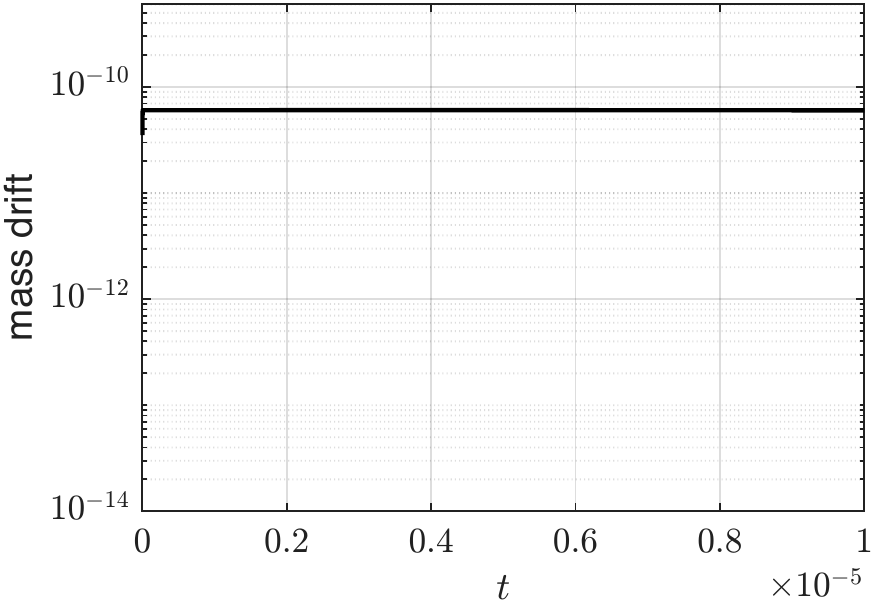}{Mass drift $\varepsilon_M^n$.}
\caption{\textbf{Cylindrical elastoplastic shock:} Final-time fields,
elementwise blending coefficient, and fully discrete mass-drift history at
$t_f=10^{-5}\,\mathrm s$: (a) density, (b) velocity magnitude,
(c) elementwise blending coefficient, and (d) logarithmic
boundary-compensated mass drift.}
\label{fig:cylindrical-shock-diagnostics}
\end{figure}

\section{Concluding remarks}
\label{sec:conclusions}

In this work, we developed an equation-of-state-independent
geometric-compatible finite-volume scheme (GC-FV) and its high-order DGSEM
extension (GC-DG) for the reduced GPR system on fixed Eulerian meshes.  The central construction
follows the same strategy as Tadmor's entropy-conservative flux
construction: we first derive a Tadmor-type geometric two-state identity
that the numerical distortion fluctuations must satisfy, then obtain a
fluctuation pair satisfying this identity through a minimum-change
projection of the path-conservative baseline, and finally lift the
two-state identity to high order using the SBP property.  Rewriting GC-DG
and GC-FV on the same LGL subcells allows them to be blended into a
geometric-compatible global subcell limiter, whose activation is controlled
by complementary modal and face-jump indicators.  The resulting method
preserves the hidden mass law at the semidiscrete level while providing the
dissipation required near discontinuities.  Numerical experiments confirm
the designed high-order accuracy, the expected temporal scaling of the
fully discrete mass drift, and robust shock capturing for fluid, nonlinear
elastic, thermal-wave, and elastoplastic problems.  Future work will
focus on less dissipative local subcell blending, positivity-preserving
stabilization, curvilinear and moving meshes, and the simultaneous
enforcement of geometric and thermodynamic compatibility.

\appendix

\section{Acoustic-tensor reduction for the mechanical wave speed}
\label{app:acoustic-tensor}

We record the derivation used in Section~\ref{sec:GC-FV-wave-speed}.  For
the mechanical energy
\begin{equation}
 e(\rho,s,\A)=e_h(\rho,s)+\frac14b^2(\rho)\Psi,
 \qquad
 \zeta_b:=\frac{\dd\log b^2}{\dd\log\rho},
 \qquad
 c_0^2:=\left.\frac{\partial p_h}{\partial\rho}\right|_s,
 \label{eq:app-acoustic-energy-class}
\end{equation}
where $p_h=\rho^2\partial_\rho e_h|_s$, differentiation of the shear
energy gives
\begin{equation}
 \partial_{\A}e_s=b^2\A D,
 \qquad
 \bm\sigma=\rho b^2GD,
 \qquad
 p_s=\rho^2\partial_\rho e_s
 =\frac14\rho\zeta_b b^2\Psi.
 \label{eq:app-acoustic-stress}
\end{equation}
Thus a density-dependent shear speed contributes to both the isotropic
pressure and the shear traction.

Freeze an admissible state and insert an isentropic plane perturbation
($\delta s=0$) propagating in the unit direction $n$.  For the two
mechanical closures used in this paper, the logarithmic density exponent
$\zeta_b=\rho\,\partial_\rho b^2/b^2$ is constant
($\zeta_b=0$ for the constant-shear gas closure and
$\zeta_b=\beta_B$ for the Barton--Romenski closure).  After eliminating the distortion amplitude with
the linearized kinematic equation, its velocity polarization $\xi$ induces
the following directional variations:
\begin{align}
 \Sigma_G&:=GD, & q&:=G\xi, & \chi&:=n\cdot\xi,
 \notag\\
 \delta G&:=n\otimes q+q\otimes n,
 &\delta D&:=\delta G-\frac13\tr(\delta G)I_3,
 &\delta\Sigma_G&:=\delta G D+G\delta D,
 &\delta\Psi&:=2D:\delta G.
 \label{eq:app-acoustic-increments}
\end{align}
The first two variations follow directly from $G=\A^T\A$ and
$D=\dev G$; the last two are the product and chain rules applied to
$\Sigma_G=GD$ and $\Psi=D:D$.  Moreover,
\begin{equation}
 \delta b^2=\zeta_b b^2\frac{\delta\rho}{\rho},
 \qquad
 \delta p_h=c_0^2\delta\rho,\qquad \delta s=0.
 \label{eq:app-acoustic-scalar-variations}
\end{equation}
Substitution of
\eqref{eq:app-acoustic-increments}--\eqref{eq:app-acoustic-scalar-variations}
into the linearized normal momentum flux gives the acoustic-tensor action
\begin{equation}
 W_n\xi=
 \left[c_0^2\chi+
 \frac{\zeta_b b^2}{4}
 \bigl((1+\zeta_b)\chi\Psi+\delta\Psi\bigr)\right]n
 +b^2\bigl((1+\zeta_b)\chi\Sigma_G+\delta\Sigma_G\bigr)n .
 \label{eq:app-mechanical-acoustic-action}
\end{equation}
The terms proportional to $1+\zeta_b$ collect the density variation of
$\rho b^2$, whereas $\delta\Psi$ and $\delta\Sigma_G$ are the direct
distortion variations.  Taking $b^2=c_s^2$, $\zeta_b=0$ gives
\eqref{eq:acoustic-ideal-action}; taking
$b^2=b_0^2\nu^{\beta_B}$, $\zeta_b=\beta_B$ gives
\eqref{eq:acoustic-Barton-action}.

For a hyperelastic closure, symmetry of the energy Hessian makes the
restriction of $W_n$ to the computational plane symmetric.  The frozen
normal mechanical Jacobian $A_n^{\rm mech}$ then has the factorization
\begin{equation}
 \det(\lambda I-A_n^{\rm mech})
 \propto(\lambda-v_n)^r
 \det\!\left(W_n-(\lambda-v_n)^2I_2\right),
 \label{eq:app-acoustic-characteristic-factorization}
\end{equation}
where $r$ is the multiplicity of the advective speed.  Hence the two
eigenvalues of $W_n$ are the squared in-plane mechanical wave speeds
relative to $v_n$.  This is an exact reduction of the frozen homogeneous
mechanical Jacobian; the linearization is precisely the usual local
characteristic analysis of a nonlinear hyperbolic system.  Evaluating the
larger root of the quadratic characteristic polynomial of $W_n$ yields
\eqref{eq:acoustic-largest-eigenvalue}.

\section{Algebraic properties and proofs for GC-FV and GC-DG}
\label{app:G-proofs}

\subsection{Minimum-change Frobenius projection}
\label{app:Frobenius-projection}

The distortion projections in
\eqref{eq:G-volume-optimization}--\eqref{eq:G-volume-constraints} and
\eqref{eq:PLF-optimization}--\eqref{eq:PLF-optimization-admissible-set}
use the same one-constraint algebraic result, which we state and prove
here.

\begin{lemma}[Minimum-change Frobenius projection]
\label{lem:app-single-projection}
Let $C\ne0$ and $X_0$ be matrices of the same size, and let $b$ be a
scalar.  The constrained problem
\begin{equation}
 \min_{C:X=b}\frac12\|X-X_0\|_F^2
 \label{eq:app-single-projection}
\end{equation}
has the unique minimizer
\begin{equation}
 X^\star=X_0+\frac{b-C:X_0}{C:C}\,C.
 \label{eq:app-single-projection-solution}
\end{equation}
\end{lemma}

\begin{proof}
Introduce the Lagrangian
\begin{equation}
 \mathcal L_{\rm proj}(X,\mu)
 :=\frac12\|X-X_0\|_F^2+\mu(C:X-b).
 \label{eq:app-projection-Lagrangian}
\end{equation}
Stationarity with respect to $X$ gives
$X-X_0+\mu C=0$.  Contracting this relation with $C$ and imposing
$C:X=b$ yield
\begin{equation}
 \mu=\frac{C:X_0-b}{C:C},
 \qquad
 X=X_0+\frac{b-C:X_0}{C:C}\,C=X^\star.
 \label{eq:app-projection-stationarity}
\end{equation}
Because $C\ne0$, the affine constraint set is nonempty and $C:C>0$.
Moreover, for any other feasible matrix $X=X^\star+Z$, one has $C:Z=0$.
Since $X^\star-X_0$ is parallel to $C$, the Frobenius orthogonality gives
\begin{equation}
 \|X-X_0\|_F^2
 =\|X^\star-X_0\|_F^2+\|Z\|_F^2
 \ge \|X^\star-X_0\|_F^2,
 \label{eq:app-projection-Pythagoras}
\end{equation}
with equality only for $Z=0$.  Thus $X^\star$ is the unique minimizer.
\end{proof}

Applying Lemma~\ref{lem:app-single-projection} independently to the endpoint
constraints gives
the GC-FV pair \eqref{eq:G-volume-pair} and the projected LF increments
\eqref{eq:PLF}.  Since admissibility implies
$\w_a\ne0$, these denominators are endpoint quantities, not vanishing
jump norms.  By contrast, enforcing a zero-sum correction to the cubic
defect produces $C=\Delta\w$ and the degenerate denominator
$\|\Delta\w\|_F^2$.

\subsection{Exchange and face-orientation identities}
\label{app:GC-exchange-identities}
Two distinct orientation properties are used in the main text.  The
fixed-normal exchange property supports
Lemma~\ref{lem:G-core-companion} and the SBP volume pair
\eqref{eq:G-full-pair}; simultaneous state exchange and normal reversal
support the element-face identity \eqref{eq:GC-FV-orientation}.  They are
stated together below but proved separately.

\begin{lemma}[Exchange and orientation properties of the GC-FV pairs]
\label{lem:app-GC-exchange-orientation}
For admissible states $U_L,U_R$ and a unit normal $n$, the nondissipative
volume pair \eqref{eq:G-full-pair} satisfies the fixed-normal exchange
identity
\begin{equation}
 \mathcal K^{-,G}(U_R,U_L;n)
 =-\mathcal K^{+,G}(U_L,U_R;n).
 \label{eq:app-G-volume-exchange}
\end{equation}
The complete dissipative face pair
\eqref{eq:GC-FV-left}--\eqref{eq:GC-FV-right} satisfies the orientation
identity
\begin{equation}
 \widehat{\mathcal K}_G^-(U_R,U_L;-n)
 =\widehat{\mathcal K}_G^+(U_L,U_R;n).
 \label{eq:app-G-face-orientation}
\end{equation}
\end{lemma}

\begin{proof}
First exchange $U_L$ and $U_R$ while keeping $n$ fixed.  Then
$\Delta\A\mapsto-\Delta\A$, $\Delta\vct\mapsto-\Delta\vct$,
$\Delta\Jv\mapsto-\Delta\Jv$, and $\Delta T\mapsto-\Delta T$, whereas
the arithmetic averages remain unchanged.  The straight-path pairs
\eqref{eq:G-baseline-minus}--\eqref{eq:G-baseline-plus} and
\eqref{eq:G-J-baseline-minus}--\eqref{eq:G-J-baseline-plus} consequently
satisfy
\begin{equation}
 \begin{aligned}
 \mathcal H_{A,0}^-(U_R,U_L;n)
 &=-\mathcal H_{A,0}^+(U_L,U_R;n),\\
 \mathcal H_{J,0}^-(U_R,U_L;n)
 &=-\mathcal H_{J,0}^+(U_L,U_R;n).
 \end{aligned}
 \label{eq:app-G-baseline-exchange}
\end{equation}
Moreover, $F_Y^0$ and $h_G^t$ are symmetric in their two state arguments,
and the endpoint projection data obey
\begin{equation}
 R_L^G(U_R,U_L;n)=-R_R^G(U_L,U_R;n),
 \qquad P_L(U_R,U_L)=P_R(U_L,U_R).
 \label{eq:app-G-exchange-data}
\end{equation}
Substitution into \eqref{eq:G-volume-pair} proves
\begin{equation}
 \mathcal H_A^{-,G}(U_R,U_L;n)
 =-\mathcal H_A^{+,G}(U_L,U_R;n),
 \label{eq:app-G-projected-exchange}
\end{equation}
which is \eqref{eq:G-core-lemma-exchange}.  Combining the conservative,
distortion, and thermal-impulse blocks proves
\eqref{eq:app-G-volume-exchange}.

Now exchange the states and simultaneously replace $n$ by $-n$.
Every nondissipative normal flux and fluctuation is odd in $n$; hence
\eqref{eq:app-G-volume-exchange} gives
\begin{equation}
 \mathcal K^{-,G}(U_R,U_L;-n)
 =\mathcal K^{+,G}(U_L,U_R;n).
 \label{eq:app-G-volume-orientation}
\end{equation}
The LF coefficient is invariant under this transformation, while
$\Delta Y\mapsto-\Delta Y$ and $\Delta\Jv\mapsto-\Delta\Jv$.  Therefore
$D_Y^-(U_R,U_L;-n)=D_Y^+(U_L,U_R;n)$ and
$D_J^-(U_R,U_L;-n)=D_J^+(U_L,U_R;n)$.  For the distortion correction,
\begin{equation}
 \Delta\A\mapsto-\Delta\A,\qquad
 a_{\Delta A}\mapsto a_{\Delta A},\qquad
 d_{\Delta A}\mapsto-d_{\Delta A},\qquad
 P_L\leftrightarrow P_R,\qquad \gamma_L\mapsto-\gamma_R.
 \label{eq:app-G-LF-exchange-data}
\end{equation}
Using \eqref{eq:PLF} therefore gives
\begin{equation}
 D_A^-(U_R,U_L;-n)
 =-\ell_f(-\Delta\A-\gamma_RP_R)
 =D_A^+(U_L,U_R;n),
 \label{eq:app-G-LF-orientation-proof}
\end{equation}
which also proves \eqref{eq:PLF-lemma-orientation}.  Adding these three LF
blocks to \eqref{eq:app-G-volume-orientation} proves
\eqref{eq:app-G-face-orientation}, equivalently
\eqref{eq:GC-FV-orientation}.
\end{proof}

\subsection{Smooth-state size of the projection and path-sum defects}
\label{app:GC-path-estimates}
Remark~\ref{rem:GC-FV-Boscheri-comparison} identifies two departures from
the straight-path fluctuation sum: one introduced by the nondissipative
companion projection and one introduced by the projected LF correction.
The purpose of this subsection is to justify the smooth-state orders quoted
there.  These are local asymptotic estimates as $U_R\to U_L$; they do not
provide a bound for a finite jump at a shock.

Let $\delta:=\|\Delta U\|=O(h)$ for two neighboring states sampled from a
smooth solution.  On a compact admissible state set, smoothness of the
determinant and velocity gives
$\Delta g=O(\delta)$, $\Delta v_n=O(\delta)$,
$a_{\Delta A}=O(\delta^2)$, and
$d_{\Delta A}=O(\delta^3)$.  Substitution into the closed
forms \eqref{eq:G-projection-residuals-closed} first gives
\begin{equation}
 R_L^G=O(\delta^2),\qquad R_R^G=O(\delta^2).
 \label{eq:app-G-projection-order}
\end{equation}
Thus each endpoint modification in \eqref{eq:G-volume-pair} vanishes
quadratically at the equal-state diagonal.  For the path-sum defect, the
leading quadratic contributions cancel.  Indeed, write
$\overline P=(P_L+P_R)/2$ and $\Delta P=P_R-P_L$.  From
\eqref{eq:G-projection-residuals} and \eqref{eq:G-baseline-defect},
\begin{equation}
 R_L^G+R_R^G=h_{0,R}-h_{0,L}
 =-\frac12\overline v_n d_{\Delta A}=O(\delta^3),
 \qquad R_R^G-R_L^G=O(\delta^2).
 \label{eq:app-G-path-cancellation}
\end{equation}
Since $P_a$ is smooth on compact admissible sets,
$\Delta P=O(\delta)$, and therefore
\begin{equation}
 \mathcal P_A^{V,G}
 =(R_L^G+R_R^G)\overline P
 +\frac12(R_R^G-R_L^G)\Delta P
 =O(\delta^3).
 \label{eq:app-G-cubic-path-defect}
\end{equation}
For the dissipative correction, $P_L+P_R=O(1)$ and
$P_R-P_L=O(\delta)$.  Equation~\eqref{eq:PLF-path-defect} consequently
gives
\begin{equation}
 \mathcal P_A^{D,G}
 =O(\ell_f\delta^2)+O(\ell_f\delta^4)
 =O(\ell_f\delta^2).
 \label{eq:app-G-dissipative-path-order}
\end{equation}
Hence the nondissipative and dissipative path-sum defects vanish,
respectively, to cubic and quadratic order for smooth neighboring states,
which proves the estimates stated in
Remark~\ref{rem:GC-FV-Boscheri-comparison}.  At a discontinuity
$\delta$ is not small, so neither conclusion should be interpreted as a
uniform shock estimate.

\subsection{Smooth equal-state regularity of the companion projection}
\label{app:GC-DG-proof}
This subsection establishes the smooth, direction-independent equal-state
extension asserted in Lemma~\ref{lem:G-core-companion} and explains why the
companion projection is consistent with the high-order construction.  The
key point is that the projection is nonsingular and changes neither the
diagonal value nor the first diagonal derivative of the straight-path
baseline.

Because $P_a=\w_a/(\w_a:\w_a)$ and $\w_a\ne0$ whenever $g_a>0$, the
endpoint projector is smooth and bounded on every compact admissible set
bounded away from $g=0$.  Let $U_L=U$ and
$U_R=U+\tau V$.  The estimate
\eqref{eq:app-G-projection-order} gives
$R_L^GP_L=O(\tau^2)$ and
$R_R^GP_R=O(\tau^2)$.  Therefore both corrections converge to zero
independently of the direction $V$, which establishes the smooth
equal-state extension of \eqref{eq:G-volume-pair}.  Moreover,
\begin{equation}
 \left.\frac{\dd}{\dd\tau}(R_L^GP_L)
 \right|_{\tau=0}
 =\left.\frac{\dd}{\dd\tau}(R_R^GP_R)
 \right|_{\tau=0}=0.
 \label{eq:app-G-zero-diagonal-derivative}
\end{equation}
The projected pair thus has the same first diagonal linearization as the
consistent straight-path pair.  Consequently, the projection has no jump
denominator, remains bounded and smooth as the two states coalesce, and
adds only an $O(\tau^2)$ perturbation to the baseline.  In this
precise local sense, it is consistent with the high-order SBP construction
and does not by itself introduce a singular or leading-order correction
that would reduce the spatial accuracy.


\begin{thebibliography}{99}

\bibitem{PeshkovRomenski2016}
I.~Peshkov and E.~Romenski,
\newblock A hyperbolic model for viscous Newtonian flows,
\newblock \emph{Continuum Mechanics and Thermodynamics}, 28 (2016), 85--104.

\bibitem{Dumbser2016}
M.~Dumbser, I.~Peshkov, E.~Romenski, and O.~Zanotti,
\newblock High order ADER schemes for a unified first order hyperbolic
formulation of continuum mechanics: viscous heat-conducting fluids and
elastic solids,
\newblock \emph{Journal of Computational Physics}, 314 (2016), 824--862.

\bibitem{Peshkov2018}
I.~Peshkov, M.~Pavelka, E.~Romenski, and M.~Grmela,
\newblock Continuum mechanics and thermodynamics in the Hamilton and the
Godunov-type formulations,
\newblock \emph{Continuum Mechanics and Thermodynamics}, 30 (2018),
1343--1378.

\bibitem{Jackson2017}
H.~Jackson,
\newblock A fast numerical scheme for the Godunov--Peshkov--Romenski model
of continuum mechanics,
\newblock \emph{Journal of Computational Physics}, 348 (2017), 514--533.

\bibitem{JacksonNikiforakis2019}
H.~Jackson and N.~Nikiforakis,
\newblock A numerical scheme for non-Newtonian fluids and plastic solids
under the GPR model,
\newblock \emph{Journal of Computational Physics}, 387 (2019), 410--429.

\bibitem{PeshkovEtAl2019}
I.~Peshkov, W.~Boscheri, R.~Loub\`ere, E.~Romenski, and M.~Dumbser,
\newblock Theoretical and numerical comparison of hyperelastic and
hypoelastic formulations for Eulerian non-linear elastoplasticity,
\newblock \emph{Journal of Computational Physics}, 387 (2019), 481--521.

\bibitem{Tavelli2020}
M.~Tavelli, S.~Chiocchetti, E.~Romenski, A.-A.~Gabriel, and M.~Dumbser,
\newblock Space--time adaptive ADER discontinuous Galerkin schemes for
nonlinear hyperelasticity with material failure,
\newblock \emph{Journal of Computational Physics}, 422 (2020), 109758.

\bibitem{Gabriel2021}
A.-A.~Gabriel, D.~Li, S.~Chiocchetti, M.~Tavelli, I.~Peshkov,
E.~Romenski, and M.~Dumbser,
\newblock A unified first-order hyperbolic model for nonlinear dynamic
rupture processes in diffuse fracture zones,
\newblock \emph{Philosophical Transactions of the Royal Society A},
379 (2021), 20200130.

\bibitem{Busto2022}
S.~Busto, M.~Dumbser, I.~Peshkov, and E.~Romenski,
\newblock On thermodynamically compatible finite volume schemes for
continuum mechanics,
\newblock \emph{SIAM Journal on Scientific Computing}, 44 (2022),
A1723--A1751.

\bibitem{BustoDumbser2022DG}
S.~Busto and M.~Dumbser,
\newblock A new family of thermodynamically compatible discontinuous
Galerkin methods for continuum mechanics and turbulent shallow water flows,
\newblock \emph{Journal of Scientific Computing}, 93 (2022), 56.

\bibitem{Abgrall2023}
R.~Abgrall, S.~Busto, and M.~Dumbser,
\newblock A simple and general framework for the construction of
thermodynamically compatible schemes for computational fluid and solid
mechanics,
\newblock \emph{Applied Mathematics and Computation}, 440 (2023), 127629.

\bibitem{Bauer2006}
A.~Bauer, D.~E.~Burton, E.~Caramana, R.~Loub\`ere, M.~Shashkov, and
P.~Whalen,
\newblock The internal consistency, stability, and accuracy of the discrete,
compatible formulation of Lagrangian hydrodynamics,
\newblock \emph{Journal of Computational Physics}, 218 (2006), 572--593.

\bibitem{Maire2009}
P.-H.~Maire,
\newblock A high-order cell-centered Lagrangian scheme for two-dimensional
compressible fluid flows on unstructured meshes,
\newblock \emph{Journal of Computational Physics}, 228 (2009), 2391--2425.

\bibitem{KluthDespres2010}
G.~Kluth and B.~Despr\'es,
\newblock Discretization of hyperelasticity on unstructured mesh with a
cell-centered Lagrangian scheme,
\newblock \emph{Journal of Computational Physics}, 229 (2010), 9092--9118.

\bibitem{BoscheriLagrangian2022}
W.~Boscheri, S.~Chiocchetti, and I.~Peshkov,
\newblock A cell-centered implicit--explicit Lagrangian scheme for a unified
model of nonlinear continuum mechanics on unstructured meshes,
\newblock \emph{Journal of Computational Physics}, 451 (2022), 110852.

\bibitem{Boscheri2024}
W.~Boscheri, R.~Loub\`ere, J.-P.~Braeunig, and P.-H.~Maire,
\newblock A geometrically and thermodynamically compatible finite volume
scheme for continuum mechanics on unstructured polygonal meshes,
\newblock \emph{Journal of Computational Physics}, 507 (2024), 112957.

\bibitem{Boscheri2026}
W.~Boscheri, M.~Dumbser, R.~Loub\`ere, and P.-H.~Maire,
\newblock A structure-preserving and thermodynamically compatible
cell-centered Lagrangian finite volume scheme for continuum mechanics,
\newblock \emph{SIAM Journal on Numerical Analysis}, 64 (2026), 685--707.

\bibitem{DumbserChiocchettiPeshkov2020}
M.~Dumbser, S.~Chiocchetti, and I.~Peshkov,
\newblock On numerical methods for hyperbolic PDE with curl involutions,
\newblock in \emph{Continuum Mechanics, Applied Mathematics and Scientific
Computing: Godunov's Legacy}, Springer, 2020, pp.~125--134.

\bibitem{Abgrall2025}
R.~Abgrall, M.~Dumbser, and P.-H.~Maire,
\newblock A simple and general framework for the construction of exactly
div--curl--grad compatible discontinuous Galerkin finite element schemes on
unstructured simplex meshes,
\newblock \emph{Journal of Computational Physics}, 541 (2025), 114340.

\bibitem{Dumbser2026}
M.~Dumbser, A.~Thomann, M.~Tavelli, and W.~Boscheri,
\newblock A structure-preserving semi-implicit four-split scheme for
continuum mechanics,
\newblock \emph{Journal of Computational Physics}, 559 (2026), 114889.

\bibitem{DumbserSubcell2014}
M.~Dumbser, O.~Zanotti, R.~Loub\`ere, and S.~Diot,
\newblock A posteriori subcell limiting of the discontinuous Galerkin finite
element method for hyperbolic conservation laws,
\newblock \emph{Journal of Computational Physics}, 278 (2014), 47--75.

\bibitem{Bonet2015}
J.~Bonet, A.~J.~Gil, and R.~Ortigosa,
\newblock A computational framework for polyconvex large strain elasticity,
\newblock \emph{Computer Methods in Applied Mechanics and Engineering},
283 (2015), 1061--1094.

\bibitem{Betsch2018}
P.~Betsch, A.~Janz, and C.~Hesch,
\newblock A mixed variational framework for the design of energy--momentum
schemes inspired by the structure of polyconvex stored energy functions,
\newblock \emph{Computer Methods in Applied Mechanics and Engineering},
335 (2018), 660--696.

\bibitem{Osher1984}
S.~Osher,
\newblock Riemann solvers, the entropy condition, and difference
approximations,
\newblock \emph{SIAM Journal on Numerical Analysis}, 21 (1984), 217--235.

\bibitem{OsherTadmor1988}
S.~Osher and E.~Tadmor,
\newblock On the convergence of difference approximations to scalar
conservation laws,
\newblock \emph{Mathematics of Computation}, 50 (1988), 19--51.

\bibitem{Tadmor1987}
E.~Tadmor,
\newblock The numerical viscosity of entropy stable schemes for systems of
conservation laws. I,
\newblock \emph{Mathematics of Computation}, 49 (1987), 91--103.

\bibitem{Chandrashekar2013}
P.~Chandrashekar,
\newblock Kinetic energy preserving and entropy stable finite volume schemes
for compressible Euler and Navier--Stokes equations,
\newblock \emph{Communications in Computational Physics}, 14 (2013),
1252--1286.

\bibitem{FisherCarpenter2013}
T.~C.~Fisher and M.~H.~Carpenter,
\newblock High-order entropy stable finite difference schemes for nonlinear
conservation laws: finite domains,
\newblock \emph{Journal of Computational Physics}, 252 (2013), 518--557.

\bibitem{CarpenterFisherNielsenFrankel2014}
M.~H.~Carpenter, T.~C.~Fisher, E.~J.~Nielsen, and S.~H.~Frankel,
\newblock Entropy stable spectral collocation schemes for the Navier--Stokes
equations: discontinuous interfaces,
\newblock \emph{SIAM Journal on Scientific Computing}, 36 (2014),
B835--B867.

\bibitem{GassnerWintersKopriva2016}
G.~J.~Gassner, A.~R.~Winters, and D.~A.~Kopriva,
\newblock Split form nodal discontinuous Galerkin schemes with
summation-by-parts property for the compressible Euler equations,
\newblock \emph{Journal of Computational Physics}, 327 (2016), 39--66.

\bibitem{ChenShu2017}
T.~Chen and C.-W.~Shu,
\newblock Entropy stable high order discontinuous Galerkin methods with
suitable quadrature rules for hyperbolic conservation laws,
\newblock \emph{Journal of Computational Physics}, 345 (2017), 427--461.

\bibitem{Chan2018}
J.~Chan,
\newblock On discretely entropy conservative and entropy stable
discontinuous Galerkin methods,
\newblock \emph{Journal of Computational Physics}, 362 (2018), 346--374.

\bibitem{Castro2013}
M.~J.~Castro, U.~S.~Fjordholm, S.~Mishra, and C.~Par\'es,
\newblock Entropy conservative and entropy stable schemes for
nonconservative hyperbolic systems,
\newblock \emph{SIAM Journal on Numerical Analysis}, 51 (2013), 1371--1391.

\bibitem{Renac2019}
F.~Renac,
\newblock Entropy stable DGSEM for nonlinear hyperbolic systems in
nonconservative form with application to two-phase flows,
\newblock \emph{Journal of Computational Physics}, 382 (2019), 1--26.

\bibitem{CockburnShu2001}
B.~Cockburn and C.-W.~Shu,
\newblock Runge--Kutta discontinuous Galerkin methods for
convection-dominated problems,
\newblock \emph{Journal of Scientific Computing}, 16 (2001), 173--261.

\bibitem{Krivodonova2004}
L.~Krivodonova, J.~Xin, J.-F.~Remacle, N.~Chevaugeon, and J.~E.~Flaherty,
\newblock Shock detection and limiting with discontinuous Galerkin methods
for hyperbolic conservation laws,
\newblock \emph{Applied Numerical Mathematics}, 48 (2004), 323--338.

\bibitem{PengSunWu2025}
M.~Peng, Z.~Sun, and K.~Wu,
\newblock OEDG: oscillation-eliminating discontinuous Galerkin method for
hyperbolic conservation laws,
\newblock \emph{Mathematics of Computation}, 94 (2025), 1147--1198.

\bibitem{YangFuOEESDG}
J.~Yang and G.~Fu,
\newblock An entropy-stable oscillation-eliminating DGSEM for the Euler
equations on curvilinear meshes,
\newblock \emph{Journal of Computational Physics}, 568 (2027), 115383,
\newblock \url{https://doi.org/10.1016/j.jcp.2026.115383}.

\bibitem{Hennemann2021}
S.~Hennemann, A.~M.~Rueda-Ram\'irez, F.~J.~Hindenlang, and G.~J.~Gassner,
\newblock A provably entropy stable subcell shock capturing approach for
high order split form DG for the compressible Euler equations,
\newblock \emph{Journal of Computational Physics}, 426 (2021), 109935.

\bibitem{Pazner2021}
W.~Pazner,
\newblock Sparse invariant domain preserving discontinuous Galerkin methods
with subcell convex limiting,
\newblock \emph{Computer Methods in Applied Mechanics and Engineering}, 382
(2021), 113876.

\bibitem{RuedaRamirezPaznerGassner2022}
A.~M.~Rueda-Ram\'irez, W.~Pazner, and G.~J.~Gassner,
\newblock Subcell limiting strategies for discontinuous Galerkin spectral
element methods,
\newblock \emph{Computers \& Fluids}, 247 (2022), 105627.

\bibitem{RuedaRamirezBolmKuzminGassner2024}
A.~M.~Rueda-Ram\'irez, B.~Bolm, D.~Kuzmin, and G.~J.~Gassner,
\newblock Monolithic convex limiting for Legendre--Gauss--Lobatto
discontinuous Galerkin spectral-element methods,
\newblock \emph{Communications on Applied Mathematics and Computation}, 6
(2024), 1860--1898.

\bibitem{LinChan2024}
Y.~Lin and J.~Chan,
\newblock High order entropy stable discontinuous Galerkin spectral element
methods through subcell limiting,
\newblock \emph{Journal of Computational Physics}, 498 (2024), 112677.

\bibitem{Vilar2025}
F.~Vilar,
\newblock Local subcell monolithic DG/FV convex property preserving scheme on
unstructured grids and entropy consideration,
\newblock \emph{Journal of Computational Physics}, 521 (2025), 113535.

\bibitem{Zhao2026}
Z.~Zhao, M.~Zhao, S.~Hao, Q.~Ding, J.~Chen, and Z.~Liu,
\newblock A subcell entropy-constrained limiter for entropy stable
discontinuous Galerkin methods,
\newblock \emph{Journal of Computational Physics}, 563 (2026), 115113.

\bibitem{RuedaRamirezGassner2024}
A.~M.~Rueda-Ram\'irez and G.~J.~Gassner,
\newblock A flux-differencing formula for split-form summation-by-parts
discretizations of non-conservative systems: Applications to subcell
limiting for magneto-hydrodynamics,
\newblock \emph{Journal of Computational Physics}, 496 (2024), 112607.

\bibitem{PerssonPeraire2006}
P.-O.~Persson and J.~Peraire,
\newblock Sub-cell shock capturing for discontinuous Galerkin methods,
\newblock in \emph{44th AIAA Aerospace Sciences Meeting and Exhibit},
AIAA Paper 2006-0112, 2006,
\newblock \url{https://doi.org/10.2514/6.2006-112}.

\bibitem{Barton2010}
P.~T.~Barton, D.~Drikakis, and E.~I.~Romenski,
\newblock An Eulerian finite-volume scheme for large elastoplastic
deformations in solids,
\newblock \emph{International Journal for Numerical Methods in
Engineering}, 81 (2010), 453--484.

\bibitem{BartonEtAl2009}
P.~T.~Barton, D.~Drikakis, E.~Romenski, and V.~A.~Titarev,
\newblock Exact and approximate solutions of Riemann problems in nonlinear
elasticity,
\newblock \emph{Journal of Computational Physics}, 228 (2009), 7046--7068.

\bibitem{DalMaso1995}
G.~Dal Maso, P.~G.~LeFloch, and F.~Murat,
\newblock Definition and weak stability of nonconservative products,
\newblock \emph{Journal de Math\'ematiques Pures et Appliqu\'ees},
74 (1995), 483--548.

\bibitem{Pares2006}
C.~Par\'es,
\newblock Numerical methods for nonconservative hyperbolic systems: a
theoretical framework,
\newblock \emph{SIAM Journal on Numerical Analysis}, 44 (2006), 300--321.

\bibitem{Rusanov1962}
V.~V.~Rusanov,
\newblock The calculation of the interaction of non-stationary shock waves
and obstacles,
\newblock \emph{USSR Computational Mathematics and Mathematical Physics},
1 (1962), 304--320.

\bibitem{TrangensteinColella1991}
J.~A.~Trangenstein and P.~Colella,
\newblock A higher-order Godunov method for modeling finite deformation in
elastic--plastic solids,
\newblock \emph{Communications on Pure and Applied Mathematics}, 44 (1991),
41--100.

\bibitem{MillerColella2001}
G.~H.~Miller and P.~Colella,
\newblock A high-order Eulerian Godunov method for elastic--plastic flow in
solids,
\newblock \emph{Journal of Computational Physics}, 167 (2001), 131--176.

\bibitem{DiGiusto2024}
T.~B.~J.~Di~Giusto, C.~H.~Lee, A.~J.~Gil, J.~Bonet, and M.~Giacomini,
\newblock A first-order hyperbolic arbitrary Lagrangian Eulerian conservation
formulation for nonlinear solid dynamics,
\newblock \emph{International Journal for Numerical Methods in Engineering},
125 (2024), e7467.

\bibitem{HindenlangGassner2020}
F.~J.~Hindenlang and G.~J.~Gassner,
\newblock On the order reduction of entropy stable DGSEM for the
compressible Euler equations,
\newblock in \emph{Spectral and High Order Methods for Partial Differential
Equations ICOSAHOM 2018}, Lecture Notes in Computational Science and
Engineering, vol.~134, Springer, 2020, pp.~21--44.
\newblock \url{https://doi.org/10.1007/978-3-030-39647-3_2}.

\bibitem{Zhang2017}
W.~Zhang, T.~Wang, J.-S.~Bai, P.~Li, Z.-H.~Wan, and D.-J.~Sun,
\newblock The piecewise parabolic method for Riemann problems in nonlinear
elasticity,
\newblock \emph{Scientific Reports}, 7 (2017), 13497.

\bibitem{JiangShu1996}
G.-S.~Jiang and C.-W.~Shu,
\newblock Efficient implementation of weighted ENO schemes,
\newblock \emph{Journal of Computational Physics}, 126 (1996), 202--228.

\bibitem{Borges2008}
R.~Borges, M.~Carmona, B.~Costa, and W.~S.~Don,
\newblock An improved weighted essentially non-oscillatory scheme for
hyperbolic conservation laws,
\newblock \emph{Journal of Computational Physics}, 227 (2008), 3191--3211.

\bibitem{vanLeer1979}
B.~van~Leer,
\newblock Towards the ultimate conservative difference scheme. V. A
second-order sequel to Godunov's method,
\newblock \emph{Journal of Computational Physics}, 32 (1979), 101--136.

\bibitem{Toro2009}
E.~F.~Toro,
\newblock \emph{Riemann Solvers and Numerical Methods for Fluid Dynamics: A
Practical Introduction}, 3rd ed.,
\newblock Springer, Berlin, 2009.

\end{thebibliography}
\end{document}